\documentclass[11pt]{article}

\usepackage{amsmath,amssymb,enumitem,amsthm,stmaryrd}
\usepackage{graphics,pstricks,pst-node,xy}
\xyoption{all}

\numberwithin{equation}{section}
\newtheorem{thm}{Theorem}[section]
\newtheorem{prp}[thm]{Proposition}
\newtheorem{lmm}[thm]{Lemma}
\newtheorem{crl}[thm]{Corollary}

\newtheorem{thmnum}{Theorem}

\theoremstyle{definition}
\newtheorem{dfn}[thm]{Definition}

\def\eset{\emptyset}

\def\BE#1{\begin{equation}\label{#1}}
\def\EE{\end{equation}}
\def\lan{\langle}
\def\ran{\rangle}
\def\lr#1{\lan#1\ran}

\def\ch#1{\check{#1}}
\def\ov#1{\overline{#1}}
\def\ti#1{\tilde{#1}}
\def\wt#1{\widetilde{#1}}
\def\e_ref#1{(\ref{#1})}
\def\smsize#1{\begin{small}#1\end{small}}

\def\sf#1{\textsf{#1}}

\def\Lau#1{\llceil{#1}\rrceil}
\def\coeff#1{\llbracket{#1}\rrbracket}
\def\bigcoeff#1{\big\llbracket{#1}\big\rrbracket}

\def\llbr{\llbracket}
\def\rrbr{\rrbracket}
\def\LR#1{\left\llbr{#1}\right\rrbr}

\def\lra{\longrightarrow}
\def\Lra{\Longrightarrow}

\def\0{\mathbf{0}}
\def\1{\mathbf{1}}
\def\al{\alpha}
\def\be{\beta}
\def\ga{\gamma}
\def\de{\delta}
\def\ep{\epsilon}

\def\la{\lambda}
\def\om{\omega}
\def\si{\sigma}
\def\th{\theta}
\def\vp{\varpi}
\def\vt{\vartheta}

\def\Ga{\Gamma}
\def\La{\Lambda}
\def\Om{\Omega}

\def\De{\Delta}

\def\i{\infty}
\def\hb{\hbar}

\def\cA{\mathcal A}
\def\bD{\mathbf D}

\def\C{\mathbb C}
\def\cC{\mathcal C}
\def\ctC{\wt{\mathcal{C}}}
\def\nc{\mathrm{c}}
\def\ntc{\wt{\mathrm{c}}}
\def\fC{\mathfrak C}
\def\bc{\mathbf c}
\def\d{\mathfrak d}

\def\D{\mathfrak D}

\def\E{\mathbf e}
\def\tE{\textnormal{E}}
\def\F{\mathcal F}

\def\cH{\mathcal H}

\def\I{\mathfrak i}

\def\bI{\mathbb I}

\def\cM{\mathcal M}
\def\M{\mathfrak M}
\def\N{\mathcal N}

\def\O{\mathcal O}
\def\cO{\mathcal O}
\def\P{\mathbb P}
\def\cP{\mathcal P}

\def\Pn{\mathbb P^{n-1}}

\def\Q{\mathbb Q}

\def\bfr{\mathbf r}
\def\bfs{\mathbf s}
\def\bS{\mathbb S}
\def\cS{\mathcal S}
\def\T{\mathbb T}

\def\X{\mathfrak X}

\def\cU{\mathcal U}
\def\V{\mathcal V}

\def\cY{\mathcal Y}
\def\Z{\mathbb Z}
\def\cZ{\mathcal Z}

\def\a{\mathbf a}
\def\b{\mathbf b}

\def\cP{\mathcal{P}}
\def\bM{\mathbf{M}}
\def\nd{\textnormal{d}}

\def\ne{\textnormal{e}}
\def\fs{\mathfrak s}
\def\x{\mathbf x}
\def\y{\mathbf y}

\def\BPS{\textnormal{BPS}}
\def\Edg{\textnormal{Edg}}

\def\ev{\textnormal{ev}}
\def\id{\textnormal{id}}

\def\GW{\textnormal{GW}}

\def\mod{\textnormal{mod~}}

\def\rk{\textnormal{rk}}
\def\bPD{\textnormal{\bf PD}}

\def\Rs#1{\underset{#1}{\mathfrak R}}
\def\Sym{\textnormal{Sym}}

\def\SQ{\textnormal{SQ}}
\def\top{\textnormal{top}}
\def\val{\textnormal{val}}
\def\Ver{\textnormal{Ver}}

\begin{document}

\title{Double and Triple Givental's $J$-functions\\ 
for Stable Quotients Invariants}
\author{Aleksey Zinger}
\date{\today}

\maketitle

\begin{abstract}
We use mirror formulas for the stable quotients analogue of Givental's $J$-function for 
twisted projective invariants obtained in a previous paper to
obtain mirror formulas for the analogues of the double and triple Givental's 
$J$-functions (with descendants at all marked points) in this setting.
We then observe that the genus~0 stable quotients invariants
need not satisfy the divisor, string, or dilaton relations of the Gromov-Witten theory, but 
they do possess the integrality properties of
the genus~0 three-marked Gromov-Witten invariants of Calabi-Yau manifolds. 
We also relate the stable quotients invariants to the BPS counts arising
in Gromov-Witten theory and obtain mirror formulas for certain twisted Hurwitz numbers.
\end{abstract}

\tableofcontents

\section{Introduction}
\label{intro_sec}

\noindent
Gromov-Witten invariants of projective varieties are counts of curves 
that are conjectured (and known in some cases) to possess a rich structure.
In particular, so-called mirror formulas relate these symplectic invariants 
of a nonsingular variety~$X$ to complex-geometric invariants of 
the mirror family of~$X$.
In genus~0, this relation is often described by assembling two-point Gromov-Witten invariants
(but without constraints on the second marked point) into a generating function,
known as Givental's $J$-function, and expressing it in terms of an explicit 
hypergeometric series.
The genus~0 Gromov-Witten invariants of a projective complete intersection~$X$
are equal to the twisted Gromov-Witten invariants of the ambient space
associated to the direct sum of positive line bundles corresponding to~$X$. 
The genus~0 mirror formula in Gromov-Witten theory extends to 
the twisted Gromov-Witten invariants associated with direct sums of line bundles 
over projective spaces; see \cite{Elezi,Gi2,LLY3}. 
By~\cite{GWvsSQ}, the analogue of Givental's $J$-function for
the twisted stable quotients invariants defined in~\cite{MOP09} satisfies
a simpler version of the mirror formula from Gromov-Witten theory.
In this paper, we obtain mirror formulas for the stable quotients analogues of 
the double and triple Givental's $J$-functions for direct sums of line bundles.
We use them to test the stable quotients invariants for the analogues of the standard
properties satisfied by Gromov-Witten invariants.
In the future, we intend to apply the methods of this paper to show that the stable quotients and
Gromov-Witten invariants of projective complete intersections
are related by a simple mirror transform, in all genera, but with at least 
one marked point.

\subsection{Stable quotients}
\label{SQ_subs}

\noindent
The moduli spaces of stable quotients, $\ov{Q}_{g,m}(X,d)$, constructed in~\cite{MOP09}
and generalized in~\cite{CKM},
provide an alternative to
the moduli spaces of stable maps, $\ov\M_{g,m}(X,d)$, for compactifying spaces of
degree~$d$ morphisms from genus~$g$ nonsingular curves with $m$~marked points to
a projective variety~$X$ (with a choice of polarization).
A \sf{stable tuple of quotients} is a tuple
\BE{SQtuple_e}\big(\cC,y_1,\ldots,y_m;S_1\!\subset\!\C^{n_1}\!\otimes\!\cO_{\cC},
\ldots,S_p\!\subset\!\C^{n_p}\!\otimes\!\cO_{\cC}\big),\EE
where $\cC$ is a connected (at worst) nodal curve, 
$y_1,\!\ldots\!,y_m\!\in\!\cC^*$ are distinct smooth points, and
$$S_1\!\subset\!\C^{n_1}\!\otimes\!\cO_{\cC},
\ldots,S_p\!\subset\!\C^{n_p}\!\otimes\!\cO_{\cC}$$
are subsheaves such that  the supports of the torsions
of $\C^{n_1}\!\otimes\!\cO_{\cC}/S_1,\ldots,\C^{n_p}\!\otimes\!\cO_{\cC}/S_p$
are contained in $\cC^*\!-\!\{y_1,\ldots,y_m\}$ and 
the $\Q$-line bundle
$$\om_{\cC}(y_1\!+\!\ldots\!+\!y_m)\otimes
\big(\La^{\top}S_1^*\big)^{\ep}\otimes\ldots\otimes\big(\La^{\top}S_p^*\big)^{\ep}
\lra \cC$$
is ample for all $\ep\!\in\!\Q^+$;
this implies that $2g\!+\!m\!\ge\!2$.\\

\noindent
In this paper, we are concerned only with the case $g\!=\!0$.
For $m,d_1,\ldots,d_p\!\in\!\Z^{\ge0}$ and 
$n_1,\ldots,n_p\!\in\!\Z^+$, the moduli space
\BE{Qgm_e} 
\ov{Q}_{0,m}\big(\P^{n_1-1}\!\times\!\ldots\!\times\!\P^{n_p-1},
(d_1,\ldots,d_p)\big)\EE
parameterizing the stable tuples of quotients as in \e_ref{SQtuple_e}
with $h^1(\cC,\cO_{\cC})\!=\!0$, i.e.~$\cC$ is a rational curve,
$\rk(S_i)\!=\!1$, and $\deg(S_i)\!=\!-d_i$, 
is a nonsingular irreducible Deligne-Mumford stack and
$$\dim \ov{Q}_{0,m}\big(\P^{n_1-1}\!\times\!\ldots\!\times\!\P^{n_p-1},
(d_1,\ldots,d_p)\big)
=(d_1\!+\!1)n_1+\ldots+(d_p\!+\!1)n_p-p-3+m\,;$$
see \cite[Propositions~2.1,2.2]{GWvsSQ}.\\

\noindent
As in the case of stable maps, there are evaluation morphisms,
$$\ev_i\!: \ov{Q}_{0,m}\big(\P^{n_1-1}\!\times\!\ldots\!\times\!\P^{n_p-1},
(d_1,\ldots,d_p)\big)\lra\P^{n_1-1}\!\times\!\ldots\!\times\!\P^{n_p-1}, \qquad i=1,2,\ldots,m,$$
corresponding to each marked point.
There is also a universal curve
$$\pi\!: \cU\lra \ov{Q}_{0,m}(\Pn,d)$$
with $m$ sections $\si_1,\ldots,\si_m$ (given by the marked points) and
$p$ universal rank~1 subsheaves
$$\cS_i\subset \C^{n_i}\!\otimes\!\cO_{\cU}\,.$$
In the case $p\!=\!1$, we will denote $\cS_1$ by $\cS$.
For each $i\!=\!1,2,\ldots,m$, let
$$\psi_i=-\pi_*(\si_i^2)   \in H^2\big( \ov{Q}_{0,m}(\Pn,d)\big)$$
be the first chern class of the universal cotangent line bundle as usual.\\

\noindent
The twisted invariants of projective spaces that we study in this paper are indexed by tuples 
$\a\!=\!(a_1,\ldots,a_l)\!\in\!(\Z^*)^l$ of nonzero integers, 
with $l\!\in\!\Z^{\ge0}$.
For each such tuple~$\a$, let 
\begin{gather*}
|\a|=\sum_{k=1}^l|a_k|,  \qquad
\lr\a=\prod_{a_k>0}\!a_k\bigg/\!\!\prod_{a_k<0}\!a_k\,, \qquad
\a!=\prod_{a_k>0}\!\!a_k!\,, \qquad
\a^{\a}=\prod_{k=1}^l a_k^{|a_k|}\,,\\
\nu_n(\a)=n-|\a|,\qquad
\ell^{\pm}(\a)=\big|\{k\!:\,(\pm1)a_k\!>\!0\}\big|, \qquad
\ell(\a)=\ell^+(\a)-\ell^-(\a).
\end{gather*}
If in addition $n,d\!\in\!\Z^+$, let
\BE{Vstandfn_e}\V_{n;\a}^{(d)}=\bigoplus_{a_k>0}R^0\pi_*\big(\cS^{*a_k}\big)
 \oplus \bigoplus_{a_k<0}R^1\pi_*\big(\cS^{*a_k}\big)
 \lra \ov{Q}_{0,m}(\Pn,d),\EE
where $\pi\!:\cU\!\lra\!\ov{Q}_{0,m}(\Pn,d)$ is the universal curve and $m\!\ge\!2$;
these sheaves are locally free.\\

\noindent
By \cite[Theorem~4.5.2 and Proposition~6.2.3]{CKM},
$$\SQ_{n;\a}^d(c_1,\ldots,c_m)\equiv
\int_{\ov{Q}_{0,m}(\Pn,d)}\!\!\!\E(\V_{n;\a}^{(d)})
\prod_{i=1}^m \ev_i^*x^{c_i}\,, \quad
m\!\ge\!2,~d\!\in\!\Z^+,~c_i\!\in\!\Z^{\ge0},$$
where $x\!\in\!H^2(\Pn)$ is the hyperplane class,
are invariants of the total space $X_{n;\a}$ of the vector bundle
\BE{VBtot_e}  
\bigoplus_{a_k<0}\cO_{\Pn}(a_k)\big|_{X_{n;(a_k)_{a_k>0}}}\lra X_{n;(a_k)_{a_k>0}}\,,\EE
where $X_{n;(a_k)_{a_k>0}}\subset\Pn$ is a nonsingular complete intersection of multi-degree
$(a_k)_{a_k>0}$.
If \hbox{$\nu_n(\a)\!=\!0$}, i.e.~$X_{n;\a}$ is a Calabi-Yau complete intersection, let
$$\SQ_{n;\a}^{c_1,\ldots,c_m}(q)=\sum_{d=0}^{\i}q^d\,\SQ_{n;\a}^d(c_1,\ldots,c_m),$$
with $\GW_{n;\a}^0(\bc)\!\equiv\!\lr\a$ if $|\bc|\!=n\!-\!4\!-\!\ell(\a)\!+\!m$
and 0 otherwise.

\subsection{SQ-invariants and GW-invariants}
\label{SQvsGW_subs}

\noindent 
In Gromov-Witten theory,
there is a natural evaluation morphism $\ev\!:\cU\!\lra\!\Pn$
from the universal curve $\pi\!:\cU\!\lra\!\ov\M_{0,m}(\Pn,d)$.
If $n,d\!\in\!\Z^+$, the sheaf
\BE{VstandfnGW_e}\V_{n;\a}^{(d)}=\bigoplus_{a_k>0}R^0\pi_*\ev^*\cO_{\Pn}(a_k)
 \oplus \bigoplus_{a_k<0}R^1\pi_*\ev^*\cO_{\Pn}(a_k)
 \lra \ov\M_{0,m}(\Pn,d),\EE
is locally free.
It is well-known that 
$$\GW_{n;\a}^d(c_1,\ldots,c_m)\equiv
\int_{\ov\M_{0,m}(\Pn,d)}\!\!\!\E(\V_{n;\a}^{(d)})
\prod_{i=1}^m \ev_i^*x^{c_i}\,, \quad
m,c_i\!\in\!\Z^{\ge0},~d\!\in\!\Z^+,$$
are also invariants of~$X_{n;\a}$.
If $\nu_n(\a)\!=\!0$ and $m\!\ge\!2$, let
$$\GW_{n;\a}^{c_1,\ldots,c_m}(Q)=\sum_{d=0}^{\i}Q^d\,\GW_{n;\a}^d(c_1,\ldots,c_m),$$
with $\GW_{n;\a}^0(\bc)\!\equiv\!\lr\a$ if $|\bc|\!=n\!-\!4\!-\!\ell(\a)\!+\!m$
and 0 otherwise.\\

\noindent
Stable-quotients invariants and Gromov-Witten invariants are equal in many cases, 
but differ for many Calabi-Yau targets,
as we now describe.
Let
\begin{gather}\label{Fdfn_e}
\dot{F}_{n;\a}(w,q)\equiv\sum_{d=0}^{\i}q^d w^{\nu_n(\a)d}\,
\frac{\prod\limits_{a_k>0}\prod\limits_{r=1}^{a_kd}(a_kw\!+\!r)
\prod\limits_{a_k<0}\!\!\prod\limits_{r=0}^{-a_kd-1}\!\!(a_kw\!-\!r)}
{\prod\limits_{r=1}^{d}\big((w\!+\!r)^n\!-\!w^n\big)} \in \Q(w)\big[\big[q\big]\big]\,,\\
\label{YJdfn_e}
\dot{I}_0(q)=\dot{F}_{n;\a}(0,q), \qquad 
J_{n;\a}(q)=\frac{1}{\dot{I}_0(q)}\frac{\partial\dot{F}_{n;\a}}{\partial w}\bigg|_{(0,q)}\,.
\end{gather}
The term $w^n$ above is irrelevant for the purposes of the main formulas of 
Sections~\ref{intro_sec}-\ref{equivthm_sec}.
Its introduction is related to the expansion~\e_ref{cYexp_e}, which is  
used in an essential way in
the proof of~\e_ref{Z2cJpt_e} in Section~\ref{3ptpf_sec}.

\begin{thmnum}\label{GWvsSQ_thm}
Let $l\!\in\!\Z^{\ge0}$, $n\!\in\!\Z^+$, and $\a\!\in\!(\Z^*)^l$ be such that 
$\nu_n(\a)\!=\!0$.
If  $m\!=\!2,3$ and $\bc\!\in\!(\Z^{\ge0})^m$, then
\begin{gather}
\label{GWvsSQ_e1}
d^{3-m}\SQ_{n;\a}^d(\bc)\in\Z \qquad\forall\,d\!\in\!\Z,\\
\label{GWvsSQ_e2}\GW_{n;\a}^{\bc}(Q)=\dot{I}_0(q)^{m-2}\,\SQ_{n;\a}^{\bc}(q)
-\de_{m,2}\lr\a J_{n;\a}(q)\,,
\end{gather}
where $\de_{m,2}$ is the Kronecker delta function and
$Q\!=\!q\ne^{J_{n;\a}(q)}$ is \textsf{the mirror map}.
Furthermore, the genus~0 three-marked stable-quotients invariants of $X_{n;\a}$ satisfy 
the analogue of the dilaton equation of Gromov-Witten theory if and only if
$\ell^-(\a)\!>\!0$ and of the divisor and string relations if and only if
$\ell^-(\a)\!>\!1$.
\end{thmnum}

\noindent
The  relation~\e_ref{GWvsSQ_e2} follows from the explicit mirror
formulas for the stable-quotients analogues of the double and triple Givental's
$J$-functions provided by Theorem~\ref{main_thm} in Section~\ref{mainthm_sec} 
and similar results in Gromov-Witten theory  \cite{Po,g0ci};
see Section~\ref{mainthm_sec} for more details.
By \cite[Theorem 1.2.2 and Corollaries~1.4.1,1.4.2]{CK},
\e_ref{GWvsSQ_e2} holds for $m\!>\!3$ as well.
As the mirror formulas of Theorem~\ref{main_thm} relate SQ-invariants to
the hypergeometric series arising in the B-model of the mirror family
without a change of variables, \e_ref{GWvsSQ_e2} illustrates the principle
that the mirror map relating Gromov-Witten theory to the B-model 
reflects the choice of the curve-counting theory in the A-model
and is not intrinsic to mirror symmetry itself.\\

\noindent
The analogue of \e_ref{GWvsSQ_e2} for GW-invariants is well-known.
By \cite[Proposition~7.3.2]{MS}, the genus~0 GW-invariants of 
a Calabi-Yau manifold with 3+ marked points are integer.
The $m\!=\!2$ case of~\e_ref{GWvsSQ_e1} for GW-invariants is implied by 
the $m\!=\!3$ case and the divisor relation.
The $m\!=\!2,3$ cases of~\e_ref{GWvsSQ_e1} for SQ-invariants follow from
the $m\!=\!2,3$ cases of~\e_ref{GWvsSQ_e1} for GW-invariants and from~\e_ref{GWvsSQ_e2},
since $\dot{I}_0(q),Q(q)\!\in\!\Z[[q]]$;
the integrality of the coefficients of $Q(q)$ whenever $\ell^-(\a)\!=\!0$ is a special case of
\cite[Theorem~1]{KR}.\footnote{The integrality of the coefficients of $\dot{I}_0(q)$
and of $Q(q)$ in the cases $\ell^-(\a)\!>\!0$ is immediate from their definitions.}
Since \e_ref{GWvsSQ_e2} extends to $m\!>\!3$ by~\cite{CK}, so does \e_ref{GWvsSQ_e1}, 
but without the $d^{3-m}$ factors.\\

\noindent
Since $\dot{I}_0(q)\!=\!1$ if and only if $\ell^-(\a)\!=\!0$ and
$J_{n;\a}(q)\!=\!0$ if and only if $\ell^-(\a)\!=\!0,1$,
\e_ref{GWvsSQ_e2} implies that the primary SQ- and GW-invariants of Calabi-Yau
complete intersections are the same if $\ell^-(\a)\!>\!1$;
by Theorem~\ref{main_thm}, this is also the case for the descendant invariants.
Stable-quotients replacements for the divisor, string, or dilaton relations
\cite[Section~26.3]{MirSym} for an arbitrary  Calabi-Yau
complete intersection~$X_{n;\a}$ are provided by \e_ref{SQdiv_e}, 
\e_ref{SQstr_e}, and~\e_ref{SQdil_e}, respectively. 
For the sake of comparison, we list a few  genus~0 SQ- and GW-invariants
of the quintic threefold $X_{5,(5)}\!\subset\!\P^4$ in Table~\ref{GWvsSQ_tbl};
these are obtained from~\e_ref{CYform_e} and~\e_ref{CYformGW_e}, respectively.

\begin{table}
\begin{center}
\begin{tabular}{||c|c|c||}
\hline\hline
$d$ & \smsize{$d\GW_{n;\a}^d(1,1)$}& \smsize{$d\SQ_{n;\a}^d(1,1)$}\\
\hline
1& \smsize{2875}& \smsize{6725}\\
\hline
2& \smsize{4876875}& \smsize{16482625}\\
\hline
3& \smsize{8564575000}& \smsize{44704818125}\\
\hline
4& \smsize{15517926796875}& \smsize{126533974065625}\\
\hline
5& \smsize{28663236110956000}& \smsize{366622331794131725}\\
\hline
6& \smsize{53621944306062201000}& \smsize{1078002594137326617625}\\
\hline
7& \smsize{101216230345800061125625}& \smsize{3201813567943782886368125}\\
\hline 
8& \smsize{192323666400003538944396875}& \smsize{9579628267176528143932815625}\\
\hline
9& \smsize{367299732093982242625847031250}& \smsize{28820906443427523291443507328125}\\
\hline
10& \smsize{704288164978454714776724365580000}& \smsize{87086311562396929291553775833982625}\\ 
\hline\hline
\end{tabular}
\end{center}
\caption{Some genus~0 GW- and SQ-invariants of a quintic threefold $X_{5;(5)}$}
\label{GWvsSQ_tbl}
\end{table}

\subsection{SQ-invariants and BPS states}
\label{SQvsBPS_subs}

\noindent
Using~\e_ref{GWvsSQ_e2}, the genus~0 two- and three-marked SQ-invariants
of a Calabi-Yau complete intersection threefold~$X_{n;\a}$ can be expressed
in terms of the BPS counts of GW-theory.
For example, by the $m\!=\!2$ case of~\e_ref{GWvsSQ_e2},  
\BE{SQvsBPS_e}\SQ_{n;\a}^{1,1}(q)=\lr\a J_{n;\a}(q)-
\sum_{d=1}^{\i}\BPS_{n;\a}^d(1,1)\ln\big(1-q^d\ne^{dJ_{n;\a}(q)}\big),\EE
where $\BPS_{n;\a}^d(1,1)$
are the genus~0 two-marked BPS counts for~$X_{n;\a}$ defined~by
$$\GW_{n;\a}^{1,1}(Q)=-
\sum_{d=1}^{\i}\BPS_{n;\a}^d(1,1)\ln\big(1-Q^d\big).$$
If all genus 0 curves in $X_{n;\a}$ of degree at most~$d$ were smooth and 
had normal bundles $\cO(-1)\!\oplus\!\cO(-1)$,
the number of degree~$d$  genus 0 curves in $X_{n;\a}$ would be~$\BPS_{n;\a}^d(1,1)$;
see \cite[Section~1]{V}.\\

\noindent
Under the regularity assumption of the previous paragraph, the moduli space
$$\ov{Q}_{0,2}^{1,1}(X_{n;\a},d)\equiv\big\{u\!\in\!\ov{Q}_{0,2}(X_{n;\a},d)\!:\,
\ev_1(u)\!\in\!H_1,~\ev_2(u)\!\in\!H_2\big\},$$
where $H_1,H_2\!\subset\!\Pn$ are generic hyperplanes, would split into the topological
components:
\begin{itemize}
\item $\cZ^{1,1}_0(d)$ of stable quotients with torsion of degree~$d$ and thus corresponding
to a constant map to~$H_1\!\cap\!H_2$;
\item $\cZ^{1,1}_C(d)$ of stable quotients with image in a genus~0 curve
$C\!\subset\!X_{n;\a}$ of degree $d_C\!\le\!d$.
\end{itemize}
For $C\!\subset\!X_{n;\a}$ as above, $\cZ^{1,1}_C(d)$ consists of the closed subspaces
$\cZ^{1,1}_{C;r}(d)$, with $r\!\in\!\Z^+$ and $d_Cr\!\le\!d$,
whose generic element has torsion of degree $d\!-\!d_Cr$.
We note~that
$$\dim\cZ^{1,1}_{C;r}(d)=2r\!-\!2+d\!-\!d_Cr+2=d-(d_C\!-\!2)r,$$
which implies that each $\cZ^{1,1}_{C;r}(d)$ is an irreducible component
if $d_C\!>\!1$.
If $d_C\!=\!1$,  $\cZ^{1,1}_{C;r}(d)$ is contained in~$\cZ^{1,1}_{C;d}(d)$,
but still gives rise to a separate contribution
to $\SQ_{n;\a}^d(1,1)$, according to~\e_ref{SQvsBPS_e}.\\

\noindent
The number $\SQ_{n;\a}^d(2,0)$, which arises from the constrained moduli space 
$$\ov{Q}_{0,2}^{2,0}(X_{n;\a},d)=\cZ^{2,0}_0(d)=\cZ^{1,1}_0(d),$$
is $\lr\a$ times the coefficient $\coeff{J_{n;\a}(q)}_d$ of $q^d$ in~$J_{n;\a}(q)$;
see \cite[Theorem~1]{GWvsSQ}.
The contribution of $\cZ^{1,1}_0(d)$ to $\SQ_{n;\a}^d(2,0)$ is the same;
this explains the first term on the right-hand side of~\e_ref{SQvsBPS_e}.
Under the above regularity assumption,
\e_ref{SQvsBPS_e} can be re-written~as
\BE{SQvsBPS_e2}\SQ_{n;\a}^d(1,1)=\lr\a\coeff{J_{n;\a}(q)}_d
+\sum_C\sum_{r=1}^{\i}\frac{1}{r}\coeff{\ne^{J_{n;\a}(q)}}_{d-rd_C}  ,\EE
where the outer sum is taken over all genus~0 curves $C\!\subset\!X_{n;\a}$.
This suggests that the contribution of $\cZ^{1,1}_{C;r}(d)$ to $\SQ_{n;\a}^d(1,1)$
is $\frac{1}{r}\coeff{\ne^{J_{n;\a}(q)}}_{d-rd_C}$.
This contribution depends on the embedding into~$\Pn$, which is as expected,
given the nature of SQ-invariants.\\

\noindent
Since the embedding $C\!\lra\!\Pn$ corresponds to an inclusion
$\O_{\P^1}(-d_C)\!\lra\!\C^n\!\otimes\!\cO_{\P^1}$,
each element of $\cZ^{1,1}_{C;r}(d)$ corresponds to a tuple
\begin{gather*}
\big(\cC,y_1,y_2;S\!\subset\!S'^{\otimes d_c},
 S'\!\subset\!\C^2\!\otimes\!\cO_{\cC} \big),
\qquad\hbox{where}\\ 
\big(\cC,y_1,y_2;S\!\subset\!\C^2\!\otimes\!\cO_{\cC} \big)
\in\ov{Q}_{0,2}(\P^1,d), \qquad
\big(\cC,y_1,y_2;S'\!\subset\!\C^2\!\otimes\!\cO_{\cC} \big)
\in\ov{Q}_{0,2}(\P^1,r).
\end{gather*}
This modular style definition readily extends to arbitrary genus, 
number of marked points, and dimension of projective space.
The arising deformation-obstruction theory can be studied as in \cite[Section~6]{MOP09}.

\subsection{Outline of the paper}
\label{outline_subs}

\noindent
Theorem~\ref{GWvsSQ_thm} is a direct consequence of Theorem~\ref{main_thm}
in Section~\ref{mainthm_sec}, which in turn is 
the non-equivariant specialization of Theorem~\ref{equiv_thm} in 
Section~\ref{equiv_thm}.
We adapt the approaches of \cite{bcov0, bcov0_ci, Po} from Gromov-Witten theory,
outlined in Sections~\ref{pfs_sec} and~\ref{poliC_subs}, 
to show that certain equivariant two-point generating functions,
including the stable-quotients analogue of the double Givental's $J$-function,
satisfy certain good properties which guarantee uniqueness.
The proof that these generating functions satisfy the required properties
follows principles similar to the proof of the analogous statements
in \cite{bcov0, bcov0_ci, Po} and uses the localization theorem of~\cite{ABo};
it is carried out in Sections~\ref{recpf_sec} and~\ref{MPCpf_sec}.\\

\noindent
This approach also implies that certain equivariant three-point generating functions,
including the stable-quotients analogue of the triple Givental's $J$-function,
are determined by three-point primary (without $\psi$-classes) SQ-invariants. 
In the Fano cases, i.e.~$\nu_n(\a)\!>\!0$, enough of these invariants are essentially
trivial for  dimensional reasons to confirm Proposition~\ref{Z3equiv_prp} in these cases;
see Corollary~\ref{MirSym_crl}.
However, there is no dimensional reason for the vanishing of these invariants 
to extend to the Calabi-Yau cases, i.e.~$\nu_n(\a)\!=\!0$;
thus, a different argument is needed in these cases.
We employ the same kind of trick as used in~\cite{GWvsSQ} to confirm mirror symmetry
for the stable quotients analogue of Givental's $J$-function 
 and essentially deduce the Calabi-Yau cases from the Fano cases.
Specifically, we show that the equivariant three-point mirror formula
of Proposition~\ref{Z3equiv_prp} is equivalent to the closed formula 
for twisted three-point Hurwitz numbers of Proposition~\ref{equiv0_prp},
whenever $|\a|\!\le\!n$.
In Section~\ref{equiv0pf_sec}, we show that the validity of the latter does not depend~$n$;
since it holds whenever $|\a|\!<\!n$, it follows that it holds for all~$\a$, and so
the equivariant three-point mirror formula of Proposition~\ref{Z3equiv_prp} holds whenever   
$|\a|\!\le\!n$.
Along with~\cite{g0ci}, Proposition~\ref{Z3equiv_prp} finally leads 
to the mirror formula for
the stable-quotients analogue of the triple Givental's $J$-function
in Theorem~\ref{equiv_thm}; see Section~\ref{3ptpf_sec}.\\

\noindent
The closed formulas for twisted Hurwitz numbers of Propositions~\ref{equiv0_prp}
and~\ref{equiv0_prp2} 
are among the key ingredients in computing the genus~1 twisted stable quotients 
invariants with 1~marked point.
At the same time, this paper and~\cite{g0ci} provide an approach to comparing
the (equivariant) genus~$g$ $m$-fold Givental's $J$-functions,
\BE{Zmultidfn_e}\sum_{d=0}^{\i}q^d\{\ev_1\times\ldots\!\times\!\ev_m\}_*  \!\!\left[\frac{e(\dot\V_{n;\a}^{(d)})}{(\hb_1\!-\!\psi_1)\ldots(\hb_m\!-\!\psi_m)}\right]
\in H^*(\Pn)[\hb_1^{-1},\ldots,\hb_m^{-1}]\big[\big[q\big]\big]\EE
in the SQ- and GW-theories  
for all $g\ge\!0$ and $m\!\ge\!1$ with $2g\!+\!m\!\ge\!2$.
By Proposition~\ref{uniqueness_prp} and Lemmas~\ref{recgen_lmm} and~\ref{polgen_lmm},
in the genus~0 case the restrictions of these generating functions to insertions at only 
one marked point agree whenever $\nu_{\a}\!>\!1$.
In all cases, the approach of~\cite{g0ci} can be adapted to show that 
\e_ref{Zmultidfn_e} is a sum over (at least) trivalent $m$-marked graphs
with coefficients that involve equivariant $m'$-pointed Hurwitz numbers with $m'\!\le\!m$,
which are conversely completely determined by the stable-quotients analogue of
the $m'$-pointed Givental's $J$-function with insertions at only one marked point
 through relations that do not involve~$n$.
Since these relations hold whenever $\nu_n(\a)\!>\!0$, they hold for all~$\a$.
We intend to clarify these points in a future paper.\\

\noindent
The Gromov-Witten analogues of Theorem~\ref{main_thm} and its equivariant version, Theorem~\ref{equiv_thm} in Section~\ref{equivthm_sec}, 
extend to the so-called \sf{concavex vector bundles} over products of 
projective spaces, i.e.~vector bundles of the form
$$\bigoplus_{k=1}^l\cO_{\P^{n_1-1}\times\ldots\times\P^{n_p-1}}(a_{k;1},\ldots,a_{k;p})
\lra \P^{n_1-1}\!\times\!\ldots\!\times\!\P^{n_p-1},$$
where for each given $k\!=\!1,2,\ldots,l$ either $a_{k;1},\ldots,a_{k;p}\!\in\!\Z^{\ge0}$,
with $a_{k;i}\!\neq\!0$ for some~$i$, 
or $a_{k;1},\ldots,a_{k;p}\!\in\!\Z^-$.
The stable quotients analogue of these bundles are the sheaves
\BE{gensheaf_e}
\bigoplus_{k=1}^l\cS_1^{*a_{k;1}}\!\otimes\!\ldots\!\otimes\cS_p^{*a_{k;p}}
\lra\cU
\lra \ov{Q}_{0,2}\big(\P^{n_1-1}\!\times\!\ldots\!\times\!\P^{n_p-1},(d_1,\ldots,d_p)
\big)\EE
with the same condition on~$a_{k;i}$,
where $\cS_i\!\lra\!\cU$ is the universal subsheaf corresponding to the $i$-th factor.
We will comment on the necessary modifications at each step of the proof.


\subsection{Acknowledgments}

\noindent
The author would like to thank 
Y.~Cooper for many enlightening discussions concerning stable quotients invariants,
R.~Pandharipande for bringing moduli spaces of stable quotients to his attention,
A.~Popa for useful comments on a previous version of this paper,
and I.~Ciocan-Fontanine for explaining~\cite{CK}.
He would also like to thank the School of Mathematics at IAS
for its hospitality during the period when the results in this paper were obtained
and the paper itself was completed.
The author's research was partially supported by NSF grants DMS-0635607 and DMS-0846978
and the IAS Fund for Math.

\section{Main theorem}
\label{mainthm_sec}

\noindent
We arrange stable quotients invariants with two and three marked points
into generating functions
in Section~\ref{GivenJ_subs} and give explicit closed formulas for them 
in Section~\ref{MirSym_subs}. 
In Section~\ref{mainapplic_subs}, we use these formulas to relate SQ and GW-invariants,
with descendants, and obtain replacements for the divisor, string, and dilaton 
relations for SQ-invariants.

\subsection{Givental's $J$-functions}
\label{GivenJ_subs}

\noindent
For computational purposes, it is  convenient to define variations of 
the bundle~\e_ref{Vstandfn_e} by
\BE{Vprdfn_e}\begin{split}
 \dot\V_{n;\a}^{(d)}&=\bigoplus_{a_k>0}R^0\pi_*\big(\cS^{*a_k}(-\si_1)\big)
 \oplus \bigoplus_{a_k<0}R^1\pi_*\big(\cS^{*a_k}(-\si_1)\big)
 \lra \ov{Q}_{0,m}(\Pn,d), \\
\ddot\V_{n;\a}^{(d)}&=\bigoplus_{a_k>0}R^0\pi_*\big(\cS^{*a_k}(-\si_2)\big)
 \oplus \bigoplus_{a_k<0}R^1\pi_*\big(\cS^{*a_k}(-\si_2)\big)
 \lra \ov{Q}_{0,m}(\Pn,d),
\end{split}\EE
where $n,d\!\in\!\Z^+$, $m\!\ge\!2$, and
$\pi\!:\cU\!\lra\!\ov{Q}_{0,m}(\Pn,d)$ is the universal curve;
these sheaves are also locally free.
Whenever $\nu_n(\a)\!\ge\!0$,
\cite[Theorem~1]{GWvsSQ} provides an explicit closed formula for the stable quotients
analogue of Givental's $J$-function, the power series
\BE{Zdfn_e}
\dot{Z}_{n;\a}(x,\hb,q) \equiv
1+\sum_{d=1}^{\i}q^d\ev_{1*}\!\!\left[\frac{e(\dot\V_{n;\a}^{(d)})}{\hb\!-\!\psi_1}\right]
\in H^*(\Pn)[\hb^{-1}]\big[\big[q\big]\big],\EE
where $\ev_1:\ov{Q}_{0,2}(\Pn,d)\lra\Pn$ is as before
and $x\!\in\!H^2(\Pn)$ is the hyperplane class.
In this paper, we obtain a closed formula for the power series
\BE{otherZdfn_e}
\ddot{Z}_{n;\a}(x,\hb,q) \equiv
1+\sum_{d=1}^{\i}q^d\ev_{1*}\!\!\left[\frac{e(\ddot\V_{n;\a}^{(d)})}{\hb\!-\!\psi_1}\right]
\in H^*(\Pn)[\hb^{-1}]\big[\big[q\big]\big];\EE
see \e_ref{Givental_e}.\\

\noindent
We also give explicit formulas for the stable quotients
analogues of the double and triple Givental's $J$-functions, the power series 
\begin{gather}\label{Z2dfn_e}
\dot{Z}_{n;\a}^*(x_1,x_2,\hb_1,\hb_2,q) \equiv
\sum_{d=1}^{\i}q^d\big\{\ev_1\!\times\!\ev_2\}_*\!\!\left[\frac{e(\dot\V_{n;\a}^{(d)})}
{(\hb_1\!-\!\psi_1)(\hb_2\!-\!\psi_2)}\right]
\in H^*(\Pn)[\hb_1^{-1},\hb_2^{-1}]\big[\big[q\big]\big],\\
\label{Z3Jdfn_e}
\dot{Z}_{n;\a}^*(x_1,x_2,x_3,\hb_1,\hb_2,\hb_3,q) \equiv
\sum_{d=1}^{\i}q^d\big\{\ev_1\!\times\!\ev_2\!\times\!\ev_3\}_*\!
\!\left[\frac{e(\dot\V_{n;\a}^{(d)})}
{(\hb_1\!-\!\psi_1)(\hb_2\!-\!\psi_2)(\hb_3\!-\!\psi_3)}\right],
\end{gather}
where  $x_i\!=\!\pi_i^*x$ is the pull-back of the hyperplane class in~$\Pn$ by 
the $i$-th projection map and
\BE{evaltot_e}\begin{split}
\ev_1\!\times\!\ev_2\!:\ov{Q}_{0,2}(\Pn,d)&\lra\Pn\!\times\!\Pn, \\
\ev_1\!\times\!\ev_2\!\times\!\ev_3\!:\ov{Q}_{0,3}(\Pn,d)&\lra
\Pn\!\times\!\Pn\!\times\!\Pn
\end{split}\EE
are the total evaluation maps.
Let
\BE{Z23full_e}\begin{split}
\dot{Z}_{n;\a}(x_1,x_2,\hb_1,\hb_2,q)
&=\frac{1}{\hb_1\!+\!\hb_2}
\sum_{\begin{subarray}{c}s_1,s_2\ge0\\ s_1+s_2=n-1\end{subarray}}
\hspace{-.2in} x_1^{s_1}x_2^{s_2}
+\dot{Z}_{n;\a}^*(x_1,x_2,\hb_1,\hb_2,q)\,,\\
\dot{Z}_{n;\a}(x_1,x_2,x_3,\hb_1,\hb_2,\hb_3,q)
&=\frac{1}{\hb_1\hb_2\hb_3}
\sum_{\begin{subarray}{c}s_1,s_2,s_3\ge0\\ s_1,s_2,s_3\le n-1\\
s_1+s_2+s_3=2n-2 \end{subarray}}
\hspace{-.3in} x_1^{s_1}x_2^{s_2}x_3^{s_3}
+\dot{Z}_{n;\a}^*(x_1,x_2,x_3,\hb_1,\hb_2,\hb_3,q).
\end{split}\EE
For each $s\!\in\!\Z^{\ge0}$, define
\BE{Zsdfn_e}\begin{split}
\dot{Z}_{n;\a}^{(s)}(x,\hb,q) &\equiv x^s+
\sum_{d=1}^{\i}\!q^d\ev_{1*}\!\!\left[\frac{e(\dot\V_{n;\a}^{(d)})\ev_2^*x^s}
{\hb\!-\!\psi_1}\right]\in H^*(\Pn)[\hb^{-1}]\big[\big[q\big]\big],\\
\ddot{Z}_{n;\a}^{(s)}(x,\hb,q) &\equiv x^s+
\sum_{d=1}^{\i}\!q^d\ev_{1*}\!\!\left[\frac{e(\ddot\V_{n;\a}^{(d)})\ev_2^*x^s}
{\hb\!-\!\psi_1}\right]\in H^*(\Pn)[\hb^{-1}]\big[\big[q\big]\big],
\end{split}\EE
where $\ev_1,\ev_2\!:\ov{Q}_{0,2}(\Pn,d)\!\lra\!\Pn$.
Thus, $\dot{Z}_{n;\a}^{(0)}\!=\!\dot{Z}_{n;\a}$, $\ddot{Z}_{n;\a}^{(0)}\!=\!\ddot{Z}_{n;\a}$, 
and
\begin{gather*}
x^{\ell^+_-(\a)}\dot{Z}_{n;\a}^{(\ell^-_+(\a)+s)}(x,\hb,q)
=x^{\ell^-_+(\a)}\ddot{Z}_{n;\a}^{(\ell^+_-(\a)+s)}(x,\hb,q) \qquad\forall\,s\!\ge\!0,\\
\hbox{where}\qquad\ell^{\pm}_{\mp}(\a)=\max\big(\pm\ell(\a),0\big).
\end{gather*}
By Theorem~\ref{main_thm} below, $\dot{Z}_{n;\a}^{(s)}$, $\ddot{Z}_{n;\a}^{(s)}$, 
and the stable quotients  analogues of the double and triple Givental's $J$-functions,
\e_ref{Z2dfn_e} and~\e_ref{Z3Jdfn_e},  are explicit transforms of
Givental's $J$-function~$\dot{Z}_{n;\a}$ and its ``reflection"~$\ddot{Z}_{n;\a}$;
this transform depends only on~$\a$ (and $s$ in the first two cases).

\subsection{Mirror symmetry}
\label{MirSym_subs}

\noindent
Givental's $J$-function~$\dot{Z}_{n;\a}$  and its ``reflection"~$\ddot{Z}_{n;\a}$ 
in Gromov-Witten and stable-quotients theories are described by 
the hypergeometric series~\e_ref{Fdfn_e} and 
\BE{Fdfn_e2}\ddot{F}_{n;\a}(w,q)\equiv\sum_{d=0}^{\i}q^d w^{\nu_n(\a)d}\,
\frac{\prod\limits_{a_k>0}\prod\limits_{r=0}^{a_kd-1}(a_kw\!+\!r)
\prod\limits_{a_k<0}\!\!\prod\limits_{r=1}^{-a_kd}\!\!(a_kw\!-\!r)}
{\prod\limits_{r=1}^{d}\big((w\!+\!r)^n-w^n\big)} \in \Q(w)\big[\big[q\big]\big].\EE
These are power series in $q$ with constant term 1 whose coefficients are rational functions 
in $w$ which are regular at~$w=0$.
We denote the subgroup of all such power series by~$\cP$
and define
\BE{DMDfn_e}
\begin{aligned}
&\bD\!:\Q(w)\big[\big[q\big]\big]\lra \Q(w)\big[\big[q\big]\big], 
&\quad& \bM:\cP\lra\cP  \qquad\hbox{by}\\
&\bD H(w,q)\equiv \left\{1+\frac{q}{w}\frac{\nd}{\nd q}\right\}H(w,q),
&\quad&
\bM H(w,q)\equiv\bD\bigg(\frac{H(w,q)}{H(0,q)}\bigg)\,.
\end{aligned}\EE
If $\nu_n(\a)\!=\!0$ and $s\in\Z^{\ge0}$, let
\BE{Ipdfn_e} \dot{I}_s(q)\equiv \bM^s\dot{F}_{n;\a}(0,q), \qquad 
\ddot{I}_s(q)\equiv \bM^s\ddot{F}_{n;\a}(0,q).\EE
For example, $\dot{I}_s(q)\!=\!1$ if $s\!<\!\ell^-(\a)$, 
$\ddot{I}_s(q)\!=\!1$ if $s\!<\!\ell^+(\a)$,
$$\dot{I}_{\ell^-(\a)}(q)=\ddot{I}_{\ell^+(\a)}(q)
=\sum_{d=0}^{\i}q^d\frac{\prod\limits_{a_k>0}(a_kd)!\,
\prod\limits_{a_k<0}\big((-1)^{a_kd}(-a_kd)!\big)}{(d!)^n}
\qquad\text{if}\quad \nu_n(\a)\!=\!0,$$
and more generally $\dot{I}_{s+\ell_+^-(\a)}(q)\!=\!\ddot{I}_{s+\ell_-^+(\a)}(q)$ 
for all $s\!\ge\!0$.
If $\nu_n(\a)\!>\!0$, we set $\dot{I}_s(q),\ddot{I}_s(q)\!=\!1$.\\

\noindent
It is also convenient to introduce
\BE{hatFdfn_e}
F_{n;\a}(w,q)\equiv\sum_{d=0}^{\i}q^d w^{\nu_n(\a)d}\,
\frac{\prod\limits_{a_k>0}\prod\limits_{r=1}^{a_kd}(a_kw\!+\!r)
\prod\limits_{a_k<0}\!\!\prod\limits_{r=1}^{-a_kd}\!\!(a_kw\!-\!r)}
{\prod\limits_{r=1}^{d}(w\!+\!r)^n} \in \Q(w)\big[\big[q\big]\big]\EE
and the associated power series $I_s(q)=\bM^s{F}_{n;\a}(0,q)$ in the $\nu_n(\a)\!=\!0$ case.  
In the case $0\!<\!\nu_n(\a)\!<\!n$, we define $\nc^{(d)}_{s,s'}\in\!\Q$
with $d,s,s'\!\ge\!0$ by 
\BE{littlec_e}\begin{split}
\sum_{d=0}^{\i}\sum_{s'=0}^{\i}\nc^{(d)}_{s,s'}w^{s'}q^d
&\equiv w^s\bD^sF_{n;\a}\big(w,q/w^{\nu_n(\a)}\big)\\
&=w^s\bD^{s+\ell^-(\a)}\dot{F}_{n;\a}\big(w,q/w^{\nu_n(\a)}\big)
=w^s\bD^{s+\ell^+(\a)}\ddot{F}_{n;\a}\big(w,q/w^{\nu_n(\a)}\big).
\end{split}\EE
Since $\nc^{(0)}_{s,s'}=\de_{s,s'}$, the relations
\BE{littletic_e}
\sum_{\begin{subarray}{c}d_1,d_2\ge0\\ d_1+d_2=d\end{subarray}}
\sum_{t=0}^{s-\nu_n(\a)d_1}\ntc^{(d_1)}_{s,t}\nc^{(d_2)}_{t,s'}
=\de_{d,0}\de_{s,s'}
\quad \forall\,d,s'\!\in\!\Z^{\ge0},\, s'\!\le\!s\!-\!\nu_n(\a)d,\EE
inductively define $\ntc^{(d)}_{s,s'}\!\in\!\Q$
in terms of the numbers $\ntc^{(d_1)}_{s,t}$ with $d_1\!<\!d$.
For example, $\ntc^{(0)}_{s,s'}\!=\!\de_{s,s'}$ and
$$\sum_{s'=0}^{s-\nu_n(\a)}\ntc^{(1)}_{s,s'}w^{s'}+
\prod_{k=1}^l\!\!a_k\,\frac{\prod\limits_{a_k>0}\prod\limits_{r=1}^{a_k-1}(a_kw\!+\!r)
\prod\limits_{a_k<0}\!\!\prod\limits_{r=1}^{-a_k-1}\!\!(a_kw\!-\!r)}
{(w+1)^{n-\ell^+(\a)-\ell^-(\a)-s}}
\in w^{s-\nu_n(\a)+1}\Q[[w]].$$
If $s'\!<\!0$ or $\nu_n(\a)\!=\!0,n$, we set 
$\ntc^{(d)}_{s,s'}=\de_{d,0}\de_{s,s'}$.\\

\noindent
For $s_1,s_2,s_3,d\!\in\!\Z^{\ge0}$ with $s_1,s_2,s_3\!\le\!n\!-\!1$, let
\BE{ntc3dfn_e}\ntc^{(d)}_{s_1,s_2,s_3}=
\begin{cases}\bigcoeff{\big((1\!-\!\a^{\a}q)\dot\bI_{s_1}^c(q)
\ddot\bI_{s_2}^c(q)\ddot\bI_{s_3}^c(q)\big)^{-1}}_d,
&\hbox{if}~\nu_n(\a)\!=\!0;\\
\sum\limits_{\begin{subarray}{c}d_0,d_1,d_2,d_3\ge0\\ d_0+d_1+d_2+d_3=d\end{subarray}}
\!\!\!\! (\a^{\a})^{d_0}
\prod\limits_{t=1}^3
\ntc^{(d_t)}_{\hat{s}_t-\ell_t(\a),\hat{s}_t-\nu_n(\a)d_t-\ell_t(\a)}\,,&
\hbox{if}~\nu_n(\a)\!>\!0;
\end{cases}\EE
where
\BE{bIdfn_e}\dot\bI_s^c=\prod_{t=s+1}^{n-\ell^+(\a)}\!\!\!\dot{I}_t\,,\quad
\ddot\bI_s^c=\prod_{t=s+1}^{n-\ell^-(\a)}\!\!\!\ddot{I}_t\,,\quad
\hat{s}_t=n\!-\!1\!-\!s_t\,, \quad
\ell_t(\a)=\begin{cases}\ell^+(\a),&\hbox{if}~t\!=\!1;\\
\ell^-(\a),&\hbox{if}~t\!=\!2,3;\end{cases}\EE
and $\coeff{f(q)}_d$ is the coefficient of $q^d$ of $f(q)\!\in\!\Q[[q]]$.
In particular, $\dot\bI_s^c\!=\!1$ if $s\!\ge\!n\!-\!\ell^+(\a)$ and 
$\ddot\bI_s^c\!=\!1$ if $s\!\ge\!n\!-\!\ell^-(\a)$.
Since $I_t\!=\!\dot{I}_{t+\ell^-(\a)}=\ddot{I}_{t+\ell^+(\a)}$, we find~that 
$$\dot\bI_s^c(q)=(1-\a^{\a}q)^{-1}~~\hbox{if}~s\!<\!\ell^-(\a), \qquad
\ddot\bI_s^c(q)=(1-\a^{\a}q)^{-1}~~\hbox{if}~s\!<\!\ell^+(\a);$$
see \cite[Proposition~4.4]{g0ci}.
This implies that 
\BE{Cntcred_e}\sum_{d=0}^{\i}\ntc^{(d)}_{s_1,s_2,s_3}q^d=1
\quad\hbox{if}~\nu_n(\a)=0,~s_1\!+\!s_2\!+\!s_3=2n\!-\!2,~
\min(s_1,s_2,s_3)<\ell^-(\a).\EE
We use this observation in Section~\ref{mainapplic_subs}.
Since $\ntc_{s,s'}^{(0)}\!=\!\de_{s,s'}$, $\ntc_{s_1,s_2,s_3}^{(0)}\!=\!1$.\\

\noindent
Finally, for each $s\!\in\!\Z^+$, we define 
$\D^s\dot{Z}_{n;\a}(x,\hb,q),\D^s\ddot{Z}_{n;\a}(x,\hb,q)\in H^*(\Pn)[\hb][[q]]$ inductively~by
\BE{DsZdfn_e}\begin{aligned}
\D^0\dot{Z}_{n;\a}(x,\hb,q)&=\dot{Z}_{n;\a}(x,\hb,q),&
\D^s\dot{Z}_{n;\a}(x,\hb,q)&=
\frac{1}{\dot{I}_s(q)}
\left\{x+\hb\, q\frac{\nd}{\nd q}\right\}\D^{s-1}\dot{Z}_{n;\a}(x,\hb,q),\\
\D^0\ddot{Z}_{n;\a}(x,\hb,q)&=\ddot{Z}_{n;\a}(x,\hb,q),&
\D^s\ddot{Z}_{n;\a}(x,\hb,q)&=
\frac{1}{\ddot{I}_s(q)}
\left\{x+\hb\, q\frac{\nd}{\nd q}\right\}\D^{s-1}\ddot{Z}_{n;\a}(x,\hb,q).
\end{aligned}\EE

\begin{thmnum}\label{main_thm}
If $l\!\in\!\Z^{\ge0}$, $n\!\in\!\Z^+$, and $\a\!\in\!(\Z^*)^l$, 
the stable quotients analogue of the double Givental's $J$-function satisfies
\BE{Z2pt_e}
\dot{Z}_{n;\a}(x_1,x_2,\hb_1,\hb_2,q)= \frac{1}{\hb_1+\hb_2}
\sum_{\begin{subarray}{c}s_1,s_2\ge0\\ s_1+s_2=n-1 \end{subarray}}
\!\!\!\!\!\!
\dot{Z}_{n;\a}^{(s_1)}(x_1,\hb_1,q)\ddot{Z}_{n;\a}^{(s_2)}(x_2,\hb_2,q)\,.\EE
If in addition $\nu_{\a}\!\ge\!0$,
\begin{gather}
\label{Z2Jpt_e}
\begin{split}
&\dot{Z}_{n;\a}(x_1,x_2,x_3,\hb_1,\hb_2,\hb_3,q)\\
&\qquad\quad = \frac{1}{\hb_1\hb_2\hb_3}
\sum_{\begin{subarray}{c}d,s_1,s_2,s_3\ge0\\
s_1,s_2,s_3\le n-1 \\
s_1+s_2+s_3+\nu_n(\a)d=2n-2\\
\end{subarray}}
\hspace{-.35in}
\ntc^{(d)}_{s_1,s_2,s_3}q^d
\dot{Z}_{n;\a}^{(s_1)}(x_1,\hb_1,q)
\prod_{t=2}^3\ddot{Z}_{n;\a}^{(s_t)}(x_t,\hb_t,q)\,,
\end{split}\\
\label{Zs_e}
\ch{Z}_{n;\a}^{(s)}(x,\hb,q)=\sum_{d=0}^{\i}\sum_{s'=0}^{s-\nu_n(\a)d}
\ntc_{s-\ell^*(\a),s'-\ell^*(\a)}^{(d)}q^d\,\hb^{s-\nu_n(\a)d-s'}\D^{s'}\ch{Z}_{n;\a}(x,\hb,q),
\end{gather} 
where $(\ch{Z},\ell^*)=(\dot{Z},\ell^-),(\ddot{Z},\ell^+)$.
\end{thmnum}

\subsection{Some computations}
\label{mainapplic_subs}

\noindent
The first identity in Theorem~\ref{main_thm} also holds for the Gromov-Witten analogues
of the generating series~$\dot{Z}_{n;\a}^*$, $\dot{Z}_{n;\a}^{(s)}$, and $\ddot{Z}_{n;\a}^{(s)}$; 
see \cite[Theorem~1.2]{Po} for the general (toric) case.
If $\nu_n(\a)\!\ge\!2\!-\!\ell^-(\a)$, the analogues of \e_ref{Z2Jpt_e},
\e_ref{Zs_e}, \e_ref{Givental_e}, and~\e_ref{ZvsY_e} hold in Gromov-Witten theory as well.
Thus, in this case 
the double Givental's $J$-functions in Gromov-Witten and stable quotients theories agree.
If $\nu_n(\a)\!=\!1$ and $\ell^-(\a)\!=\!0$, the analogue of~\e_ref{Zs_e}
in Gromov-Witten theory  holds  with 
$\left\{x+\hb\, q\frac{\nd}{\nd q}\right\}$
replaced by $\left\{\a! q+x+\hb\, q\frac{\nd}{\nd q}\right\}$
in~\e_ref{DsZdfn_e}.
Finally, if $\nu_n(\a)\!=\!0$ and $\ell^-(\a)\!\le\!1$,
the analogue of~\e_ref{Zs_e} in Gromov-Witten theory  holds  with
\begin{equation*}\begin{split}
\D^s\dot{Z}_{n;\a}(x,\hb,Q)&=
\frac{\dot{I}_1(q)}{\dot{I}_s(q)}
\left\{x+\hb\, Q\frac{\nd}{\nd Q}\right\}\D^{s-1}\dot{Z}_{n;\a}(x,\hb,Q)\,~~~\forall~s\!\in\!\Z^+,\\
\D^s\ddot{Z}_{n;\a}(x,\hb,Q)&=
\frac{\dot{I}_1(q)}{\ddot{I}_s(q)}
\left\{x+\hb\, Q\frac{\nd}{\nd Q}\right\}\D^{s-1}\ddot{Z}_{n;\a}(x,\hb,Q)\,~~~\forall~s\!\in\!\Z^+,
\end{split}\end{equation*}
where $Q=q\ne^{J_{n;\a}(q)}$.
The same comparison applies to the equivariant version of Theorem~\ref{main_thm},
Theorem~\ref{equiv_thm} in Section~\ref{equivthm_sec}, and its Gromov-Witten analogue;
see \cite[Theorem~4.1]{Po} for the general toric case.
Thus, just as is the case for the standard Givental's $J$-function,
the mirror formulas for 
the double Givental's $J$-function in the stable quotients theory 
are simpler versions of the mirror formulas for
the double Givental's $J$-function in the Gromov-Witten theory.
Furthermore, just as in Gromov-Witten theory, the generating functions 
$\dot{Z}_{n;\a}^{(s)}$, $\ddot{Z}_{n;\a}^{(s)}$, and 
$\dot{Z}_{n;\a}^*$ above do not change when the tuple $(a_1,\ldots,a_l)$
is replaced by~$(a_1,\ldots,a_l,1)$;
this is consistent with \cite[Proposition~4.6.1]{CKM}.\\

\noindent
Comparing Theorem~\ref{main_thm} and \cite[(1.7)]{GWvsSQ}
with \cite[Theorem~1.2]{Po} and the $m\!=\!3$ case of \cite[Theorem~A]{g0ci}, 
we find that 
\BE{GWvsSQ23_e}\begin{split}
&\dot{Z}_{n;\a}^{\GW}(x_1,\ldots,x_3,\hb_1,\ldots,\hb_3,Q)\\
&\hspace{1in}=
\dot{I}_0(q)^{m-2}
\ne^{-J_{n;\a}(q)\left(\frac{x_1}{\hb_1}+\ldots+\frac{x_m}{\hb_m}\right)}
\dot{Z}_{n;\a}(x_1,\ldots,x_m,\hb_1,\ldots,\hb_m,q)
\end{split}\EE
with $Q\!=\!q\ne^{J_{n;\a}(q)}$ as before and $m\!=\!2,3$;
we intend to extend this comparison to $m\!>\!3$ in a future paper.
The same relations hold between the generating series $Z_{n:\a}$ described below.
For $m\!=\!2,3$, $b_1,b_2,b_3,c_1,c_2,c_3\!\ge\!0$, let
\begin{equation*}\begin{split}
\SQ_{n;\a}^0(\tau_{b_1}c_1,\tau_{b_2}c_2,\tau_{b_3}c_3)&=
\begin{cases}
\lr\a,&\hbox{if}~b_1,b_2,b_3\!=\!0,~ 
c_1\!+\!c_2\!\!+\!c_3\!=\!n\!-\!1\!-\!\ell(\a);\\
0,&\hbox{otherwise};
\end{cases}\\
\SQ_{n;\a}^0(\tau_{b_1}c_1,\tau_{b_2}c_2)&=
\begin{cases}
\lr\a,&\hbox{if}~b_1,b_2\!=\!0,\,c_1\!+\!c_2\!=\!n\!-\!2\!-\!\ell(\a);\\
\frac{\lr\a}{2},&\hbox{if}~\{b_1,b_2\}=\{0,-1\},c_1\!+\!c_2\!=\!n\!-\!1\!-\!\ell(\a);\\
0,&\hbox{otherwise};
\end{cases}\\
\SQ_{n;\a}^d(\tau_{b_1}c_1,\ldots,\tau_{b_m}c_m)&=
\int_{\ov{Q}_{0,m}(\Pn,d)}\!\!\!\E(\V_{n;\a}^{(d)})
\prod_{i=1}^m\big(\psi_i^{b_i}\ev_i^*x^{c_i}\big) \qquad\forall~d\in\Z^+,\\
\SQ_{n;\a}^{c_1,\ldots,c_m}(q)_{b_1,\ldots,b_m}
&=\sum_{d=0}^{\i}q^d\,\SQ_{n;\a}^d(\tau_{b_1}c_1,\ldots,\tau_{b_m}c_m).
\end{split}\end{equation*}
Since GW-invariants satisfy the divisor, string, and dilaton relations, 
\e_ref{GWvsSQ23_e} leads to modified versions of these relations for SQ-invariants:
\begin{gather}\label{SQdiv_e}
\begin{split}
\dot{I}_0(q)\dot{I}_1(q)\SQ_{n;\a}^{c_1,c_2,1}(q)_{b_1,b_2,0}=
q\frac{\nd}{\nd q}\SQ_{n;\a}^{c_1,c_2}(q)_{b_1,b_2}
&+\SQ_{n;\a}^{c_1+1,c_2}(q)_{b_1-1,b_2}\\
&+\SQ_{n;\a}^{c_1,c_2+1}(q)_{b_1,b_2-1}\,,
\end{split}\\
\label{SQstr_e}
\dot{I}_0\SQ_{n;\a}^{c_1,c_2,0}(q)_{b_1,b_2,0}=
\SQ_{n;\a}^{c_1,c_2}(q)_{b_1-1,b_2}+\SQ_{n;\a}^{c_1,c_2}(q)_{b_1,b_2-1}\,, \\ 
\label{SQdil_e}
\SQ_{n;\a}^{c_1,c_2,0}(q)_{b_1,b_2,1}=
-J_{n;\a}(q)\SQ_{n;\a}^{c_1,c_2,1}(q)_{b_1,b_2,0}\,.
\end{gather}
The discrepancy from the corresponding relations of GW-invariants
is exhibited by the power series $\dot{I}_0$ and $\dot{I}_1$ (or equivalently $J_{n;\a}(q)$).\\

\noindent 
By \e_ref{equivGivental_e}, \e_ref{cYdfn_e}, \e_ref{Fdfn_e}, and~\e_ref{Fdfn_e2} 
\BE{Givental_e}
\dot{Z}_{n;\a}(x,\hb,q)=\frac{\dot{F}_{n;\a}(x/h,q/x^{\nu_n(\a)})}{\dot{I}_0(q)},\qquad
\ddot{Z}_{n;\a}(x,\hb,q)=\frac{\ddot{F}_{n;\a}(x/h,q/x^{\nu_n(\a)})}{\ddot{I}_0(q)}\,,\EE
if $\nu_n(a)\!\ge\!0$.\footnote{The right-hand sides of these expressions should be first simplified
in $\Q(x,\hb)[[q]]$, eliminating division by~$x$, and then projected to $H^*(\Pn)[\hb][[q]]$.}
For $s\!\in\Z^+$, define
\begin{equation*}\begin{aligned}
\D^0\dot{F}_{n;\a}(w,q)&=\frac{\dot{F}_{n;\a}(w,q)}{\dot{I}_0(q)},& 
\D^s\dot{F}_{n;\a}(w,q)&=
\frac{1}{\dot{I}_s(q)}
\left\{1+\frac{q}{w}\frac{\nd}{\nd q}\right\}\D^{s-1}\dot{F}_{n;\a}(w,q),\\
\D^0\ddot{F}_{n;\a}(w,q)&=\frac{\ddot{F}_{n;\a}(w,q)}{\ddot{I}_0(q)},&
\D^s\ddot{F}_{n;\a}(w,q)&=\frac{1}{\ddot{I}_s(q)}
\left\{1+\frac{q}{w}\frac{\nd}{\nd q}\right\}\D^{s-1}\ddot{F}_{n;\a}(w,q).
\end{aligned}\end{equation*}
Combining \e_ref{Givental_e} with \e_ref{Z2pt_e} and \e_ref{Zs_e}, we find~that 
\BE{ZvsY_e}
\dot{Z}_{n;\a}(x_1,x_2,\hb_1,\hb_2,q)
= \frac{1}{\hb_1\!+\!\hb_2}
\sum_{\begin{subarray}{c}s_1,s_2\ge0\\ s_1+s_2=n-1 \end{subarray}}
\!\!\!\!
x_1^{s_1}\dot{F}_{n;\a}^{(s_1)}\bigg(\frac{x_1}{\hb_1},\frac{q}{x_1^{\nu_n(\a)}}\bigg)\cdot
x_2^{s_2}\ddot{F}_{n;\a}^{(s_2)}\bigg(\frac{x_2}{\hb_2},\frac{q}{x_2^{\nu_n(\a)}}\bigg)\,,
\EE
where
\BE{Zs_e2}
\ch{F}_{n;\a}^{(s)}(w,q)=\sum_{d=0}^{\i}\sum_{s'=0}^{s-\nu_n(\a)d}
\frac{\ntc_{s-\ell^*(\a),s'-\ell^*(\a)}^{(d)}q^d}{w^{s-\nu_n(\a)d-s'}}
\,\D^{s'}\ch{F}_{n;\a}(w,q),.\EE
with $(\ch{F},\ell^*)=(\dot{F},\ell^-),(\ddot{F},\ell^+)$.\footnote{The right-hand 
side of \e_ref{ZvsY_e} should be first simplified
in $\Q(x_1,x_2,\hb_1,\hb_2)[[q]]$, eliminating division by~$x_1$ and~$x_2$, and 
then projected to $H^*(\Pn\!\times\!\Pn)[\hb_1,\hb_2][[q]]$.}
Thus, \e_ref{Givental_e} and Theorem~\ref{main_thm} provide
closed formulas for the twisted genus~0 two-point and three-point
SQ-invariants of projective spaces.\\

\noindent
The equivariant versions of the generating functions $\dot{Z}_{n;\a}$ defined 
in~\e_ref{Z23full_e} are ideally suited for further computations, such as of genus~0 
invariants with more marked points and of positive-genus twisted invariants with 
at least one marked point.  
However, for the purposes of computing the genus~0 two-point and three-point invariants, 
it is more natural to consider the generating functions
\BE{Z2dfn_e2}\begin{split}
Z_{n;\a}^*(x_1,x_2,\hb_1,\hb_2,q) &\equiv
\sum_{d=1}^{\i}q^d\big\{\ev_1\!\times\!\ev_2\}_*\!\!\left[\frac{e(\V_{n;\a}^{(d)})}
{(\hb_1\!-\!\psi_1)(\hb_2\!-\!\psi_2)}\right],\\
Z_{n;\a}^*(x_1,x_2,x_3,\hb_1,\hb_2,\hb_3,q) &\equiv
\sum_{d=1}^{\i}q^d\big\{\ev_1\!\times\!\ev_2\!\times\!\ev_3\}_*\!\!
\left[\frac{e(\V_{n;\a}^{(d)})}{(\hb_1\!-\!\psi_1)(\hb_2\!-\!\psi_2)(\hb_3\!-\!\psi_3)}\right],
\end{split}\EE
where $\V_{n;\a}^{(d)}$ is given by~\e_ref{Vstandfn_e} and the evaluation maps are 
as in~\e_ref{evaltot_e}.
In the case $\ell(\a)\!\ge\!0$, \e_ref{ZvsY_e} immediately gives
\BE{ZvsY_e2a}\begin{split}
&Z_{n;\a}^*(x_1,x_2,\hb_1,\hb_2,q)\\
&\qquad= \frac{\lr\a x_1^{\ell(\a)}}{\hb_1\!+\!\hb_2}
\sum_{\begin{subarray}{c}s_1,s_2\ge0 \\ s_1+s_2=n-1\end{subarray}}
\!\!\Bigg(-x_1^{s_1}x_2^{s_2}+
x_1^{s_1}x_2^{s_2}
\dot{F}_{n;\a}^{(s_1)}\bigg(\frac{x_1}{\hb_1},\frac{q}{x_1^{\nu_n(\a)}}\bigg)\cdot
\ddot{F}_{n;\a}^{(s_2)}\bigg(\frac{x_2}{\hb_2},\frac{q}{x_2^{\nu_n(\a)}}\bigg)\Bigg)\,
\end{split}\EE
and similarly for the three-point generating function in~\e_ref{Z2dfn_e2}.
In general, \e_ref{cZ2pt_e2}, \e_ref{cZs_e2}, \e_ref{cZs_e}, 
the second identity in \e_ref{equivGivental_e}, \e_ref{cYvsFdots_e},
the middle identity in \e_ref{cCred_e}, and~\e_ref{Zs_e2}  
give
\BE{ZvsY_e2}\begin{split}
Z_{n;\a}^*(x_1,x_2,\hb_1,\hb_2,q)&=\frac{\lr\a}{\hb_1\!+\!\hb_2}
\sum_{\begin{subarray}{c} s_1,s_2\ge0\\ s_1+s_2=n-1  \end{subarray}}
\!\!\Bigg(x_1^{s_1}
x_2^{s_2+\ell(\a)}\dot{F}_{n;\a}^{(s_2)*}\bigg(\frac{x_2}{\hb_2},\frac{q}{x_2^{\nu_n(\a)}}\bigg)\\
&\qquad\qquad+
x_1^{s_1+\ell(\a)}x_2^{s_2}
\dot{F}_{n;\a}^{(s_1)*}\bigg(\frac{x_1}{\hb_1},\frac{q}{x_1^{\nu_n(\a)}}\bigg)
\ddot{F}_{n;\a}^{(s_2)}\bigg(\frac{x_2}{\hb_2},\frac{q}{x_2^{\nu_n(\a)}}\bigg)\Bigg),
\end{split}\EE
where $\dot{F}_{n;\a}^{(s)*}(w,q)\equiv\dot{F}_{n;\a}^{(s)}(w,q)-1$.\footnote{The right-hand 
side of \e_ref{ZvsY_e2} should be first simplified
in $\Q(x_1,x_2,\hb_1,\hb_2)[[q]]$, eliminating division by~$x_1$ and~$x_2$, and 
then projected to $H^*(\Pn\!\times\!\Pn)[\hb_1,\hb_2][[q]]$.}\\

\noindent
An analogue of~\e_ref{ZvsY_e2} for the three-point generating function in~\e_ref{Z2dfn_e2} can
be similarly obtained from~\e_ref{cZ3pt_e2}, the last identity in~\e_ref{cCred_e}, 
and~\e_ref{Cntcred_e}.
In particular, in the Calabi-Yau case, $\nu_n(\a)\!=\!0$, 
\BE{ZvsY_e3}\begin{split}
&Z_{n;\a}^*(x_1,x_2,x_3,\hb_1,\hb_2,\hb_3,q) 
= \frac{\lr\a}{\hb_1\hb_2\hb_3}  
\Bigg\{\hspace{-.2in}
\sum_{\begin{subarray}{c}s_1,s_2,s_3\ge0\\ s_1,s_2,s_3\le n-1 \\
s_1+s_2+s_3=2n-2\\ \end{subarray}}\hspace{-.2in}
\Bigg(x_1^{s_1}x_2^{s_2}x_3^{s_3+\ell(\a)}
\dot{F}_{n;\a}^{(s_3)*}\bigg(\frac{x_3}{\hb_3},q\bigg)\\
&\hspace{1.3in}
+x_1^{s_1}x_2^{s_2+\ell(\a)}x_3^{s_3}
\dot{F}_{n;\a}^{(s_2)*}\bigg(\frac{x_2}{\hb_2},q\bigg)
\ddot{F}_{n;\a}^{(s_3)}\bigg(\frac{x_3}{\hb_3},q\bigg)\\
&\hspace{1.3in}
+x_1^{s_1+\ell(\a)}x_2^{s_2}x_3^{s_3}
\ntc_{s_1,s_2,s_3}(q)
\dot{F}_{n;\a}^{(s_1)*}\bigg(\frac{x_1}{\hb_1},q\bigg)
\prod_{t=2}^3\ddot{F}_{n;\a}^{(s_t)}\bigg(\frac{x_t}{\hb_t},q\bigg)\Bigg)\\
&\hspace{1.6in}+ 
\sum_{\begin{subarray}{c}s_1\ge\ell^-(\a),\,s_2,s_3\ge0\\ s_1,s_2,s_3\le n-1 \\
s_1+s_2+s_3=2n-2\\ \end{subarray}}\hspace{-.3in}
x_1^{s_1+\ell(\a)}x_2^{s_2}x_3^{s_3}
\ntc_{s_1,s_2,s_3}^*(q)
\prod_{t=2}^3\ddot{F}_{n;\a}^{(s_t)}\bigg(\frac{x_t}{\hb_t},q\bigg)\Bigg\},
\end{split}\EE
where
$$\ntc_{s_1,s_2,s_3}(q)\equiv 1+ \ntc_{s_1,s_2,s_3}^*(q)
= \frac1{(1\!-\!\a^{\a}q)\dot\bI_{s_1}^c(q)\ddot\bI_{s_2}^c(q)\ddot\bI_{s_3}^c(q)}\,.$$
This presentation of the three-point formula eliminates division by~$x_1$,
even if $\ell(\a)\!<\!0$, since $\dot{F}_{n;\a}^{(s)*}(w,q)$ is
divisible by~$w^{\ell^-(\a)-s}$ for $s\!\le\!\ell^-(\a)$.\\

\noindent
In the Calabi-Yau case, i.e.~$\nu_n(\a)\!=\!0$, we find that
\BE{CYform_e}
\lr\a+q\frac{\nd}{\nd q}\SQ_{n;\a}^{c_1,c_2}(q)=\lr\a \dot{I}_{c_1+1}(q)\,,\qquad
\SQ_{n;\a}^{c_1,c_2,c_3}(q)
=\frac{\lr{\a}}{(1\!-\!\a^{\a}q)
\prod\limits_{t=1}^{t=3}\prod\limits_{c=0}^{c=c_t}\!\dot{I}_c(q)}\,,\EE
whenever $c_1,c_2,c_3\!\in\Z^{\ge0}$, $c_1\!+\!c_2=n\!-\!2\!-\!\ell(\a)$
in the first equation, and $c_1\!+\!c_2\!+\!c_3=n\!-\!1\!-\!\ell(\a)$
in the second equation.
The $c_1\!=\!0$ case of~\e_ref{CYform_e} agrees with 
the $W\!\sslash\!G\!=\!X_{n;\a}$ case of \cite[Corollary~5.5.4(bc)]{CK}.
By~\e_ref{CYform_e},
\begin{equation*}\begin{split}
\max(c_1,c_2)\ge n\!-\!\ell^+(\a) \qquad&\Lra\qquad
\SQ_{n;\a}^d(c_1,c_2)(q)=0 \quad\forall\,d\!\in\!\Z^+,\\
\max(c_1,c_2,c_3)\ge n\!-\!\ell^+(\a) \qquad&\Lra\qquad
\SQ_{n;\a}^d(c_1,c_2,c_3)(q)=0 \quad\forall\,d\!\in\!\Z^+,
\end{split}\end{equation*}
as the case should be for intrinsic invariants of~$X_{n;\a}$.
On the other hand,  
\BE{CYformGW_e}
\lr\a+Q\frac{\nd}{\nd Q}
\GW_{n;\a}^{c_1,c_2}(Q)=\lr\a \frac{\dot{I}_{c_1+1}(q)}{\dot{I}_1(q)}\,,\qquad
\GW_{n;\a}^{c_1,c_2,c_3}(Q)
=\frac{\lr{\a}\dot{I}_0(q)}{(1\!-\!\a^{\a}q)
\prod\limits_{t=1}^{t=3}\prod\limits_{c=0}^{c=c_t}\!\dot{I}_c(q)},\EE
with the same assumptions on $c_1,c_2,c_3$ as in \e_ref{CYform_e} 
and $Q=q\ne^{J_{n;\a}(q)}$, as before;
see \cite[(1.5)]{bcov0_ci} and \cite[(1.7)]{g0ci}, respectively.\\

\noindent
In the case of  products of projective spaces and concavex sheaves~\e_ref{gensheaf_e},
the analogues of the above mirror formulas relate power series
\begin{alignat}{1}
\label{genY_e}
&\ch{F}_{n_1,\ldots,n_p;\a}\in\Q(w)\big[\big[q_1,\ldots,q_p\big]\big],\\
\label{genZ_e1}
&\ch{Z}_{n_1,\ldots,n_p;\a}^{(s_1,\ldots,s_p)}
\in H^*\big(\P^{n_1-1}\!\times\!\ldots\!\times\!\P^{n_p-1}\big)\big[\hb^{-1}\big]
\big[\big[q_1,\ldots,q_p\big]\big], \\
\label{genZ_e2}
&\ch{Z}_{n_1,\ldots,n_p;\a}^*
\in H^*\big((\P^{n_1-1}\!\times\!\ldots\!\times\!\P^{n_p-1})^m\big)
\big[\hb_1^{-1},\ldots,\hb_m^{-1}\big]
\big[\big[q_1,\ldots,q_p\big]\big],
\end{alignat}
with $\ch{F}$ and $\ch{Z}$ denoting $F$, $\dot{F}$, $\ddot{F}$,
$Z$, $\dot{Z}$, or $\ddot{Z}$ and $m\!=\!2,3$.
The coefficients of $q_1^{d_1}\!\ldots\!q_p^{d_p}$ in~\e_ref{genZ_e1} and~\e_ref{genZ_e2}
are defined by the same pushforwards as in~ \e_ref{Z2dfn_e}, \e_ref{Z3Jdfn_e}, 
\e_ref{Zsdfn_e}, and~\e_ref{Z2dfn_e2}, with the degree~$d$
of the stable quotients replaced by~$(d_1,\ldots,d_p)$
and $x^s$ by $x_1^{s_1}\!\ldots\!x_p^{s_p}$.
The coefficients of $q_1^{d_1}\!\ldots\!q_p^{d_p}$ in~\e_ref{genY_e}
are obtained from the coefficients in~\e_ref{Fdfn_e}, \e_ref{Fdfn_e2},
and~\e_ref{hatFdfn_e}
by replacing $a_kd$ and $a_kx$
by $a_{k;1}d_1\!+\!\ldots\!+\!a_{k;p}d_p$ and $a_{k;1}x_1\!+\!\ldots\!+\!a_{k;p}x_p$
in the numerator and taking the product of the denominators with
$(n,x,d)\!=\!(n_i,x_i,d_i)$ for each $i\!=\!1,\ldots,p$;
$$x_1,\ldots,x_p\in H^*(\P^{n_1-1}\!\times\!\ldots\!\times\!\P^{n_p-1})$$
now correspond to the pullbacks of the hyperplane classes by the projection maps.
If $\ell^-(\a)\!=\!0$, the analogue of~\e_ref{ZvsY_e2a} with $\lr\a x^{\ell(\a)}$ replaced
by the products of $a_{k;1}x_{1;1}\!+\!\ldots\!+\!a_{k;p}x_{1;p}$ and sums over pairs of $p$-tuples 
$(s_{1;1},\ldots,s_{1;p})$ and $(s_{2;1},\ldots,s_{2;p})$ with 
$s_{1;i}\!+\!s_{2;i}\!=\!n_i\!-\!1$ provides a closed formula 
for~$Z_{n_1,\ldots,n_p;\a}^*$.
In general, the relation~\e_ref{ZvsY_e2} extends to this case by replacing
$\lr\a x_i^{\ell(\a)}$ by the products and ratios of the terms 
$a_{k;1}x_{i;1}\!+\!\ldots\!+\!a_{k;p}x_{i;p}$.

\section{Equivariant mirror formulas}
\label{equivthm_sec}

\noindent
We begin this section by reviewing the equivariant setup used 
in \cite{bcov0,bcov0_ci,GWvsSQ}, closely following \cite[Section~3]{GWvsSQ}.
After defining equivariant versions of the generating functions 
$\dot{Z}_{n;\a}^{(s)}$, $\ddot{Z}_{n;\a}^{(s)}$, 
$\dot{Z}_{n;\a}^*$, and~$Z_{n;\a}^*$ 
and of the hypergeometric series $\dot{F}_{n;\a}$ and~$\ddot{F}_{n;\a}$,
we state an equivariant version of Theorem~\ref{main_thm}; 
see Theorem~\ref{equiv_thm} below.
This theorem immediately implies Theorem~\ref{main_thm}.
The proof of the two-point mirror formulas in Theorem~\ref{equiv_thm}
is outlined in Sections~\ref{pfs_sec} and~\ref{poliC_subs} 
and completed in Sections~\ref{recpf_sec} and~\ref{MPCpf_sec}.
We conclude this section with a specialization of the three-point formula of 
Theorem~\ref{equiv_thm} in Proposition~\ref{Z3equiv_prp}, 
which is proved in Section~\ref{equiv0pf_sec} and is a key step
in the proof of the full three-point formula of 
Theorem~\ref{equiv_thm} in Section~\ref{3ptpf_sec}.

\subsection{Equivariant setup}
\label{equivsetup_subs}

\noindent
The quotient of the classifying space for the $n$-torus $\T$ is 
$B\T\equiv(\P^{\i})^n$.
Thus, the group cohomology of~$\T$ is
$$H_{\T}^*\equiv H^*(B\T)=\Q[\al_1,\ldots,\al_n],$$
where $\al_i\!\equiv\!\pi_i^*c_1(\ga^*)$,
$\ga\!\lra\!\P^{\i}$ is the tautological line bundle,
and $\pi_i\!: (\P^{\i})^n\!\lra\!\P^{\i}$ is
the projection to the $i$-th component.
The field of fractions of $H^*_{\T}$ will be denoted by
$$\Q_{\al}\equiv \Q(\al_1,\ldots,\al_n).$$
We denote the equivariant $\Q$-cohomology of a topological space $M$
with a $\T$-action by $H_{\T}^*(M)$.
If the $\T$-action on $M$ lifts to an action on a complex vector bundle $V\!\lra\!M$,
let $\E(V)\in H_{\T}^*(M)$ denote the \sf{equivariant euler class of} $V$.
A continuous $\T$-equivariant map $f\!:M\!\lra\!M'$ between two compact oriented 
manifolds induces a pushforward homomorphism
$$f_*\!: H_{\T}^*(M) \lra H_{\T}^*(M').$$\\

\noindent
The standard action of $\T$ on $\C^{n-1}$,
$$\big(\ne^{\I\th_1},\ldots,\ne^{\I\th_n}\big)\cdot (z_1,\ldots,z_n) 
\equiv\big(\ne^{\I\th_1}z_1,\ldots,\ne^{\I\th_n}z_n\big),$$
descends to a $\T$-action on $\Pn$, which has $n$ fixed points:
\begin{equation}\label{fixedpt_e} 
P_1=[1,0,\ldots,0], \qquad P_2=[0,1,0,\ldots,0], 
\quad\ldots,\quad P_n=[0,\ldots,0,1].
\end{equation}
This standard $\T$-action on~$\Pn$ lifts to a natural 
$\T$-action on the tautological line bundle $\ga\!\lra\!\Pn$,
since $\ga\!\subset\!\Pn\!\times\!\C^n$ is preserved by the diagonal $\T$-action.
With 
$$\x\equiv\E(\ga^*)\in H_{\T}^*(\Pn)$$
denoting the equivariant hyperplane class,
the equivariant cohomology of $\Pn$ is given~by
\BE{pncoh_e}
H_{\T}^*(\Pn)= \Q[\x,\al_1,\ldots,\al_n]\big/(\x\!-\!\al_1)\ldots(\x\!-\!\al_n).\EE
Let $\x_t\in H_{\T}^*((\Pn)^m)$ be the pull-back
of $\x$ by the $t$-th projection map.\\

\noindent
The standard $\T$-representation on $\C^n$ (as well as any other representation) induces
$\T$-actions on $\ov{Q}_{0,m}(\Pn,d)$, $\cU$, $\V_{n;\a}^{(d)}$, $\dot\V_{n;\a}^{(d)}$, 
and $\ddot\V_{n;\a}^{(d)}$; see \e_ref{Vstandfn_e} and~\e_ref{Vprdfn_e} for the notation.
Thus, $\V_{n;\a}^{(d)}$, $\dot\V_{n;\a}^{(d)}$, and $\ddot\V_{n;\a}^{(d)}$ have well-defined
equivariant euler classes 
$$\E(\V_{n;\a}^{(d)}), \E(\dot\V_{n;\a}^{(d)}), \E(\ddot\V_{n;\a}^{(d)})
\in H_{\T}^*\big(\ov{Q}_{0,m}(\Pn,d)\big).$$
The universal cotangent line bundle for the $i$-th marked point
also has a well-defined equivariant Euler class, which will still be
denoted by~$\psi_i$.\\

\noindent
Similarly to~\e_ref{Zdfn_e} and~\e_ref{otherZdfn_e}, let
\BE{cZdfn_e}\begin{split}
\dot\cZ_{n;\a}(\x,\hb,q) &\equiv
1+\sum_{d=1}^{\i}q^d\ev_{1*}\!\!\left[\frac{\E(\dot\V_{n;\a}^{(d)})}{\hb\!-\!\psi_1}\right]
\in H_{\T}^*(\Pn)\big[\big[\hb^{-1},q\big]\big],\\
\ddot\cZ_{n;\a}(\x,\hb,q) &\equiv
1+\sum_{d=1}^{\i}q^d\ev_{1*}\!\!\left[\frac{\E(\ddot\V_{n;\a}^{(d)})}{\hb\!-\!\psi_1}\right]
\in H_{\T}^*(\Pn)\big[\big[\hb^{-1},q\big]\big].
\end{split}\EE
For each $s\!\in\!\Z^{\ge0}$, let
\BE{cZsdfn_e}\begin{split}
\dot\cZ_{n;\a}^{(s)}(\x,\hb,q) &\equiv x^s+
\sum_{d=1}^{\i}\!q^d\ev_{1*}\!\!\left[\frac{\E(\dot\V_{n;\a}^{(d)})\ev_2^*\x^s}
{\hb\!-\!\psi_1}\right]\in H_{\T}^*(\Pn)\big[\big[\hb^{-1},q\big]\big],\\
\ddot\cZ_{n;\a}^{(s)}(\x,\hb,q) &\equiv x^s+
\sum_{d=1}^{\i}\!q^d\ev_{1*}\!\!\left[\frac{\E(\ddot\V_{n;\a}^{(d)})\ev_2^*\x^s}
{\hb\!-\!\psi_1}\right]\in H_{\T}^*(\Pn)\big[\big[\hb^{-1},q\big]\big].
\end{split}\EE
Similarly to~\e_ref{Z2dfn_e} and \e_ref{Z3Jdfn_e}, we define
\begin{gather}\label{cZ2dfn_e}
\dot\cZ_{n;\a}^*(\x_1,\x_2,\hb_1,\hb_2,q) \equiv
\sum_{d=1}^{\i}q^d\big\{\ev_1\!\times\!\ev_2\}_*\!\!\left[\frac{\E(\dot\V_{n;\a}^{(d)})}
{(\hb_1\!-\!\psi_1)(\hb_2\!-\!\psi_2)}\right]
\in H_{\T}^*(\Pn)\big[\big[\hb_1^{-1},\hb_2^{-1},q\big]\big],\\
\label{cZ3Jdfn_e}
\dot\cZ_{n;\a}^*(\x_1,\x_2,\x_3,\hb_1,\hb_2,\hb_3,q) \equiv
\sum_{d=1}^{\i}q^d\big\{\ev_1\!\times\!\ev_2\!\times\!\ev_3\}_*\!
\!\left[\frac{e(\dot\V_{n;\a}^{(d)})}
{(\hb_1\!-\!\psi_1)(\hb_2\!-\!\psi_2)(\hb_3\!-\!\psi_3)}\right],
\end{gather}
with the total pushforwards by the total evaluation maps taken in equivariant 
cohomology.
Similarly to~\e_ref{Z23full_e}, let
\BE{Z23cfull_e}\begin{split}
\dot\cZ_{n;\a}(\x_1,\x_2,\hb_1,\hb_2,q)
&=\frac{\bPD(\De_{\Pn}^{(2)})}{\hb_1\!+\!\hb_2}
+\dot\cZ_{n;\a}^*(\x_1,\x_2,\hb_1,\hb_2,q)\,,\\
\dot\cZ_{n;\a}(\x_1,\x_2,\x_3,\hb_1,\hb_2,\hb_3,q)
&=\frac{\bPD(\De_{\Pn}^{(3)})}{\hb_1\hb_2\hb_3}
+\dot\cZ_{n;\a}^*(\x_1,\x_2,\x_3,\hb_1,\hb_2,\hb_3,q),
\end{split}\EE
where $\bPD(\De_{\Pn}^{(2)})$ and $\bPD(\De_{\Pn}^{(3)})$ are the equivariant
Poincar\'{e} duals of the (small) diagonals in 
$\Pn\!\times\!\Pn$ and $\Pn\!\times\!\Pn\!\times\!\Pn$, respectively.\\

\noindent
The above Poincar\'{e} duals can be written~as 
\BE{PDequiv_e}\begin{split}
\bPD(\De_{\Pn}^{(2)})&=
\sum_{\begin{subarray}{c}s_1,s_2,r\ge0\\ s_1+s_2+r=n-1 \end{subarray}}
\!\!\!\!\!\!\!\!\!(-1)^r \bfs_r\x_1^{s_1}\x_2^{s_2}\,,\\
\bPD(\De_{\Pn}^{(3)})&=
\sum_{\begin{subarray}{c}s_1,s_2,s_3,r\ge0\\ s_1+s_2+s_3+r=2n-2 \end{subarray}}
\!\!\!\!\!\!\!\!\!\bfs_r^{(2)}\x_1^{s_1}\x_2^{s_2}\x_3^{s_3}\\
&=\sum_{\begin{subarray}{c}s_1,s_2,s_3,r\ge0\\ 
s_1,s_2,s_3\le n-1\\
s_1+s_2+s_3+r=2n-2 \end{subarray}}
\sum_{\begin{subarray}{c}r_0,r_1,r_2,r_3\ge0\\
r_1\le\hat{s}_1,r_2\le\hat{s}_2,r_3\le\hat{s}_3\\ 
r_0+r_1+r_2+r_3=r\end{subarray}}\!\!\!\!\!\!\!
(-1)^{r_1+r_2+r_3}\eta_{r_0}\bfs_{r_1}\bfs_{r_2}\bfs_{r_3}
\x_1^{s_1}\x_2^{s_2}\x_3^{s_3}\,,
\end{split}\EE
where $\bfs_r,\eta_r,\bfs_r^{(2)}\!\in\!\Q[\al_1,\ldots,\al_n]$ are 
the $r$-th elementary symmetric polynomial
 in $\al_1,\ldots,\al_n$,
the sum of all degree~$r$ monomials in $\al_1,\ldots,\al_n$, and
the degree~$r$ term in $(1\!-\!\bfs_1\!+\!\bfs_2-\ldots)^2$,
respectively.
All three expressions for the Poincar\'{e} duals can be confirmed by pairing
them with $\x_1^{t_1}\x_2^{t_2}$ and $\x_1^{t_1}\x_2^{t_2}\x_3^{t_3}$,
with $t_1,t_2,t_3\!\le\!n\!-\!1$, and using the Localization Theorem of~\cite{ABo}
on $(\Pn)^m$
and the Residue Theorem on $S^2$ to reduce the equivariant integrals of 
$\x^{s+t}$ on~$\Pn$ to the polynomials~$\eta_r$;
these are the homogeneous polynomials in the power series expansion of 
$1/(1\!-\!\al_1)(1\!-\!\al_2)\ldots$
The coefficient of $\x_1^{s_1}\x_2^{s_2}\x_3^{s_3}$
in the second expression for  $\bPD(\De_{\Pn}^{(3)})$ is precisely
$\wt\cC_{s_1,s_2,s_3}^{(r)}(0)$, with $\wt\cC_{s_1,s_2,s_3}^{(r)}$
as in Theorem~\ref{equiv_thm}; see the end of Section~\ref{equivMS_subs}.
This provides a direct check of the degree~0 term in~\e_ref{Z2cJpt_e}.

\subsection{Equivariant mirror symmetry}
\label{equivMS_subs}

\noindent
The equivariant analogues of the power series in~\e_ref{Fdfn_e} 
and~\e_ref{Fdfn_e2}  are given~by
\BE{cYdfn_e}\begin{split}
\dot\cY_{n;\a}(\x,\hb,q)&\equiv\sum_{d=0}^{\i}q^d
\frac{\prod\limits_{a_k>0}\prod\limits_{r=1}^{a_kd}\!(a_k\x\!+\!r\hb)
\prod\limits_{a_k<0}\!\!\prod\limits_{r=0}^{-a_kd-1}\!\!(a_k\x\!-\!r\hb)}
{\prod\limits_{r=1}^d\left(\prod\limits_{k=1}^n(\x\!-\!\al_k\!+\!r\hb)-
\prod\limits_{k=1}^n(\x\!-\!\al_k)\right)}
\in \Q[\al_1,\ldots,\al_n,\x]\big[\big[\hb^{-1},q\big]\big],\\
\ddot\cY_{n;\a}(\x,\hb,q)&\equiv\sum_{d=0}^{\i}q^d
\frac{\prod\limits_{a_k>0}\prod\limits_{r=0}^{a_kd-1}\!(a_k\x\!+\!r\hb)
\prod\limits_{a_k<0}\!\!\prod\limits_{r=1}^{-a_kd}\!\!(a_k\x\!-\!r\hb)}
{\prod\limits_{r=1}^d\left(\prod\limits_{k=1}^n(\x\!-\!\al_k\!+\!r\hb)-
\prod\limits_{k=1}^n(\x\!-\!\al_k)\right)}
\in \Q[\al_1,\ldots,\al_n,\x]\big[\big[\hb^{-1},q\big]\big].
\end{split}\EE
The second products in the denominators above are irrelevant for 
the statements in this section, but are material 
to~\e_ref{cYexp_e} and thus to the proof of~\e_ref{Z2cJpt_e} in this paper.\\

\noindent
For each $s\!\in\!\Z^+$, we define 
$\D^s\dot\cZ_{n;\a}(\x,\hb,q),\D^s\ddot\cZ_{n;\a}(\x,\hb,q)\in H_{\T}^*(\Pn)[[\hb^{-1},q]]$ 
inductively~by
\BE{DscZdfn_e}\begin{aligned}
\D^0\dot\cZ_{n;\a}(\x,\hb,q)&=\dot\cZ_{n;\a}(\x,\hb,q),&
\D^s\dot\cZ_{n;\a}(\x,\hb,q)&=
\frac{1}{\dot{I}_s(q)}
\left\{\x+\hb\, q\frac{\nd}{\nd q}\right\}\D^{s-1}\dot\cZ_{n;\a}(\x,\hb,q),\\
\D^0\ddot\cZ_{n;\a}(\x,\hb,q)&=\ddot\cZ_{n;\a}(\x,\hb,q),&
\D^s\ddot\cZ_{n;\a}(\x,\hb,q)&=
\frac{1}{\ddot{I}_s(q)}
\left\{\x+\hb\, q\frac{\nd}{\nd q}\right\}\D^{s-1}\ddot\cZ_{n;\a}(\x,\hb,q).
\end{aligned}\EE

\begin{thmnum}\label{equiv_thm}
If $l\!\in\!\Z^{\ge0}$, $n\!\in\!\Z^+$, and $\a\!\in\!(\Z^*)^l$, then
\BE{main_e2}
\dot\cZ_{n;\a}(\x_1,\x_2,\hb_1,\hb_2,q)
 =\frac{1}{\hb_1+\hb_2}
\sum_{\begin{subarray}{c}s_1,s_2,r\ge0\\ s_1+s_2+r=n-1 \end{subarray}}
\!\!\!\!\!\!\!\!\!(-1)^r\bfs_r
\dot\cZ_{n;\a}^{(s_1)}(\x_1,\hb_1,q)\ddot\cZ_{n;\a}^{(s_2)}(\x_2,\hb_2,q),
\EE
where $\bfs_r\!\in\!\Q_{\al}$ is the $r$-th elementary symmetric polynomial
in $\al_1,\ldots,\al_n$.
If in addition $\nu_n(\a)\!\ge\!0$,  
\BE{equivGivental_e}
\dot\cZ_{n;\a}(\x,\hb,q)=\frac{\dot\cY_{n;\a}(\x,\hb,q)}{\dot{I}_0(q)},\qquad
\ddot\cZ_{n;\a}(\x,\hb,q)=\frac{\ddot\cY_{n;\a}(\x,\hb,q)}{\ddot{I}_0(q)}\,,\EE
and there exist 
$\ctC_{s,s'}^{(r)},\ctC_{s_1,s_2,s_3}^{(r)}\in\Q[\al_1,\ldots,\al_n][[q]]$
with $s,s',s_1,s_2,s_3,r\!\in\!\Z^{\ge0}$ 
such~that  
\BE{cCred_e}
\ctC_{s,s'}^{(r)}(0)=\de_{0,r}\de_{s,s'}, \quad
\bigcoeff{\ctC_{s,s'}^{(\nu_n(\a)d)}(q)\big|_{\al=0}}_d=\ntc_{s,s'}^{(d)}\,,\quad
\bigcoeff{\ctC_{s_1,s_2,s_3}^{(\nu_n(\a)d)}(q)\big|_{\al=0}}_d=\ntc_{s_1,s_2,s_3}^{(d)}\,,\EE
the coefficients of $q^d$ in $\ctC_{s,s'}^{(r)}(q)$ and 
$\ctC_{s_1,s_2,s_3}^{(r)}(q)$ are degree~$r\!-\!\nu_n(\a)d$ 
homogeneous symmetric polynomial in $\al_1,\al_2,\ldots,\al_n$, and
\begin{gather}\label{Z2cJpt_e}
\begin{split}
&\dot\cZ_{n;\a}(\x_1,\x_2,\x_3,\hb_1,\hb_2,\hb_3,q)\\
&\qquad\quad = \frac{1}{\hb_1\hb_2\hb_3}
\sum_{\begin{subarray}{c}r,s_1,s_2,s_3\ge0\\
s_1,s_2,s_3\le n-1 \\
s_1+s_2+s_3+r=2n-2\\
\end{subarray}}
\hspace{-.35in}
\ctC^{(r)}_{s_1,s_2,s_3}(q)
\dot\cZ_{n;\a}^{(s_1)}(\x_1,\hb_1,q)
\prod_{t=2}^3\ddot\cZ_{n;\a}^{(s_t)}(\x_t,\hb_t,q)\,,
\end{split}\\
\label{cZs_e}
\ch\cZ_{n;\a}^{(s)}(\x,\hb,q)=\sum_{r=0}^s\sum_{s'=0}^{s-r}
\ctC_{s-\ell^*(\a),s'-\ell^*(\a)}^{(r)}(q)\,\hb^{s-r-s'}\D^{s'}\ch\cZ_{n;\a}(\x,\hb,q),
\end{gather}
where $(\ch\cZ,\ell^*)=(\dot\cZ,\ell^-),(\ddot\cZ,\ell^+)$.
\end{thmnum} 

\noindent
Setting $\al\!=\!0$ in~\e_ref{main_e2}, \e_ref{equivGivental_e},
\e_ref{Z2cJpt_e}, and \e_ref{cZs_e}, 
we obtain \e_ref{Z2pt_e}, \e_ref{Givental_e}, \e_ref{Z2Jpt_e} and~\e_ref{Zs_e},
respectively.\\

\noindent
We now completely describe the power series $\ctC_{s,s'}^{(r)}$ of Theorem~\ref{equiv_thm};
it will be shown in Section~\ref{pfs_sec}  that they indeed satisfy~\e_ref{cZs_e}.
Let 
\BE{hatdfn_e}
\D^0\cY_{n;\a}(\x,\hb,q)=\frac1{I_0(q)}\sum_{d=0}^{\i}q^d\frac{
\prod\limits_{a_k>0}\prod\limits_{r=1}^{a_kd}(a_k\x\!+\!r\hb)
\prod\limits_{a_k<0}\!\!\prod\limits_{r=1}^{-a_kd}\!\!(a_k\x\!-\!r\hb)}
{\prod\limits_{r=1}^d\prod\limits_{k=1}^n(\x-\al_k+r\hb)}\,.\EE
For $s\!\in\!\Z^+$, let
\BE{DscYdfn_e}\D^s\cY_{n;\a}(\x,\hb,q)=
\frac{1}{I_s(q)}
\left\{\x+\hb\, q\frac{\nd}{\nd q}\right\}\D^{s-1}\cY_{n;\a}(\x,\hb,q)
\in \x^s+q\cdot\Q[\al_1,\ldots,\al_n]\big[\big[q\big]\big].\EE
Comparing with \e_ref{hatFdfn_e}, we find that 
\begin{gather}\label{cYvsF_e}
\D^s\cY_{n;\a}(\x,\hb,q)\big|_{\al=0}=
\x^s\D^sF_{n;\a}(\x/\hb,q/\x^{\nu_n(\a)}),\qquad\hbox{where}\\
\D^0F_{n;\a}(w,q)=\frac{F_{n;\a}(w,q)}{I_0(q)}\,,\quad
\D^sF_{n;\a}(w,q)=\frac{1}{I_s(q)}\left\{1+\frac{q}{w}\frac{\nd}{\nd q}\right\}
\D^{s-1}F_{n;\a}(w,q) \quad\forall\,s\!\in\!\Z^+.\notag
\end{gather}
For $r,s,s'\!\ge\!0$, define $\cC_{s,s'}^{(r)}\in\Q[\al_1,\ldots,\al_n][[q]]$ by
\BE{Crec_e}
\hb^s\sum_{s'=0}^{\i}\sum_{r=0}^{s'}\cC_{s,s'}^{(r)}(q)\x^{s'-r}\hb^{-s'}
=\D^s\cY_{n;\a}(\x,\hb,q)\,.\EE
By \e_ref{hatdfn_e}, \e_ref{DscYdfn_e}, and~\e_ref{Crec_e},  
the coefficient of $q^d$ in $\cC^{(r)}_{s,s'}$ is a degree~$r\!-\!\nu_n(\a)d$ 
homogeneous symmetric polynomial in~$\al$.
By~\e_ref{DscYdfn_e} and~\e_ref{cYvsF_e}, 
\BE{cC_e}
\cC_{s,s}^{(0)}(q)=1, \qquad \cC_{s,s'}^{(0)}(q)=0~~\forall\,s>s', \qquad 
\cC_{s,s'}^{(r)}(0)=\de_{r,0}\de_{s,s'}\,.\EE
By the first two statements above, the relations
\BE{tiCrec_e}
\sum_{\begin{subarray}{c}r_1,r_2\ge0\\ r_1+r_2=r\end{subarray}}\sum_{t=0}^{s-r_1}
\ctC_{s,t}^{(r_1)}(q)\cC^{(r_2)}_{t,s'-r_1}(q)=\de_{r,0}\de_{s,s'}
\quad\forall~r,s'\!\in\!\Z^{\ge0},\,r\!\le\!s'\!\le\!s,\EE
inductively define $\ctC_{s,s'-r}^{(r)}\in\Q[\al_1,\ldots,\al_n][[q]]$
with $r\!\le\!s'\!\le\!s$ in terms of the power series $\ctC^{(r_1)}_{s,t}$
with $r_1\!<\!r$ or $r_1\!=\!r$ and $t\!<\!s'\!-\!r$.
By~\e_ref{cC_e} and~\e_ref{tiCrec_e}, 
$$\ctC^{(0)}_{s,s'}=\de_{s,s'}\,, \qquad
\ctC^{(r)}_{s,s'}(0)=\de_{r,0}\de_{s,s'}\,,$$
and the coefficient of $q^d$ in $\ctC^{(r)}_{s,s'}$ is a degree~$r\!-\nu_n(\a)d$ 
homogeneous symmetric polynomial in~$\al$.
If $s'\!<\!0$, we set $\ctC^{(r)}_{s,s'}=\de_{r,0}\de_{s,s'}$.
If $\nu_n(\a)\!>\!0$,
$$\cC^{(\nu_n(\a)d)}_{s,s'}\big|_{\al=0}=\nc^{(d)}_{s,s'-\nu_n(\a)d}q^d
\qquad\forall\,s'\!\ge\!\nu_n(\a)d$$
by \e_ref{Crec_e}, \e_ref{cYvsF_e},  and \e_ref{littlec_e}. 
Thus, setting $\al\!=\!0$ in~\e_ref{tiCrec_e} and comparing with~\e_ref{littletic_e}
with~$s'$ replaced by $s'\!-\!\nu_n(\a)d$, we obtain the second identity in~\e_ref{cCred_e}.\\

\noindent
We next completely describe the power series $\ctC_{s_1,s_2,s_3}^{(r)}$ of Theorem~\ref{equiv_thm};
it will be shown in Section~\ref{3ptpf_sec}  that they indeed satisfy~\e_ref{Z2cJpt_e}.
For each $r\!\in\!\Z^{\ge0}$, let $p_r,\cH^{(r)}\!\in\!\Q[z_1,z_2,\ldots]$ be such that 
\BE{Hrdfn_e}p_r(\al_1,\al_2,\ldots)=\al_1^r+\al_2^r+\ldots
=\cH^{(r)}(\bfs_1,\bfs_2,\ldots).\EE
For $r,\nu\!\in\!\Z^{\ge0}$, we define $\cH^{(r)}_{\nu}\!\in\!\Q[\bfs_1,\bfs_2,\ldots][[z]]$ by
\BE{Hrnudfn_e}\cH^{(r)}_{\nu}(z)=
\begin{cases}
(1\!-\!z)^{-1},&\hbox{if}~\nu\!=\!0,~r\!=\!0;\\
\frac{1}{r}\frac{\nd}{\nd z}
\cH^{(r)}\big((1\!-\!z)^{-1}\bfs_1,(1\!-\!z)^{-1}\bfs_2,\ldots\big),
&\hbox{if}~\nu\!=\!0,~r\!\ge\!1;\\
\frac{1}{r+\nu}\frac{\nd}{\nd z}
\cH^{(r+\nu)}(\bfs_1,\ldots,\bfs_{\nu-1},\bfs_{\nu}\!-\!(-1)^{\nu}z,\bfs_{\nu+1},\ldots),
&\hbox{if}~\nu\!>\!0.
\end{cases}\EE
In particular, the coefficient of $z^d$ in $\cH^{(r)}_{\nu}(z)$ is a degree~$r\!-\!\nu d$ 
homogeneous symmetric polynomial in~$\al$, 
\BE{cHrestr_e}
\cH^{(r)}_{\nu}(0)=\eta_r,\qquad
\bigcoeff{\cH^{(\nu d)}_{\nu}(z)\big|_{\al=0}}_d=1.\EE
The second identity above follows from \cite[Lemma~B.3]{g0ci}.
Using induction via Newton's identity \cite[p577]{Artin}, 
the first identity in~\e_ref{cHrestr_e} can be reduced~to
$$\sum_{t=0}^r(-1)^t\eta_{r-t}\bfs_t=0, \quad 
\sum_{t=0}^r(-1)^t(r\!-\!t)\eta_{r-t}\bfs_t=p_r
\quad\forall~r\!\in\!\Z^+;$$
these two identities are equivalent~to
$$\frac{(1\!-\!\al_1 u)(1\!-\!\al_2 u)\ldots}{(1\!-\!\al_1 u)(1\!-\!\al_2 u)\ldots}=1,
\qquad
\frac{\nd}{\nd z}
\frac{(1\!-\!\al_1 u)(1\!-\!\al_2 u)\ldots}{(1\!-\!\al_1 uz)(1\!-\!\al_2 uz)\ldots}
\bigg|_{z=0}=\frac{\al_1 u}{1\!-\!\al_1 u}+\frac{\al_2 u}{1\!-\!\al_2 u}+\ldots$$
Let
\BE{ctC3dfn_e}\ctC^{(r)}_{s_1,s_2,s_3}(q)=
\sum_{\begin{subarray}{c}r_0,r_1,r_2,r_3\ge0\\
r_1\le\hat{s}_1,r_2\le\hat{s}_2,r_3\le\hat{s}_3\\ 
r_0+r_1+r_2+r_3=r\end{subarray}}
\frac{\cH_{\nu_n(\a)}^{(r_0)}(\a^{\a}q)}
{\dot\bI_{s_1+r_1}^c(q)\ddot\bI_{s_2+r_2}^c(q)\ddot\bI_{s_3+r_3}^c(q)}
\ddot\cC^{(r_1)}_{\hat{s}_1}(q)\dot\cC^{(r_2)}_{\hat{s}_2}(q)
\dot\cC^{(r_3)}_{\hat{s}_3}(q)\,,\EE
where
\BE{chcCdfn_e}\ch\cC^{(r)}_s(q)=
\sum_{\begin{subarray}{c}r',r''\ge0\\ r'+r''=r\end{subarray}}(-1)^{r'}\bfs_{r'}
\ch\cC^{(r'')}_{s-r'-\ell^*(\a),s-r-\ell^*(\a)}(q)\,\EE
with $(\ch\cC,\ell^*)=(\dot\cC,\ell^-),(\ddot\cC,\ell^+)$.
Since the coefficients of $q^d$ in $\cH_{\nu}^{(r)}$ and in $\ctC^{(r)}_{s,s'}$ 
are degree~$r\!-\nu_n(\a)d$ homogeneous symmetric polynomials in~$\al$,
the coefficient of $q^d$ in $\ctC^{(r)}_{s_1,s_2,s_3}$ is also 
a degree~$r\!-\nu_n(\a)d$ homogeneous symmetric polynomial in~$\al$.
The last identity in~\e_ref{cCred_e} follows from \e_ref{ctC3dfn_e},
the second identity in
\e_ref{cHrestr_e}, the middle identity in~\e_ref{cCred_e}, 
and~\e_ref{ntc3dfn_e}.

\subsection{Related mirror formulas}
\label{otherform_subs}

\noindent
Similarly to \e_ref{Z2dfn_e2}, we define
\BE{cZ2dfn_e2}\begin{split}
\cZ_{n;\a}^*(\x_1,\x_2,\hb_1,\hb_2,q) &\equiv
\sum_{d=1}^{\i}q^d\big\{\ev_1\!\times\!\ev_2\}_*\!\!\left[\frac{\E(\V_{n;\a}^{(d)})}
{(\hb_1\!-\!\psi_1)(\hb_2\!-\!\psi_2)}\right],\\
\cZ_{n;\a}^*(\x_1,\x_2,\x_3,\hb_1,\hb_2,\hb_3,q) &\equiv
\sum_{d=1}^{\i}q^d\big\{\ev_1\!\times\!\ev_2\!\times\!\ev_3\}_*\!\!
\left[\frac{e(\V_{n;\a}^{(d)})}{(\hb_1\!-\!\psi_1)(\hb_2\!-\!\psi_2)(\hb_3\!-\!\psi_3)}\right],
\end{split}\EE
with the evaluation maps as in~\e_ref{evaltot_e}.
For each $s\!\in\!\Z^{\ge0}$, let
\begin{equation*}\begin{split}
\cZ_{n;\a}^{(s)*}(\x,\hb,q)&\equiv
\sum_{d=1}^{\i}q^d\ev_{1*}\!\!\left[\frac{\E(\V_{n;\a}^{(d)})\ev_2^*\x^s}
{\hb\!-\!\psi_1}\right]\in H_{\T}^*(\Pn)\big[\big[\hb^{-1},q\big]\big],\\
\ddot\cZ_{n;\a}^{(s)*}(\x,\hb,q)&\equiv
\sum_{d=1}^{\i}q^d\ev_{1*}\!\!\left[\frac{\E(\ddot\V_{n;\a}^{(d)})\ev_2^*\x^s}
{\hb\!-\!\psi_1}\right]\in H_{\T}^*(\Pn)\big[\big[\hb^{-1},q\big]\big].
\end{split}\end{equation*}
Since $\x_1,\x_2\!\in\!H_{\T}^*(\Pn\!\times\!\Pn)\!\otimes_{\Q[\al_1,\ldots,\al_n]}\Q_{\al}$
are invertible,
the first equation in~\e_ref{PDequiv_e} gives
\begin{equation*}\begin{split}
\lr\a\!\!\!\!\!\!\!\!\!
\sum_{\begin{subarray}{c}s_1,s_2,r\ge0\\ s_1+s_2+r=n-1 \end{subarray}}
\!\!\!\!\!\!\!\!\!\!\!(-1)^r \bfs_r\x_1^{s_1}\ddot\cZ_{n;\a}^{(s_2)*}(\x_2,\hb,q)
&=\sum_{d=1}^{\i}q^d
\big\{\id\!\times\!\ev_1\big\}_*\!\!\left[
\frac{\pi_2^*\E(\V_{n;\a}^{(d)})\,\{\id\!\times\!\ev_2\}^*(\bPD(\De_{\Pn})\x_2^{-\ell(\a)})}
{\hb\!-\!\psi_1}\right]\\
&=\sum_{d=1}^{\i}q^d
\big\{\id\!\times\!\ev_1\big\}_*\!\!\left[
\frac{\pi_2^*\E(\V_{n;\a}^{(d)})\,\{\id\!\times\!\ev_2\}^*(\bPD(\De_{\Pn})\x_1^{-\ell(\a)})}
{\hb\!-\!\psi_1}\right]\\
&=\x_1^{-\ell(\a)}\!\!\!\!\!\!\!\!\!
\sum_{\begin{subarray}{c}s_1,s_2,r\ge0\\ s_1+s_2+r=n-1 \end{subarray}}
\!\!\!\!\!\!\!\!\!\!\!(-1)^r \bfs_r\x_1^{s_1}\cZ_{n;\a}^{(s_2)*}(\x_2,\hb,q),
\end{split}\end{equation*}
where $\pi_2\!:\Pn\!\times\!\ov{Q}_{0,2}(\Pn,d)\!\lra\!\ov{Q}_{0,2}(\Pn,d)$ is 
the projection map.
Combining the last identity with~\e_ref{main_e2}, we obtain
\BE{cZ2pt_e2}\begin{split}
&\cZ_{n;\a}^*(\x_1,\x_2,\hb_1,\hb_2,q)\\
&\qquad =\frac{1}{\hb_1+\hb_2}
\sum_{\begin{subarray}{c}s_1,s_2,r\ge0\\ s_1+s_2+r=n-1\end{subarray}}
\!\!\!\!\!\!\!\!\!(-1)^r\bfs_r\big(\x_1^{s_1}\cZ_{n;\a}^{(s_2)*}(\x_2,\hb_2,q)+
\cZ_{n;\a}^{(s_1)*}(\x_1,\hb_1,q)\ddot\cZ_{n;\a}^{(s_2)}(\x_2,\hb_2,q)\big).
\end{split}\EE
Similar reasoning gives
\BE{cZ3pt_e2}\begin{split}
&\cZ_{n;\a}^*(\x_1,\x_2,\x_3,\hb_1,\hb_2,\hb_3,q) 
= \frac{1}{\hb_1\hb_2\hb_3}  \hspace{-.2in}
\sum_{\begin{subarray}{c}r,s_1,s_2,s_3\ge0\\ s_1,s_2,s_3\le n-1 \\
s_1+s_2+s_3+r=2n-2\\ \end{subarray}}\hspace{-.2in}
\Bigg(\ctC^{(r)}_{s_1,s_2,s_3}(0)\x_1^{s_1}\x_2^{s_2}
\cZ_{n;\a}^{(s_3)*}(\x_3,\hb_3,q)\\
&\hspace{2in}
+\ctC^{(r)}_{s_1,s_2,s_3}(0)\x_1^{s_1}\cZ_{n;\a}^{(s_2)*}(\x_2,\hb_2,q)
\ddot\cZ_{n;\a}^{(s_3)}(\x_3,\hb_3,q)\\
&\hspace{2in}+\ctC^{(r)}_{s_1,s_2,s_3}(0)
\cZ_{n;\a}^{(s_1)*}(\x_1,\hb_1,q)
\prod_{t=2}^3\ddot\cZ_{n;\a}^{(s_t)}(\x_t,\hb_t,q)\\
&\hspace{2in} 
+\lr\a\x_1^{\ell^-(\a)}\ctC^{(r)*}_{s_1,s_2,s_3}(q)
\dot\cZ_{n;\a}^{(s_1)}(\x_1,\hb_1,q)
\prod_{t=2}^3\ddot\cZ_{n;\a}^{(s_t)}(\x_t,\hb_t,q)\Bigg),
\end{split}
\EE
where $\wt\cC_{s_1,s_2,s_3}^{(r)*}(q)\!=\!\wt\cC_{s_1,s_2,s_3}^{(r)}(q)\!-\! 
\wt\cC_{s_1,s_2,s_3}^{(r)}(0)$.\\

\noindent
On the other hand,  by \e_ref{cZs_e} and the first identity in~\e_ref{equivGivental_e},
\BE{cZs_e2}\begin{split}
\cZ_{n;\a}^{(s)*}(\x,\hb,q)&=-\lr\a\x^{\ell(\a)+s}\\
&\qquad +\lr\a\x^{\ell(\a)}
\sum_{r=0}^s\sum_{s'=0}^{s-r}
\ctC_{s-\ell^-(\a),s'-\ell^-(\a)}^{(r)}(q)\,\hb^{s-r-s'}\D^{s'}\dot\cY_{n;\a}(\x,\hb,q),
\end{split}\EE
where
$$\D^0\ch\cY_{n;\a}(\x,\hb,q)=\frac{\ch\cY_{n;\a}(\x,\hb,q)}{\ch{I}_0(q)}, \quad
\D^s\ch\cY_{n;\a}(\x,\hb,q)=
\frac{1}{\ch{I}_s(q)}
\left\{\x+\hb\, q\frac{\nd}{\nd q}\right\}\D^{s-1}\ch\cY_{n;\a}(\x,\hb,q)$$
for all $s\!\in\!\Z^+$ and 
$(\ch\cY,\ch{I})\!=\!(\dot\cY,\dot{I}),(\ddot\cY,\ddot{I})$.
By \e_ref{cYdfn_e}, \e_ref{Fdfn_e}, and~\e_ref{Fdfn_e2}, 
\begin{gather}\label{cYvsFdots_e}
\D^s\ch\cY_{n;\a}(\x,\hb,q)\big|_{\al=0}=
\x^s\D^s\ch{F}_{n;\a}(\x/\hb,q/\x^{\nu_n(\a)}),\qquad\hbox{where}\\
\D^0\ch{F}_{n;\a}(w,q)=\frac{\ch{F}_{n;\a}(w,q)}{\ch{I}_0(q)}\,,\quad
\D^s\ch{F}_{n;\a}(w,q)=\frac{1}{\ch{I}_s(q)}\left\{1+\frac{q}{w}\frac{\nd}{\nd q}\right\}
\D^{s-1}\ch{F}_{n;\a}(w,q) \quad\forall\,s\!\in\!\Z^+,\notag
\end{gather}
with $(\ch\cY,\ch{F},\ch{I})\!=\!(\dot\cY,\dot{F},\dot{I}),(\ddot\cY,\ddot{F},\ddot{I})$.
Simplifying the right-hand side of \e_ref{cZs_e2} in 
$\Q_{\al}(\x,\hb)[[\hb^{-1},q]]$ to eliminate division by~$\x$  
and setting $\al\!=\!0$, we obtain~\e_ref{ZvsY_e2}.

\subsection{Other three-point generating functions}
\label{OtherCon_subs}

\noindent
The main step in the proof of the mirror formula~\e_ref{Z2cJpt_e}
for the stable quotients analogue of the triple Givental's $J$-function
involves determining a mirror formula for the generating function
\BE{cZ3dfn_e0}
\dot\cZ_{n;\a;3}^{(\0,\1)}(\x,\hb,q)\equiv
1+\sum_{d=1}^{\i}q^d\ev_{1*}\!\!\left[\frac{\E(\dot\V_{n;\a}^{(d)})}{\hb\!-\!\psi_1}\right]
\in H_{\T}^*(\Pn)\big[\big[\hb^{-1},q\big]\big],\EE
where $\ev_1\!:\ov{Q}_{0,3}(\Pn,d)\!\lra\!\Pn$ is the evaluation map at the first marked point.
By~\e_ref{cZ3vsY_e}, the SQ-invariants do {\it not} satisfy the string relation
\cite[Section~26.3]{MirSym} in the pure Calabi-Yau cases, $\nu_n(\a)\!=\!0$ and $\ell^-(\a)\!=\!0$ 
(when $\dot{I}_0(q)\!\neq\!1$), even though the relevant forgetful morphism, $f_{2,3}$ below, is
defined.
Since in these cases the twisted invariants of~$\Pn$ are intrinsic invariants 
of the corresponding complete intersection~$X_{n;\a}$,
this implies that the construction of virtual fundamental class in~\cite{CKM}
does not respect the forgetful morphism 
$$f_{2,3}\!:\ov{Q}_{0,3}(X_{n;\a},d)\lra\ov{Q}_{0,2}(X_{n;\a},d),$$
at least in the Calabi-Yau cases.

\begin{prp}\label{Z3equiv_prp}
If $l\!\in\!\Z^{\ge0}$, $n\!\in\!\Z^+$, and $\a\!\in\!(\Z^*)^l$ are such that
$\nu_n(\a)\!\ge\!0$, then 
\BE{cZ3vsY_e} \dot\cZ_{n;\a;3}^{(\0,\1)}(\x,\hb,q)=\hb^{-1} 
\frac{\dot\cZ_{n;\a}(\x,\hb,q)}{\dot{I}_0(q)}.\EE\\
\end{prp}

\noindent
In principle, this proposition is contained in \cite[Corollary~1.4.1]{CK}.
We give a direct proof, along the lines of~\cite{GWvsSQ}.
In the process of proving this proposition, we establish the mirror formula
for equivariant Hurwitz numbers in Proposition~\ref{equiv0_prp}.
This in turn allows us to derive~\e_ref{Z2cJpt_e}  from~\e_ref{main_e2}
and~\e_ref{cZs_e} following the approach of~\cite{g0ci}; see Section~\ref{3ptpf_sec}.\\

\noindent
Similarly to~\e_ref{cZ3dfn_e0}, let
\BE{cZ3dfn_e}\begin{split}
\dot\cZ_{n;\a;2}^{(\0,\1)}(\x,\hb,q)&\equiv
1+\sum_{d=1}^{\i}q^d\ev_{1*}\!\!\left[\frac{f_{2,3}^*\E(\dot\V_{n;\a}^{(d)})}{\hb\!-\!\psi_1}\right]
\in H_{\T}^*(\Pn)\big[\big[\hb^{-1},q\big]\big],
\end{split}\EE
where $\ev_1\!:\ov{Q}_{0,3}(\Pn,d)\!\lra\!\Pn$ is the evaluation map at the first marked point
and 
$$f_{2,3}\!:\ov{Q}_{0,3}(\Pn,d)\lra\ov{Q}_{0,2}(\Pn,d)$$ 
is the forgetful morphism.
By the proof of the string relation \cite[Section~26.3]{MirSym},
\BE{cZstr_e}\dot\cZ_{n;\a;2}^{(\0,\1)}(\x,\hb,q)=\hb^{-1}\dot\cZ_{n;\a}(\x,\hb,q).\EE
We use this identity to establish the mirror formula for Hurwitz numbers in
Proposition~\ref{equiv0_prp2}.\\

\noindent
As stated in Section~\ref{intro_sec},
Theorem~\ref{equiv_thm} generalizes to products of projective spaces and
concavex sheaves~\e_ref{gensheaf_e}.
The relevant torus action is then the product of the actions on the components
described above.
If its weights are denoted by $\al_{i;j}$, with $i\!=\!1,\ldots,p$ and $j\!=\!1,\ldots,n_i$,
the analogues of the above mirror formulas relate power series
\begin{alignat}{1}
\label{gencY_e}
\ch\cY_{n_1,\ldots,n_p;\a}
&\in \Q[\al_{1;1},\ldots,\al_{p;n_p},\x_1,\ldots,\x_p]\big[\big[\hb^{-1},q_1,\ldots,q_p\big]\big],\\
\label{gencZ_e1}
\ch\cZ_{n_1,\ldots,n_p;\a}^{(s_1,\ldots,s_p)}
&\in H_{\T}^*\big(\P^{n_1-1}\!\times\!\ldots\!\times\!\P^{n_p-1}\big)
\big[\big[\hb^{-1},q_1,\ldots,q_p\big]\big], \\
\label{gencZ_e2}
\ch\cZ_{n_1,\ldots,n_p;\a}^*
&\in H_{\T}^*\big((\P^{n_1-1}\!\times\!\ldots\!\times\!\P^{n_p-1})^m\big)
\big[\big[\hb_1^{-1},\ldots,\hb_M^{-1},q_1,\ldots,q_p\big]\big],
\end{alignat}
with $\ch\cY$ and $\ch\cZ$ denoting $\cY$, $\dot\cY$, $\ddot\cY$,
$\cZ$, $\dot\cZ$, or $\ddot\cZ$ and $m\!=\!2,3$.
The coefficients of $q_1^{d_1}\!\ldots\!q_p^{d_p}$ in~\e_ref{gencZ_e1} and~\e_ref{gencZ_e2}
are defined by the same pushforwards as in~\e_ref{cZsdfn_e}, 
\e_ref{cZ2dfn_e}, \e_ref{cZ3Jdfn_e}, and~\e_ref{cZ2dfn_e2} with the degree~$d$
of the stable quotients replaced by $(d_1,\ldots,d_p)$ and $\x^s$ by $\x_1^{s_1}\!\ldots\!\x_p^{s_p}$.
The coefficients of $q_1^{d_1}\!\ldots\!q_p^{d_p}$ in~\e_ref{gencY_e}
are obtained from the coefficients in~\e_ref{cYdfn_e}  and~\e_ref{hatdfn_e}
by replacing $a_kd$ and $a_k\x$
by $a_{k;1}d_1\!+\!\ldots\!+\!a_{k;p}d_p$ and $a_{k;1}\x_1\!+\!\ldots\!+\!a_{k;p}\x_p$
in the numerator and taking the product of the denominators with
$(n,\x,d)\!=\!(n_i,\x_i,d_i)$ for each $s\!=\!1,\ldots,p$;
in the $i$-th factor, $\al_k$ is also replaced by~$\al_{i;k}$;
$$\x_1,\ldots,\x_p\in H_{\T}^*(\P^{n_1-1}\!\times\!\ldots\!\times\!\P^{n_p-1})$$
now correspond to the pullbacks of the equivariant hyperplane classes by the projection maps.
The statements of Theorem~\ref{equiv_thm}, \e_ref{cZ2pt_e2}, and~\e_ref{cZ3pt_e2} 
extend by replacing the symmetric polynomials by products 
of symmetric polynomials in the $p$~different sets of variables and
$\lr\a\x^{\ell(\a)}$ by the products and ratios of the terms 
$a_{k;1}\x_1\!+\!\ldots\!+\!a_{k;p}\x_p$;
our proofs extend directly to this situation.

\section{Equivariant twisted Hurwitz numbers}
\label{Hirwitz_sec}

\noindent
The fixed loci of the $\T$-action on $\ov{Q}_{0,m}(\Pn,d)$ involve moduli spaces of weighted
curves and certain vector bundles, which we describe in this section.
As a corollary of the proof of Theorem~\ref{equiv_thm}, we obtain closed formulas
for euler classes of these vector bundles in some cases.
These formulas, described in Propositions~\ref{equiv0_prp} and~\ref{equiv0_prp2}
below, are a key ingredient in computing the genus~1 stable quotients invariants.\\

\noindent
A \sf{stable $d$-tuple of~flecks on a  quasi-stable $m$-marked curve} is a tuple
\BE{torsiontuple_e} (\cC,y_1,\ldots,y_m;\hat{y}_1,\ldots,\hat{y}_d),\EE
where $\cC$ is a connected (at worst) nodal curve, $y_1,\ldots,y_m\!\in\!\cC^*$
are distinct smooth points, and $\hat{y}_1,\ldots,\hat{y}_d\in\cC^*\!-\!\{y_1,\ldots,y_m\}$,
such that the $\Q$-line bundle
$$\om_{\cC}\big(y_1\!+\!\ldots\!+\!y_m+\ep(\hat{y}_1\!+\!\ldots\!+\!\hat{y}_d)\big)
\lra \cC$$
is ample for all $\ep\!\in\!\Q^+$;
this again implies that $2g\!+\!m\!\ge\!2$.
An \sf{isomorphism}
$$\phi\!:(\cC,y_1,\ldots,y_m;\hat{y}_1,\ldots,\hat{y}_d)
\lra (\cC',y_1',\ldots,y_m';\hat{y}_1',\ldots,\hat{y}_d')$$
between curves with $m$ marked points and $d$ flecks is an isomorphism $\phi\!:\cC\!\lra\!\cC'$
such that
$$\phi(y_i)=y_i'~~~\forall~i=1,\ldots,m, \qquad
\phi(\hat{y}_j)=\hat{y}_j'~~~\forall~j=1,\ldots,d.$$
The automorphism group of any stable curve with $m$ marked points and $d$ flecks is finite.
For $g,m,d\!\in\!\Z^{\ge0}$, the moduli space $\ov\cM_{g,m|d}$
parameterizing the stable $d$-tuples of flecks as in~\e_ref{torsiontuple_e}
with $h^1(\cC,\cO_{\cC})\!=\!g$ 
is a nonsingular irreducible proper Deligne-Mumford stack;
see~\cite[Proposition~2.3]{GWvsSQ}.
If $m\!\ge\!m'\!\ge\!2$, let
\begin{equation*}\begin{split}
f_{m',m}\!:\ov\cM_{0,m|d}&\lra\ov\cM_{0,m'|d+m-m'}, \\
(\cC,y_1,\ldots,y_m;\hat{y}_1,\ldots,\hat{y}_d)&\lra
(\cC',y_1,\ldots,y_{m'};\hat{y}_1,\ldots,\hat{y}_d,y_{m'+1},\ldots,y_{m}),
\end{split}\end{equation*}
be the morphism converting the last $m\!-\!m'$ marked points 
into the last $m\!-\!m'$ flecks 
and contracting components of~$\cC$ if necessary.\\

\noindent
Any tuple as in~\e_ref{torsiontuple_e} induces a quasi-stable quotient
$$\cO_{\cC}\big(-\hat{y}_1-\ldots-\hat{y}_d\big)\subset
\cO_{\cC} \equiv \C^1\!\otimes\!\cO_{\cC}\,.$$
For any ordered partition $d\!=\!d_1\!+\!\ldots\!+\!d_p$ with $d_1,\ldots\!,d_p\!\in\!\Z^{\ge0}$,
this correspondence gives rise to a morphism
$$\ov\cM_{g,m|d}\lra
\ov{Q}_{g,m}\big(\P^0\!\times\!\ldots\!\times\!\P^0,(d_1,\ldots,d_p)\big).$$
In turn, this morphism induces an isomorphism
\BE{curvquot_e}\phi\!:\ov\cM_{g,m|d}\big/\bS_{d_1}\!\times\!\ldots\!\times\!\bS_{d_p}
\stackrel{\sim}{\lra} 
\ov{Q}_{g,m}\big(\P^0\!\times\!\ldots\!\times\!\P^0,(d_1,\ldots,d_p)\big),\EE
with the symmetric group $\bS_{d_1}$ acting on $\ov\cM_{g,m|d}$ by permuting
the points $\hat{y}_1,\ldots,\hat{y}_{d_1}$,
$\bS_{d_2}$ acting on $\ov\cM_{g,m|d}$ by permuting
the points $\hat{y}_{d_1+1},\ldots,\hat{y}_{d_1+d_2}$, etc.\\

\noindent
There is again  a universal curve
$$\pi\!: \cU\lra \ov\cM_{g,m|d}$$
with sections $\si_1,\ldots,\si_m$ and $\hat\si_1,\ldots\hat\si_d$.
Let 
$$\psi_i=-\pi_*(\si_i^2),\hat\psi_i=-\pi_*(\hat\si_i^2)\in H^2\big(\ov\cM_{g,m|d}\big)$$
be the first chern classes of the universal cotangent line bundles.
For $m\!\ge\!2$, $d',d\!\in\!\Z^+$ with $d'\!\le\!d$, and 
$\bfr\!\equiv\!(r_1,\ldots,r_{d'})\!\in\!(\Z^{\ge0})^{d'}$, 
let
$$\cS_{\bfr}=\cO\big(-\hat\si_1-\ldots-\hat\si_{d-d'}
-r_1\hat\si_{d-d'+1}-\ldots-r_{d'}\hat\si_d\big)
\lra \cU\lra \ov\cM_{0,m|d}\,.$$
If $\be\!\in\!H_{\T}^2$, denote~by
\BE{cSbe_e} \cS_{\bfr}^*(\be)\lra \cU\lra \ov\cM_{0,m|d}\EE
the sheaf $\cS_{\bfr}^*$ with the $\T$-action so that
$$\E\big(\cS_{\bfr}^*(\be)\big)=\be\!\times\!1+1\!\times\!e(\cS_{\bfr}^*)
\in H_{\T}^*(\cU)=
H_{\T}^*\otimes H^*(\cU).$$
Similarly to \e_ref{Vprdfn_e}, let
\BE{V0prdfn_e}\begin{split}
\dot\V_{\a;\bfr}'^{(d)}(\be)&=\bigoplus_{a_k>0}R^0\pi_*\big(\cS_{\bfr}^*(\be)^{a_k}(-\si_1)\big)
 \oplus \bigoplus_{a_k<0}R^1\pi_*\big(\cS_{\bfr}^*(\be)^{a_k}(-\si_1)\big)
 \lra \ov\cM_{0,m|d},\\
\ddot\V_{\a;\bfr}'^{(d)}(\be)&=\bigoplus_{a_k>0}R^0\pi_*\big(\cS_{\bfr}^*(\be)^{a_k}(-\si_2)\big)
 \oplus \bigoplus_{a_k<0}R^1\pi_*\big(\cS_{\bfr}^*(\be)^{a_k}(-\si_2)\big)
 \lra \ov\cM_{0,m|d}\,,
\end{split}\EE
where $\pi\!:\cU\!\lra\!\ov\cM_{0,m|d}$ is the projection as before;
these sheaves are locally free.
If $m'\!\in\!\Z^+$, $2\!\le\!m'\!\le\!m$, and $\bfr\!\in\!(\Z^{\ge0})^{m-m'}$, let
\BE{W0dfn_e}\dot\V_{\a;\bfr}^{(d)}(\be)=f_{m',m}^*\dot\V_{\a;\bfr}'^{(d)}(\be),~
\ddot\V_{\a;\bfr}^{(d)}(\be)=f_{m',m}^*\ddot\V_{\a;\bfr}'^{(d)}(\be)
\lra  \ov\cM_{0,m|d}.\EE
In the case $m'\!=\!m$, we will denote the bundles 
$\dot\V_{\a;\bfr}^{(d)}(\be)$ and $\ddot\V_{\a;\bfr}^{(d)}(\be)$ by
$\dot\V_{\a}^{(d)}(\be)$ and $\ddot\V_{\a}^{(d)}(\be)$, respectively.\\

\noindent
The equivariant euler classes of the bundles $\dot\V_{\a;\bfr}^{(d)}(\be)$
and $\ddot\V_{\a;\bfr}^{(d)}(\be)$ enter into the localization computations
in Sections~\ref{recpf_sec}-\ref{equiv0pf_sec}.
As a corollary of these computations, we obtain closed formulas for the euler classes
of these bundles in the case $m\!=\!3$; see Propositions~\ref{equiv0_prp} 
and~\ref{equiv0_prp2} below.
These formulas are a key ingredient in computing 
the  genus~0 three-point and genus~1 SQ-invariants.\\

\noindent
If $f\!\in\!\Q_{\al}[[q]]$ and $d\!\in\!\Z^{\ge0}$,
let $\coeff{f}_{q;d}\!\in\!\Q_{\al}$ denote the coefficient of~$q^d$ in~$f$.
If $f\!=\!f(z)$ is a rational function in $z$ and possibly some other
variables, for any $z_0\!\in\!\P^1\!\supset\!\C$ let
\BE{Rsdfn_e}\Rs{z=z_0}f(z) \equiv \frac{1}{2\pi\I}\oint f(z)\nd z,\EE
where the integral is taken over a positively oriented loop around $z\!=\!z_0$
with no other singular points of $f\nd z$,
denote the residue of the 1-form~$f\nd z$.
If $z_1,\ldots,z_k\!\in\!\P^1$ is any collection of points, let
\BE{Rssumdfn_e}\Rs{z=z_1,\ldots,z_k}f(z)
\equiv\sum_{i=1}^{i=k}\Rs{z=z_i}f(z)\EE
be the sum of the corresponding residues.\\

\noindent
For any variable $\y$ and $r\!\in\!\Z^{\ge0}$, let $\bfs_r(y)$ denote
the $r$-th elementary symmetric polynomial in $\{\y\!-\!\al_k\}$.
We define power series $L_{n;\a},\xi_{n;\a}\in\Q_{\al}[\x][[q]]$ by
\begin{alignat}{2}
\label{Ldfn_e}
L_{n;\a}&\in \x+q\Q_{\al}[\x][[q]], &\qquad 
\bfs_n\big(L_{n;\a}(\x,q)\big)-q\a^{\a}L_{n;\a}(\x,q)^{|\a|}&=\bfs_n(\x), \\
\xi_{n;\a}&\in q\Q_{\al}[\x][[q]],&\qquad  
\x+q\frac{\nd}{\nd q}\xi_{n;\a}(\x,q)&=L_{n;\a}(\x,q).\notag
\end{alignat}
By \cite[Remark~4.5]{g0ci}, the coefficients of the power series 
$$\ne^{-\xi_{n;\a}(\al_i,q)/\hb}\dot\cY_{n;\a}(\al_i,\hb,q)\in\Q_{\al}[\hb][[q]]$$
are regular at $\hb\!=\!0$.
Thus, there is an expansion
\BE{cYexp_e}\ne^{-\xi_{n;\a}(\al_i,q)/\hb}\dot\cY_{n;\a}(\al_i,\hb,q)
=\sum_{r=0}^{\i}\dot\Phi_{n;\a}^{(r)}(\al_i,q)\hb^r\EE
with $\dot\Phi_{n;\a}^{(0)}(\x,q)\!-\!1,\dot\Phi_{n;\a}^{(1)}(\x,q),\dot\Phi_{n;\a}^{(2)}(\x,q),
\ldots\in q\Q_{\al}[\x][[q]]$.
Furthermore, 
\BE{Phi0dfn_e}\dot\Phi_{n;\a}^{(0)}(\x,q)=
\left(\frac{\x\cdot\bfs_{n-1}(\x)}
{L_{n;\a}(\x,q)\,\bfs_{n-1}(L_{n;\a}(\x,q))-|\a|q\a^{\a}L_{n;\a}(\x,q)^{|\a|}}\right)^{\frac12}
\bigg(\frac{L_{n;\a}(\x,q)}{\x}\bigg)^{\frac{\ell(\a)+1}{2}}\,.\footnotemark\EE
\footnotetext{Only the case $\ell^-(\a)\!=\!0$ is explicitly
considered in~\cite{g0ci}, but the argument is the same in all cases.}

\begin{prp}\label{equiv0_prp}
If $l\!\in\!\Z^{\ge0}$, $n\!\in\!\Z^+$, and $\a\!\in\!(\Z^*)^l$, then
\begin{equation*}\begin{split}
&\sum_{d=0}^{\i}\frac{q^d}{d!}
\int_{\ov\cM_{0,3|d}}\frac{\E(\dot\V_{\a}^{(d)}(\al_i))}
{\prod\limits_{k\neq i}\!\!\E(\dot\V_1^{(d)}(\al_i\!-\!\al_k))\,
(\hb_1\!-\!\psi_1)(\hb_2\!-\!\psi_2)(\hb_3\!-\!\psi_3)}\\
&\hspace{2in}
=\frac{\ne^{\frac{\xi_{n;\a}(\al_i,q)}{\hb_1}+\frac{\xi_{n;\a}(\al_i,q)}{\hb_2}+
\frac{\xi_{n;\a}(\al_i,q)}{\hb_3}}}{\hb_1\hb_2\hb_3\,\dot\Phi_{n;\a}^{(0)}(\al_i,q)}
\in \Q_{\al}\big[\big[\hb_1^{-1},\hb_2^{-1},\hb_3^{-1},q\big]\big]
\end{split}\end{equation*}
for every $i\!=\!1,\ldots,n$.
\end{prp}

\begin{prp}\label{equiv0_prp2}
If $l\!\in\!\Z^{\ge0}$, $n\!\in\!\Z^+$, and $\a\!\in\!(\Z^*)^l$, then
\begin{equation*}\begin{split}
&\sum_{b=0}^{\i}\sum_{r=0}^{\i}\sum_{d=0}^{\i}\frac{q^d}{d!}
\int_{\ov\cM_{0,3|d}}\frac{\E(\dot\V_{\a;r}^{(d)}(\al_i))\psi_3^b}
{\prod\limits_{k\neq i}\!\!\E(\dot\V_1^{(d)}(\al_i\!-\!\al_k))\,
(\hb_1\!-\!\psi_1)(\hb_2\!-\!\psi_2)}
\Rs{\hb=0}\frac{(-1)^b}{\hb^{b+1}}\LR{\dot\cY_{n;\eset}(\al_i,\hb,q)}_{q;r}q^r\\
&\hspace{3in}
=\frac{\ne^{\frac{\xi_{n;\a}(\al_i,q)}{\hb_1}+\frac{\xi_{n;\a}(\al_i,q)}{\hb_2}}}{\hb_1\hb_2}
\in \Q_{\al}\big[\big[\hb_1^{-1},\hb_2^{-1},q\big]\big]
\end{split}\end{equation*}
for every $i\!=\!1,\ldots,n$.
\end{prp}

\section{Outline of the proof of Theorem~\ref{equiv_thm}}
\label{pfs_sec}

\noindent
The first identity in~\e_ref{equivGivental_e} is the subject of \cite[Theorem~3]{GWvsSQ}.
The proof of the remaining statements of Theorem~\ref{equiv_thm} follows 
the same principle as
the proof of \cite[Theorem~4]{bcov0_ci}; it is outlined below.
However, its adaptation to the present situation requires a number
of modifications.
In particular, the twisted stable quotients invariants are not known to satisfy 
the analogue of the string relation of Gromov-Witten theory
(in fact, by Proposition~\ref{Z3equiv_prp}, in general they do not).
This requires a direct proof of the key properties for 
the stable quotients analogue of double Givental's $J$-function
described in Lemmas~\ref{recgen_lmm} and~\ref{polgen_lmm} below;
in Gromov-Witten theory, these properties are deduced from the analogous
properties for three-point invariants, which simplifies the argument.
We thus describe the argument in detail.\\

\noindent
Let $\Q_{\al}\Lau{\hb}\equiv \Q_{\al}[[\hb^{-1}]]+\Q_{\al}[\hb]$
denote the $\Q_{\al}$-algebra of Laurent series in $\hb^{-1}$ (with finite principal part).
We will view the $\Q_{\al}$-algebra $\Q_{\al}(\hb)$ of rational functions in~$\hb$
with coefficients in~$\Q_{\al}$
as a subalgebra of $\Q_{\al}\Lau{\hb}$
via the embedding given by taking the Laurent series of rational functions at $\hb^{-1}\!=\!0$.
If
$$\F(\hb,q)=\sum_{d=0}^{\i}\sum_{r=-N_d}^{\i}\!\!\!\F^{(r)}(d)\hb^{-r}q^d
\in \Q_{\al}\Lau{\hb}\big[\big[q\big]\big]$$
for some $N_d\!\in\!\Z$ and $\F^{(r)}(d)\!\in\!\Q_{\al}$, we define
$$\F(\hb,q) \cong 
\sum_{d=0}^{\i}
\sum_{r=-N_d}^{p-1}\!\!\!\F^{(r)}(d)\hb^{-r}\quad(\mod \hb^{-p}),$$ 
i.e.~we drop $\hb^{-p}$ and higher powers of $\hb^{-1}$, 
instead of higher powers of~$\hb$.\\

\noindent
For $1\!\le\!i,j\!\leq\!n$ with $i\neq j$ and $d\in\Z^{+}$, let
\BE{Cdfn}\begin{split}
\dot\fC_i^j(d)&\equiv  \frac{\prod\limits_{a_k>0}    
  \prod\limits_{r=1}^{a_kd}\left(a_k\al_i+r\,\frac{\al_j-\al_i}{d}\right)\,
 \prod\limits_{a_k<0} 
  \prod\limits_{r=0}^{-a_kd-1}\left(a_k\al_i-r\,\frac{\al_j-\al_i}{d}\right)}
{d\underset{(r,k)\neq(d,j)}{\prod\limits_{r=1}^d\prod\limits_{k=1}^n}
            \left(\al_i-\al_k+r\,\frac{\al_j-\al_i}{d}\right)}\in\Q_{\al} \,,\\
\ddot\fC_i^j(d)&\equiv  \frac{\prod\limits_{a_k>0}    
  \prod\limits_{r=0}^{a_kd-1}\left(a_k\al_i+r\,\frac{\al_j-\al_i}{d}\right)\,
 \prod\limits_{a_k<0} 
  \prod\limits_{r=1}^{-a_kd}\left(a_k\al_i-r\,\frac{\al_j-\al_i}{d}\right)}
{d\underset{(r,k)\neq(d,j)}{\prod\limits_{r=1}^d\prod\limits_{k=1}^n}
            \left(\al_i-\al_k+r\,\frac{\al_j-\al_i}{d}\right)}\in\Q_{\al} \,. 
\end{split}\EE
We will follow the five steps in \cite[Section~1.3]{bcov0} to verify 
\e_ref{main_e2}, the second statement in~\e_ref{equivGivental_e}, and~\e_ref{cZs_e}:

\begin{enumerate}[label=(M\arabic*),leftmargin=*]

\item\label{transM_item} 
if $\F,\F'\in H_{\T}^*(\Pn)\Lau{\hb}\big[\big[q\big]\big]$, 
$$\F(\x\!=\!\al_i,\hb,q)\in \Q_{\al}(\hb)\big[\big[q\big]\big]\subset 
\Q_{\al}\Lau{\hb}\big[\big[q\big]\big] \qquad\forall\,i\!=\!1,2,\ldots,n,$$
$\F'$ is recursive in the sense of Definition~\ref{recur_dfn}, and
$\F$ and $\F'$ satisfy a mutual polynomiality condition (MPC) of Definition~\ref{MPC_dfn}, 
then the transforms of $\F'$ of Lemma~\ref{Phistr_lmm4} are also recursive
and satisfy the same MPC 
with respect to~$\F$;

\item\label{uniqM_item} 
if  $\F,\F'\in H_{\T}^*(\Pn)\Lau{\hb}\big[\big[q\big]\big]$,
$$\F(\x\!=\!\al_i,\hb,q)\in \Q_{\al}^*+q\cdot \Q_{\al}(\hb)\big[\big[q\big]\big]
\subset \Q_{\al}\Lau{\hb}\big[\big[q\big]\big]
 \qquad\forall\,i\!=\!1,2,\ldots,n,$$
$\F'$ is recursive in the sense of Definition~\ref{recur_dfn}, and
$\F$ and $\F'$ satisfy a fixed MPC, 
then $\F'$ is determined by its  ``mod $\hb^{-1}$ part'';

\item\label{recM_item} the two sides of the second identity in~\e_ref{equivGivental_e}
and the $\ddot\cZ$ case in~\e_ref{cZs_e} are $\ddot\fC$-recursive in the sense of Definition~\ref{recur_dfn} 
with $\ddot\fC$ as in~\e_ref{Cdfn}, 
while the two sides of the $\dot\cZ$ case in~\e_ref{cZs_e} are $\dot\fC$-recursive 
in the sense of Definition~\ref{recur_dfn} 
with $\dot\fC$ as in~\e_ref{Cdfn};

\item\label{polM_item}
the two sides of each of the equations in \e_ref{equivGivental_e} and \e_ref{cZs_e}
satisfy the same MPC (dependent on the equation) with respect to  
$\dot\cY_{n;\a}(\x,\hb,q)$;

\item\label{equalmodM_item} 
the two sides of each of the four equations in \e_ref{equivGivental_e} and \e_ref{cZs_e}, 
viewed as elements of $H_{\T}^*(\Pn)\Lau{\hb}\big[\big[q\big]\big]$,  agree mod~$\hb^{-1}$. 

\end{enumerate}

\noindent
The first two claims above, \ref{transM_item} and \ref{uniqM_item}, sum up
Lemma~\ref{Phistr_lmm4} and Proposition~\ref{uniqueness_prp}, respectively.
By Lemmas~\ref{recgen_lmm} and~\ref{polgen_lmm},
the stable quotients generating functions $\dot\cZ_{n;\a}^{(s)}$ 
and~$\ddot\cZ_{n;\a}^{(s)}$ 
are $\dot\fC$-recursive and $\ddot\fC$-recursive and satisfy MPCs with respect 
to $\dot\cZ_{n;\a}(\x,\hb,q)$.
Along with the first identity in~\e_ref{equivGivental_e}, 
the latter implies that they satisfy MPCs with respect to~$\dot\cY_{n;\a}$.
It is immediate from~\e_ref{cZsdfn_e} that 
\BE{Z2wmod_e}
\dot\cZ_{n;\a}^{(s)}(\x,\hb,q),
\ddot\cZ_{n;\a}^{(s)}(\x,\hb,q) \cong \x^s \quad(\mod\hb^{-1}\big)
\qquad\forall~s\in\Z^{\ge0}\,.\EE
By the proof of the first identity in \e_ref{equivGivental_e}, as well as of its
Gromov-Witten analogue, the power series~$\dot\cY_{n;\a}$ is $\dot\fC$-recursive
and satisfies the same MPC with respect to~$\dot\cY_{n;\a}$ as~$\dot\cZ_{n;\a}^{(s)}$;
see \cite[Lemma~5.4]{GWvsSQ}.
A nearly identical argument shows that the power series~$\ddot\cY_{n;\a}$ is $\ddot\fC$-recursive
and satisfies the same MPC with respect to~$\dot\cY_{n;\a}$ as~$\ddot\cZ_{n;\a}^{(s)}$;
see \cite[Section~4.3]{bcov0_ci} for the $\ell^-(\a)\!=\!0$ case.
Since 
$$\ddot\cY_{n;\a}(\x,\hb,q) \cong 1 \quad(\mod\hb^{-1}\big),$$
this establishes the second identity in~\e_ref{equivGivental_e}.
Along with~\e_ref{equivGivental_e}, 
the admissibility of transforms~\ref{deriv_ch} and~\ref{mult_ch} in Lemma~\ref{Phistr_lmm4} 
implies that both sides of the $\dot\cZ$ equation in~\e_ref{cZs_e} are $\dot\fC$-recursive 
and satisfy the same MPC with respect to~$\dot\cY_{n;\a}$, 
no matter what the coefficients $\ctC_{s,s'}^{(r)}$ are.
Similarly, both sides of the $\ddot\cZ$ equation in~\e_ref{cZs_e} are $\ddot\fC$-recursive 
and satisfy the same MPC with respect to~$\dot\cY_{n;\a}$.
By~\e_ref{DscZdfn_e}, \e_ref{equivGivental_e},  \e_ref{cYdfn_e},
\e_ref{tiCrec_e}, \e_ref{Crec_e}, \e_ref{DscYdfn_e}, and~\e_ref{hatdfn_e},
\BE{cZscong_e}\begin{split}
\sum_{r=0}^s\sum_{s'=0}^{s-r}
\ctC_{s-\ell^-(\a),s'-\ell^-(\a)}^{(r)}(q)\,\hb^{s-r-s'}\D^{s'}\dot\cZ_{n;\a}(\x,\hb,q)
&\cong  \x^s \quad(\mod\hb^{-1}\big),\\
\sum_{r=0}^s\sum_{s'=0}^{s-r}
\ctC_{s-\ell^+(\a),s'-\ell^+(\a)}^{(r)}(q)\,\hb^{s-r-s'}\D^{s'}\ddot\cZ_{n;\a}(\x,\hb,q)
&\cong \x^s \quad(\mod\hb^{-1}\big).\,\footnotemark
\end{split}\EE
Thus, \e_ref{cZs_e} follows from~\ref{uniqM_item}.
\footnotetext{LHS of \e_ref{tiCrec_e} with $s$ replaced by $s\!-\!\ell^-(\a)$
is the coefficient of $\hb^s\x^{-r}(\x/\hb)^{s'+\ell^-(\a)}$ 
in the first identity in~\e_ref{cZscong_e} if $s\!\ge\!\ell^-(\a)$;
LHS of \e_ref{tiCrec_e} with $s$ replaced by $s\!-\!\ell^+(\a)$
is the coefficient of $\hb^s\x^{-r}(\x/\hb)^{s'+\ell^+(\a)}$ 
in the second identity in~\e_ref{cZscong_e} if $s\!\ge\!\ell^+(\a)$.}\\

\noindent
The proof of \e_ref{main_e2} follows the same principle, which we apply to a multiple
of~\e_ref{main_e2}.
For each $i\!=\!1,2,\ldots,n$, let
\BE{phidfn_e} \phi_i\equiv\prod_{k\neq i}(\x\!-\!\al_k) \in H_{\T}^*(\Pn).\EE
By the Localization Theorem~\cite{ABo},
$\phi_i$ is the equivariant Poincar\'{e} dual of the fixed point $P_i\!\in\!\Pn$;
see \cite[Section~3.1]{bcov0}.
Since $\x|_{P_i}\!=\!\al_i$,
\begin{equation}\label{pushch_e}\begin{split}
\dot\cZ_{n;\a}(\al_i,\al_j,\hb_1,\hb_2,q)
&=\int_{P_i\times P_j}\!\!\!\dot\cZ_{n;\a}(\x_1,\x_2,\hb_1,\hb_2,q)
=\int_{\Pn\times\Pn}\!\!\dot\cZ_{n;\a}(\x_1,\x_2,\hb_1,\hb_2,q)\phi_i\!\times\!\phi_j\\
&=\frac{1}{\hb_1\!+\!\hb_2}\prod_{k\neq i}(\al_j-\al_k)
+\sum_{d=1}^{\i}q^d\!\!
\int_{\ov{Q}_{0,2}(\Pn,d)}\!
\frac{\E(\dot\V_{n;\a}^{(d)})\,\ev_1^*\phi_i\,\ev_2^*\phi_j}{(\hb_1\!-\!\psi_1)(\hb_2\!-\!\psi_2)}\,;
\end{split}\end{equation}
the last equality holds 
by the defining property of the cohomology push-forward \cite[(3.11)]{bcov0}.
By Lemmas~\ref{recgen_lmm} and~\ref{polgen_lmm},
$\dot\cZ_{n;\a}(\x_1,\x_2,\hb_1,\hb_2,q)$ is $\dot\fC$-recursive and satisfies the same MPC as
$\dot\cZ_{n;\a}$ 
with respect to $\dot\cZ_{n;\a}(\x,\hb,q)$ for $(\x,\hb)\!=\!(\x_1,\hb_1)$ and 
$\x_2\!=\!\al_j$ fixed.\footnote{In other words, the coefficient of every power of $\hb_2^{-1}$
in $\dot\cZ_{n;\a}(\x,\al_j,\hb,\hb_2,q)$ is $\dot\fC$-recursive and satisfies the same MPC as
$\dot\cZ_{n;\a}(\x,\hb,q)$ with respect to $\dot\cZ_{n;\a}(\x,\hb,q)$.}
It is also $\ddot\fC$-recursive and satisfies the same MPC as $\ddot\cZ_{n;\a}$
with respect to $\dot\cZ_{n;\a}(\x,\hb,q)$ for $(\x,\hb)\!=\!(\x_2,\hb_2)$ and 
$\x_1\!=\!\al_i$ fixed.
By~\ref{transM_item} and~\ref{uniqM_item},  it is thus sufficient to compare 
\BE{Zcomp_e} 
(\hb_1\!+\!\hb_2)\dot\cZ_{n;\a}(\x_1,\x_2,\hb_1,\hb_2,q)
\quad\hbox{and}\quad
\sum_{\begin{subarray}{c}s_1,s_2,r\ge0\\ s_1+s_2+r=n-1 \end{subarray}}
\!\!\!\!\!\!\!\!\!(-1)^r\bfs_r
\dot\cZ_{n;\a}^{(s_1)}(\x_1,\hb_1,q)\ddot\cZ_{n;\a}^{(s_2)}(\x_2,\hb_2,q)\EE
for all $\x_1\!=\!\al_i$ and $\x_2\!=\!\al_j$ with $i,j\!=\!1,2,\ldots,n$
modulo~$\hb_1^{-1}$:
\begin{alignat*}{1}
&(\hb_1\!+\!\hb_2)\dot\cZ_{n;\a}(\al_i,\al_j,\hb_1,\hb_2,q) \cong   
\sum_{\begin{subarray}{c}s_1,s_2,r\ge0\\ s_1+s_2+r=n-1 \end{subarray}}
\!\!\!\!\!\!\!\!(-1)^r\bfs_r\al_i^{s_1}\al_j^{s_2}
+\sum_{d=1}^{\i}q^d\int_{\ov{Q}_{0,2}(\Pn,d)}\!\!\!
\frac{\E(\dot\V_{n;\a}^{(d)})\ev_1^*\phi_i\ev_2^*\phi_j}{\hb_2\!-\!\psi_2};\\
&\sum_{\begin{subarray}{c}s_1,s_2,r\ge0\\ s_1+s_2+r=n-1 \end{subarray}}
\!\!\!\!\!\!\!\!
(-1)^r\bfs_r \dot\cZ_{n;\a}^{(s_1)}(\al_i,\hb_1,q)\ddot\cZ_{n;\a}^{(s_2)}(\al_j,\hb_2,q)\cong
\sum_{\begin{subarray}{c}s_1,s_2,r\ge0\\ s_1+s_2+r=n-1\end{subarray}}\!\!\!\!\!\!\!
(-1)^r\bfs_r\al_i^{s_1}\ddot\cZ_{n;\a}^{(s_2)}(\al_j,\hb_2,q).
\end{alignat*}
In order to see that the two right-hand side power series are the same,
it is sufficient to compare them modulo~$\hb_2^{-1}$:
\begin{alignat*}{2}
&\sum_{\begin{subarray}{c}s_1,s_2,r\ge0\\ s_1+s_2+r=n-1\end{subarray}}
\!\!\!\!\!\!\!\!\! (-1)^r\bfs_r \al_i^{s_1}\al_j^{s_2}
+\sum_{d=1}^{\i}q^d\int_{\ov{Q}_{0,2}(\Pn,d)}\!\!\!
\frac{\E(\dot\V_{n;\a}^{(d)})\ev_1^*\phi_i\ev_2^*\phi_j}{\hb_2\!-\!\psi_2}\cong
\sum_{\begin{subarray}{c}s_1,s_2,r\ge0\\ s_1+s_2+r=n-1\end{subarray}}
   \!\!\!\!\!\!\!\!(-1)^r\bfs_r \al_i^{s_1}\al_j^{s_2};\\
&\sum_{\begin{subarray}{c} s_1,s_2,r\ge0\\ s_1+s_2+r=n-1\end{subarray}}
\!\!\!\!\!\!
(-1)^r\bfs_r\al_i^{s_1}\ddot\cZ_{n;\a}^{(s_2)}(\al_j,\hb_2,q)\cong
\sum_{\begin{subarray}{c}s_1,s_2,r\ge0\\ s_1+s_2+r=n-1\end{subarray}}
\!\!\!\!\!\!(-1)^r\bfs_r\al_i^{s_1}\al_j^{s_2}.
\end{alignat*}
From this we conclude that the two expressions in~\e_ref{Zcomp_e} are the same;
this proves~\e_ref{main_e2}.\\

\noindent
By Proposition~\ref{uniqueness_prp} and Lemmas~\ref{recgen_lmm} and~\ref{polgen_lmm},
the stable quotients analogue of triple Givental's $J$-function is determined
by the primary three-point SQ-invariants.
Since all such invariants are related to the corresponding GW-invariants by 
\cite[Theorem 1.2.2 and Corollaries~1.4.1,1.4.2]{CK},
a version of~\e_ref{Z2cJpt_e} can be proved by comparing it to its GW-analogue
provided by \cite[Theorem~B]{g0ci}.
We instead prove~\e_ref{Z2cJpt_e} directly in Section~\ref{3ptpf_sec}
by reducing the computation to the two-point formulas of Theorem~\ref{equiv_thm}
and the mirror formula for Hurwitz numbers in Propositions~\ref{equiv0_prp}.
In the process, we obtain a precise description of the equivariant structure
coefficients appearing in~\e_ref{Z2cJpt_e}, which is not done in~\cite{g0ci}.

\section{Recursivity, polynomiality,  and admissible transforms} 
\label{poliC_subs}
 
This section describes the algebraic observations used in the proof of 
Theorem~\ref{equiv_thm}.
It is based on \cite[Sections 2.1, 2.2]{bcov0} and \cite[Section~4.1]{bcov0_ci}.
Let
$$[n]=\{1,2,\ldots,n\}.$$

\begin{dfn}\label{recur_dfn}
Let $C\equiv(C_i^j(d))_{d,i,j\in\Z^{+}}$ be any collection of elements of~$\Q_{\al}$.
A~power series $\F\!\in\!H_{\T}^*(\Pn)\Lau{\hb}[[q]]$ is \sf{$C$-recursive} if
the following holds:
if $d^*\!\in\!\Z^{\ge0}$ is such~that
$$\bigcoeff{\F(\x\!=\!\al_i,\hb,q)}_{q;d^*-d}\in \Q_{\al}(\hb)
\subset \Q_{\al}\Lau{\hb}
\qquad\forall\,d\!\in\![d^*],\,i\!\in\![n],$$
and $\bigcoeff{\F(\al_i,\hb,q)}_{q;d}$ is regular at
$\hb\!=\!(\al_i\!-\!\al_j)/d$ for all $d\!<\!d^*$ and $i\!\neq\!j$, then
\BE{recurdfn_e}
\bigcoeff{\F(\al_i,\hb,q)}_{q;d^*}-
\sum_{d=1}^{d^*}\sum_{j\neq i}\frac{C_i^j(d)}{\hb-\frac{\al_j-\al_i}{d}}
\bigcoeff{\F(\al_j,z,q)}_{q;d^*-d}\big|_{z=\frac{\al_j-\al_i}{d}}
\in \Q_{\al}[\hb,\hb^{-1}]\subset \Q_{\al}\Lau{\hb}.\EE
\end{dfn}

\noindent
Thus, if $\F\!\in\!H_{\T}^*(\Pn)\Lau{\hb}[[q]]$ is $C$-recursive, for any collection~$C$, 
then
$$\F(\x\!=\!\al_i,\hb,q)\in \Q_{\al}(\hb)\big[\big[q\big]\big]
\subset \Q_{\al}\Lau{\hb}[[q]]
\qquad\forall\,i\!\in\![n],$$
as can be seen by induction on $d$, and
\BE{recurdfn_e2} \F(\al_i,\hb,q)=\sum_{d=0}^{\i}\sum_{r=-N_d}^{N_d}
\F_i^r(d)\hb^rq^d+
\sum_{d=1}^{\i}\sum_{j\neq i}\frac{C_i^j(d)q^d}{\hb-\frac{\al_j-\al_i}{d}}
\F(\al_j,(\al_j\!-\!\al_i)/d,q) \qquad\forall~i\!\in\![n],\EE
for some $\F_i^r(d)\!\in\!\Q_{\al}$.
The nominal issue with defining $C$-recursivity by~\e_ref{recurdfn_e2}, as is normally done,
is that a priori the evaluation
of $\F(\al_j,\hb,q)$ at $\hb\!=\!(\al_j\!-\!\al_i)/d$ need not be well-defined,
since $\F(\al_j,\hb,q)$ is a power series with coefficients in $\Q_{\al}\Lau{\hb^{-1}}$;
a priori they  may not converge anywhere.
However, taking the coefficient of each power of~$q$ in \e_ref{recurdfn_e2}
shows by induction on the degree~$d$
that this evaluation does make sense;
this is the substance of Definition~\ref{recur_dfn}.

\begin{dfn}\label{MPC_dfn}
Let $\eta\!\in\!\Q_{\al}(\x)$ be such that $\eta(\x\!=\!\al_i)\!\in\!\Q_{\al}$ 
is well-defined and nonzero for every $i\!\in\![n]$.
For any $\F\!\equiv\!\F(\x,\hb,q),\F'\!\equiv\!\F'(\x,\hb,q)\in H_{\T}^*(\Pn)\Lau{\hb}[[q]]$, 
let
\BE{PhiZdfn_e}
\Phi_{\F,\F'}^{\eta}(\hb,z,q)\equiv \sum_{i=1}^n
\frac{\eta(\al_i)\ne^{\al_iz}}{\prod\limits_{k\neq i}(\al_i\!-\!\al_k)}
\F\big(\al_i,\hb,q\ne^{\hb z}\big)\F(\al_i,-\hb,q) \in \Q_{\al}\Lau{\hb}[[z,q]].\EE
If $\F,\F'\in H_{\T}^*(\Pn)\Lau{\hb}[[q]]$, the pair $(\F,\F')$ 
\sf{satisfies the $\eta$ mutual polynomiality condition ($\eta$-MPC)} if
$\Phi_{\F,\F'}^{\eta}\in \Q_{\al}[\hb][[z,q]]$.
\end{dfn}

\noindent
If $\F,\F'\!\in\!H_{\T}^*(\Pn)\Lau{\hb}[[q]]$ and
\BE{YZrat_e}
\F(\x\!=\!\al_i,\hb,q), \F'(\x\!=\!\al_i,\hb,q)\in 
\Q_{\al}(\hb)\big[\big[q\big]\big] \qquad\forall\,i\!\in\![n],\EE
then the pair $(\F,\F')$ satisfies the $\eta$-MPC if and only if 
the pair $(\F',\F)$ does; 
see \cite[Lemma~2.2]{bcov0} for the $\eta\!=\!1$, 
$\ell^+(\a)\!=\!1$, $\ell^-(\a)\!=\!0$ case
(the proof readily carries over to the general case).
Thus, if \e_ref{YZrat_e} holds, the statement that $\F$ and $\F'$ satisfy the MPC
is unambiguous.

\begin{prp}\label{uniqueness_prp}
Let $\eta\!\in\!\Q_{\al}(\x)$ be such that $\eta(\x\!=\!\al_i)\!\in\!\Q_{\al}$ is well-defined 
and nonzero for every $i\!\in\![n]$.
If $\F,\F'\in H_{\T}^*(\Pn)\Lau{\hb}[[q]]$,
$$\F(\x\!=\!\al_i,\hb,q)\in \Q_{\al}^*+q\cdot \Q_{\al}(\hb)\big[\big[q\big]\big]
\subset \Q_{\al}\Lau{\hb}\big[\big[q\big]\big]
 \qquad\forall\,i\!\in\![n],$$
$\F'$ is recursive, and $\F$ and $\F'$ satisfy the $\eta$-MPC, 
then  $\F'\cong 0~(\mod \hb^{-1})$ if and only if $\F'=0$.
\end{prp}

\noindent
This is essentially \cite[Proposition~2.1]{bcov0}, with the assumptions corrected in
\cite[Footnote~3]{bcov0_ci}.
The proof in~\cite{bcov0}, which treats the $\eta\!=\!1$ case, readily extends to
the general case; see also the paragraph following \cite[Proposition~4.3]{bcov0_ci}.

\begin{lmm}\label{Phistr_lmm4}
Let $C\equiv(C_i^j(d))_{d,i,j\in\Z^{+}}$ be any collection of elements of~$\Q_{\al}$ and
$\eta\!\in\!\Q_{\al}(\x)$ be such that $\eta(\x\!=\!\al_i)\!\in\!\Q_{\al}$ is well-defined 
and nonzero for every $i\!\in\![n]$.
If  $\F,\F'\!\in\!H_{\T}^*(\Pn)\Lau{\hb}[[q]]$,
$$\F(\x\!=\!\al_i,\hb,q)\in \Q_{\al}(\hb)\big[\big[q\big]\big]
\subset \Q_{\al}\Lau{\hb}\big[\big[q\big]\big] \qquad\forall\,i\!\in\![n],$$
$\F'$ is $C$-recursive (and satisfies the $\eta$-MPC with respect to~$\F$), then
\begin{enumerate}[label=(\roman*)]

\item\label{deriv_ch}  
$\left\{\x\!+\!\hb\, q\frac{\nd}{\nd q}\right\}\F'$
is $C$-recursive (and satisfies the $\eta$-MPC with respect to~$\F$);

\item\label{mult_ch} if $f\!\in\!\Q_{\al}[\hb][[q]]$, then 
$f\F'$ is $C$-recursive (and satisfies the $\eta$-MPC with respect to~$\F$).

\end{enumerate}
\end{lmm}

\noindent
This lemma is essentially contained in \cite[Lemma 2.3]{bcov0}.
The proof in~\cite{bcov0}, which treats the $\eta\!=\!1$ case, readily extends to
the general case; see also the paragraph following \cite[Lemma~4.4]{bcov0_ci}.\\

\noindent
The next two sections establish Lemmas~\ref{recgen_lmm} and~\ref{polgen_lmm} below, 
the $m\!=\!2$ cases of which complete the proofs of 
\e_ref{main_e2}, the second statement in~\e_ref{equivGivental_e}, and~\e_ref{cZs_e}.
The $m\!=\!3$ cases of these lemmas are used in the proof of 
Proposition~\ref{Z3equiv_prp} and~\ref{equiv0_prp} in Section~\ref{equiv0pf_sec}.
If $m\!\ge\!m'\!\ge\!2$, let 
\BE{fmmdfn_e}f_{m',m}\!: \ov{Q}_{0,m}(\Pn,d)\lra \ov{Q}_{0,m'}(\Pn,d)\EE
denote the forgetful morphism dropping the last $m\!-\!m'$ points;
this morphism is defined if $m'\!>\!2$ or $d\!>\!0$.
With the bundles
$$\dot\V_{n;\a}^{(d)},\ddot\V_{n;\a}^{(d)}\lra\ov{Q}_{0,m'}(\Pn,d)$$
defined by~\e_ref{Vprdfn_e}, let 
\BE{Wdfn_e}
\dot\V_{n;\a;m'}^{(d)}=f_{m',m}^*\dot\V_{n;\a}^{(d)},\,
\ddot\V_{n;\a;m'}^{(d)}=f_{m',m}^*\ddot\V_{n;\a}^{(d)}\lra\ov{Q}_{0,m}(\Pn,d)\,.\EE
For $\b\!\equiv\!(b_2,\ldots,b_m)\!\in\!(\Z^{\ge0})^{m-1}$ and 
$\vp\!\equiv\!(\vp_2,\ldots,\vp_m)\!\in\!H_{\T}^*(\Pn)^{m-1}$, let
\BE{cZgendfn_e}\begin{split}
\dot\cZ_{n;\a;m'}^{(\b,\vp)}(\x,\hb,q) &\equiv
\sum_{d=0}^{\i}q^d\ev_{1*}\!\!\left[\frac{\E(\dot\V_{n;\a;m'}^{(d)})}{\hb\!-\!\psi_1}
\prod_{j=2}^{j=m}\!(\psi_j^{b_j}\ev_j^*\vp_j)\right]
\in H_{\T}^*(\Pn)[\hb^{-1}]\big[\big[q\big]\big],\\
\ddot\cZ_{n;\a;m'}^{(\b,\vp)}(\x,\hb,q) &\equiv
\sum_{d=0}^{\i}q^d\ev_{1*}\!\!\left[\frac{\E(\ddot\V_{n;\a;m'}^{(d)})}{\hb\!-\!\psi_1}
\prod_{j=2}^{j=m}\!(\psi_j^{b_j}\ev_j^*\vp_j)\right]
\in H_{\T}^*(\Pn)[\hb^{-1}]\big[\big[q\big]\big],
\end{split}\EE
where $\ev_j:\ov{Q}_{0,m}(\Pn,d)\lra\Pn$ is the evaluation map at the $j$-th marked point
and the degree~0 terms in the $m'\!=\!2$ case are defined~by
\begin{alignat*}{2}
\E(\dot\V_{n;\a;2}^{(0)}),\E(\ddot\V_{n;\a;2}^{(0)})&=1 &\qquad&\hbox{if}~~m\ge3, \\
\ev_{1*}\!\!\left[\frac{\E(\dot\V_{n;\a;2}^{(0)})}{\hb\!-\!\psi_1}(\psi_2^{b_2}\ev_2^*\vp_2)\right],
\ev_{1*}\!\!\left[\frac{\E(\ddot\V_{n;\a;2}^{(0)})}{\hb\!-\!\psi_1}
(\psi_2^{b_2}\ev_2^*\vp_2)\right] &=(-\hb)^{b_2}\vp_2 &\qquad&\hbox{if}~~m=2.
\end{alignat*}

\begin{lmm}\label{recgen_lmm} 
Let $l\!\in\!\Z^{\ge0}$, $m,m',n\!\in\!\Z^+$ with $m\!\ge\!m'\!\ge\!2$, and $\a\!\in\!(\Z^*)^l$.
For all $\b\!\in\!(\Z^{\ge0})^{m-1}$ and $\vp\!\in\!H_{\T}^*(\Pn)^{m-1}$,
the power series 
$\dot\cZ_{n;\a;m'}^{(\b,\vp)}$ and $\ddot\cZ_{n;\a;m'}^{(\b,\vp)}$ 
defined by~\e_ref{cZgendfn_e} are $\dot\fC$
and $\ddot\fC$-recursive, respectively.
\end{lmm}

\begin{lmm}\label{polgen_lmm} 
Let $l\!\in\!\Z^{\ge0}$, $m,m',n\!\in\!\Z^+$ with $m\!\ge\!m'\!\ge\!2$, 
$\a\!\in\!(\Z^*)^l$, 
$$\dot\eta(\x)=\lr\a\x^{\ell(\a)}\,,\qquad \ddot\eta(\x)=1\,.$$
For all $\b\!\in\!(\Z^{\ge0})^{m-1}$ and $\vp\!\in\!H_{\T}^*(\Pn)^{m-1}$,
the power series 
$$\hb^{m-2}\dot\cZ_{n;\a;m'}^{(\b,\vp)}(\x,\hb,q) \qquad\hbox{and}\qquad 
\hb^{m-2}\ddot\cZ_{n;\a;m'}^{(\b,\vp)}(\x,\hb,q)$$
satisfy the $\dot\eta$ and $\ddot\eta$-MPC, respectively,
with respect to  the power series $\dot\cZ_{n;\a}(\x,\hb,q)$ defined by~\e_ref{cZdfn_e}.
\end{lmm}

\noindent
By Lemma~\ref{recgen_lmm}, the power series $\dot\cZ_{n;\a}^{(s)}$ and $\ddot\cZ_{n;\a}^{(s)}$ 
defined 
by~\e_ref{cZsdfn_e} are $\dot\fC$ and $\ddot\fC$-recursive, respectively.
Furthermore, the power series $\dot\cZ_{n;\a}$ defined by~\e_ref{Z23cfull_e}
is $\dot\fC$-recursive for $(\x,\hb)\!=\!(\x_1,\hb_1)$ and $\x_2\!=\!\al_j$ fixed
and is $\ddot\fC$-recursive for $(\x,\hb)\!=\!(\x_2,\hb_2)$ and $\x_1\!=\!\al_j$ fixed.
By Lemma~\ref{polgen_lmm}, $\dot\cZ_{n;\a}^{(s)}$ and $\ddot\cZ_{n;\a}^{(s)}$ 
satisfy the $\dot\eta$ and $\ddot\eta$-MPC, respectively,
with respect to  the power series $\dot\cZ_{n;\a}(\x,\hb,q)$ defined by~\e_ref{cZdfn_e}.
Furthermore, the power series $\dot\cZ_{n;\a}$ defined by~\e_ref{Z23cfull_e}
satisfies the $\dot\eta$-MPC with respect to $\dot\cZ_{n;\a}(\x,\hb,q)$
for $(\x,\hb)\!=\!(\x_1,\hb_1)$ and $\x_2\!=\!\al_j$ fixed
and the $\ddot\eta$-MPC with respect to $\dot\cZ_{n;\a}(\x,\hb,q)$
for $(\x,\hb)\!=\!(\x_2,\hb_2)$ and $\x_1\!=\!\al_j$ fixed.\\

\noindent
In the case of  products of projective spaces and concavex sheaves~\e_ref{gensheaf_e},
Definition~\ref{recur_dfn} becomes inductive on the total degree
$d_1\!+\!\ldots\!+\!d_p$ of $q_1^{d_1}\!\ldots\!q_p^{d_p}$.
The power series~$\F$ is evaluated at $(\x_1,\ldots,\x_p)\!=\!(\al_{1;i_1},\ldots,\al_{p;i_p})$
for the purposes of the $C$-recursivity condition~\e_ref{recurdfn_e}
and~\e_ref{recurdfn_e2}.
The relevant structure coefficients, extending~\e_ref{Cdfn}, are given~by
\begin{equation*}\begin{split}
\dot\fC_{i_1\ldots i_p}^j(s;d)&\equiv  \frac{
\prod\limits_{a_{k;1}\ge0}
\prod\limits_{r=1}^{a_{k;s}d}\!\!\left(
 \sum\limits_{t=1}^p a_{k;t}\al_{t;i_t}+r\,\frac{\al_{s;j}-\al_{s;i_s}}{d}\right)
\prod\limits_{a_{k;1}<0}\!\!
\prod\limits_{r=0}^{-a_{k;s}d-1}\!\!\left(\sum\limits_{t=1}^p a_{k;t}\al_{t;i_t}
-r\,\frac{\al_{s;j}-\al_{s;i_s}}{d}\right)}
{d\underset{(r,k)\neq(d,j)}{\prod\limits_{r=1}^d\prod\limits_{k=1}^{n_s}}
            \left(\al_{s;i_s}-\al_{s;k}+r\,\frac{\al_{s;j}-\al_{s;i_s}}{d}\right)},\\
\ddot\fC_{i_1\ldots i_p}^j(s;d)&\equiv  \frac{
\prod\limits_{a_{k;1}\ge0}\prod\limits_{r=0}^{a_{k;s}d-1}\!\!\left(
 \sum\limits_{t=1}^p a_{k;t}\al_{t;i_t}+r\,\frac{\al_{s;j}-\al_{s;i_s}}{d}\right)
\prod\limits_{a_{k;1}<0}\!\!
\prod\limits_{r=1}^{-a_{k;s}d}\!\!\left(\sum\limits_{t=1}^p a_{k;t}\al_{t;i_t}
-r\,\frac{\al_{s;j}-\al_{s;i_s}}{d}\right)}
{d\underset{(r,k)\neq(d,j)}{\prod\limits_{r=1}^d\prod\limits_{k=1}^{n_s}}
            \left(\al_{s;i_s}-\al_{s;k}+r\,\frac{\al_{s;j}-\al_{s;i_s}}{d}\right)},            
\end{split}\end{equation*}
with $s\!\in\![p]$ and $j\!\neq\!i_s$.
The double sums in these equations are then replaced by triple sums over $s\!\in\![p]$,
$j\!\in\![n_s]\!-\!i_s$, and $d\!\in\!\Z^+$, and with $\F$ evaluated at
$$\x_t=\begin{cases}
\al_{s;j},&\hbox{if}~t\!=\!s;\\
\al_{t;i_t},&\hbox{if}~t\!\neq\!s;
\end{cases} \qquad
z=\frac{\al_{s;j}\!-\!\al_{s;i_s}}{d}.$$
The secondary coefficients~$\F_i^r(d)$ in~\e_ref{recurdfn_e2} now become
$\F_{i_1\ldots i_p}^r(d_1,\ldots,d_p)$, with $i_s\!\in\![n_s]$ and $d_s\!\in\!\Z^{\ge0}$.
In the analogue of Definition~\ref{MPC_dfn}, $\eta\!\in\!R(\x_1,\ldots,\x_p)$
is such that the evaluation of~$\eta$ at $(\al_{1;i_1},\ldots,\al_{p;i_p})$
for all elements $(i_1,\ldots,i_p)$
of $[n_1]\!\times\!\ldots\!\times\![n_p]$ is well-defined and not zero,
$\Phi_{\F}$  is a power series in $z_1,\ldots,z_p$
and $q_1,\ldots,q_p$, the sum is taken over all elements $(i_1,\ldots,i_p)$
of $[n_1]\!\times\!\ldots\!\times\![n_p]$, the leading fraction is replaced~by
$$\frac{\eta(\al_{1;i_1},\ldots,\al_{p;i_p})\ne^{\al_{1;i_1}z_1+\ldots+\al_{p;i_p}z_p}}
{\prod\limits_{s=1}^p\prod\limits_{k\neq i_s}(\al_{s;i_s}\!-\!\al_{s;k})}\,,$$
and the $q\ne^{\hb z}$-insertion in the first power series is replaced
by the insertions $q_1\ne^{\hb z_1},\ldots,q_p\ne^{\hb z_p}$.
Lemma~\ref{polgen_lmm} holds with
$$\dot\eta(\x_1,\ldots,\x_p)=
\frac{\prod\limits_{a_{k;1}\ge0}\sum\limits_{s=1}^pa_{k;s}\x_s}
{\prod\limits_{a_{k;1}<0}\sum\limits_{s=1}^pa_{k;s}\x_s}.$$

\section{Recursivity for stable quotients} 
\label{recpf_sec}

\noindent
In this section, we use the classical localization theorem~\cite{ABo}
to show that the generating functions $\dot\cZ_{n;\a;m'}^{(\b,\vp)}$ and 
$\ddot\cZ_{n;\a;m'}^{(\b,\vp)}$
defined in~\e_ref{cZgendfn_e} are recursive.
The argument is similar to the proof in \cite[Section~6]{GWvsSQ} of recursivity for 
the generating function~$\dot\cZ_{n;\a}$ defined by~\e_ref{cZdfn_e}, 
but requires some modifications.\\

\noindent
If $\T$ acts smoothly on a smooth compact oriented manifold $M$, there is a well-defined
integration-along-the-fiber homomorphism
$$\int_M\!: H_{\T}^*(M)\lra H_{\T}^*$$
for the fiber bundle $BM\!\lra\!B\T$.
The classical localization theorem of~\cite{ABo} relates it to
integration along the fixed locus of the $\T$-action.
The latter is a union of smooth compact orientable manifolds~$F$;
$\T$ acts on the normal bundle $\N F$ of each~$F$.
Once an orientation of $F$ is chosen, there is a well-defined
integration-along-the-fiber homomorphism
$$\int_F\!: H_{\T}^*(F)\lra H_{\T}^*.$$
The localization theorem states that
\begin{equation}\label{ABothm_e}
\int_M\eta = \sum_F\int_F\frac{\eta|_F}{\E(\N F)} \in \Q_{\al}
\qquad\forall~\eta\in H_{\T}^*(M),
\end{equation}
where the sum is taken over all components $F$ of the fixed locus of $\T$.
Part of the statement of~\e_ref{ABothm_e} is that $\E(\N F)$
is invertible in $H_{\T}^*(F)\!\otimes_{\Q[\al_1,\ldots,\al_n]}\!\Q_{\al}$.
In the case of the standard action of $\T$ on~$\P^{n-1}$, \e_ref{ABothm_e} implies that
\BE{phiprop_e} \eta|_{P_i}=\int_{\P^{n-1}}\eta\phi_i \in\Q_{\al}
\qquad\qquad\forall~\eta\!\in\!H_{\T}^*(\P^{n-1}),~i=1,2,\ldots,n,\EE
with $\phi_i$ as in~\e_ref{phidfn_e}.

\subsection{Fixed locus data} 
\label{FixLoc_subs}

\noindent
The proof of Lemma~\ref{recgen_lmm} involves a localization computation on $\ov{Q}_{0,m}(\Pn,d)$.
Thus, we need to describe the fixed loci of the $\T$-action on $\ov{Q}_{0,m}(\Pn,d)$,
their normal bundles, and the restrictions of the relevant cohomology classes
to these fixed loci.\\

\noindent
As in the case of stable maps described in \cite[Section~27.3]{MirSym},
the fixed loci of the $\T$-action on $\ov{Q}_{0,m}(\P^{n-1},d)$
are indexed by \sf{decorated graphs}, 
\BE{decortgraphdfn_e} \Ga = \big(\Ver,\Edg;\mu,\d,\vt\big),\EE
where $(\Ver,\Edg)$ is a connected graph that has no loops, with $\Ver$ and $\Edg$ denoting
its sets of vertices and edges, and
$$\mu\!:\Ver\lra [n], \qquad  \d\!: \Ver\!\sqcup\!\Edg\lra\Z^{\ge0},
\quad\hbox{and}\quad 
\vt\!: [m]\lra\Ver$$
are maps such that
\begin{gather}\label{decorgraphcond_e}
\mu(v_1)\neq\mu(v_2)  ~~~\hbox{if}~~\{v_1,v_2\}\in\Edg, \qquad
\d(e)\neq0~~\forall\,e\!\in\!\Edg,\\
\val(v)\equiv \big|\vt^{-1}(v)\big|\!+\!\big|\{e\!\in\!\Edg\!:~v\!\in\!e\}\big|\!+\!\d(v)\ge2
\qquad\forall\,v\!\in\!\Ver.\notag
\end{gather}
In Figure~\ref{loopgraph_fig}, the vertices of a decorated graph $\Ga$
are indicated by dots.
The values of the map $(\mu,\d)$ on some of the vertices
are indicated next to those vertices.
Similarly, the values of the map $\d$ on some of the edges are indicated next to~them.
The elements of the~sets $[m]$ are shown in bold face;
they are linked by line segments to their images under~$\vt$.
By~\e_ref{decorgraphcond_e}, no two consecutive vertices have the same first label
and thus $j\!\neq\!i$.\\

\noindent
With $\Ga$ as in~\e_ref{decortgraphdfn_e}, let
$$|\Ga|\equiv\sum_{v\in\Ver}\!\!\d(v)+\sum_{e\in\Edg}\!\!\d(e)$$
be the degree of~$\Ga$.
For each $v\!\in\!\Ver$, let
$$\tE_v=\big\{e\!\in\!\Edg\!:~v\!\in\!e\big\}$$
be the set of edges leaving from $v$.
There is a unique partial order $\prec$ on $\Ver$ 
that has a unique minimal element~$v_{\min}$ such that $v_{\min}\!=\!\vt(1)$ and $v\!\prec\!w$ if 
there exist distinct vertices $v_1,\ldots,v_k\!\in\!\Ver$ such~that 
$$v\in\{v_{\min},v_1,\ldots,v_{k-1}\},\qquad w=v_k,
\quad\hbox{and}\quad
\{v_{\min},v_1\},\{v_1,v_2\},\ldots,\{v_{k-1},v_k\}\in\Edg,$$
i.e.~$v$ lies between $v_0$ and $w$ in $(\Ver,\Edg)$.
If $e\!=\!\{v_1,v_2\}\!\in\!\Edg$ is any edge in $\Ga$ with $v_1\!\prec\!v_2$, let
\begin{gather*}
\Ga_e\equiv \big(\{v_1,v_2\},\{e\};\mu_e,\d_e,\vt_e\big), \qquad\hbox{where}\\
\mu_e=\mu|_e, \qquad \d_e(e)=\d(e),~~\d_e|_e=0, 
\qquad \vt_e\!:\{1,2\}\lra e, ~~~\vt_e(1)=v_1,~~\vt_e(2)=v_2;
\end{gather*}
see Figure~\ref{subgraphs_fig}.\\

\noindent
With $m'\!\le\!m$ as in Lemmas~\ref{recgen_lmm} and~\ref{polgen_lmm}, let
\begin{gather*}
\Ver_{m'}=\big\{v\!\in\!\Ver\!:~v\!\preceq\!\vt(i)~\hbox{for some}~i\!\in\![m']\big\},
\qquad \Edg_{m'}=\big\{\{v_1,v_2\}\!\in\!\Edg\!:~v_1,v_2\!\in\!\Ver_{m'}\big\};
\end{gather*}
in particular, the graph $(\Ver_{m'},\Edg_{m'})$ is a tree.
For each $v\!\in\!\Ver_{m'}$, define
\begin{gather*}
r_{m';v}\!: \tE_v\!-\!\Edg_{m'}\lra\Z^+ \qquad\hbox{by}\qquad
r_{m';v}\big(\{v,v'\}\big)
=\sum_{\begin{subarray}{c}\{v_1,v_2\}\in\Edg\\ v'\preceq v_2\end{subarray}}
\!\!\!\!\!\!\!\!\d\big(\{v_1,v_2\}\big)
+\sum_{\begin{subarray}{c}w\in\Ver\\ v'\preceq w\end{subarray}}\!\!\d(w),\\
\d_{m'}(v)=\d(v)+\sum_{e\in\tE_v-\Edg_{m'}}\!\!\!\!\!\!r_{m;'v}(e)\,.
\end{gather*}\\

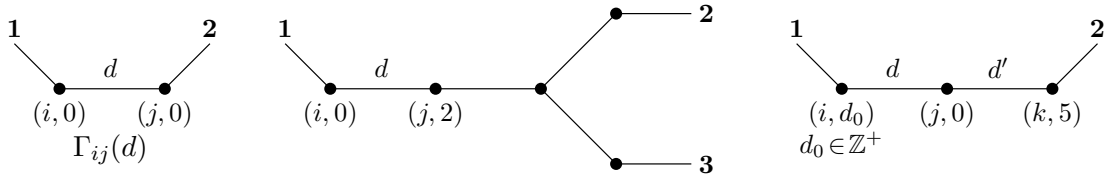
\begin{figure}
\begin{pspicture}(-.7,-1)(10,1.2)
\psset{unit=.4cm}
\psline[linewidth=.04](2.5,0)(6,0)\rput(4.2,.7){\smsize{$d$}}
\pscircle*(2.5,0){.2}\rput(2.5,-.85){\smsize{$(i,0)$}}
\pscircle*(6,0){.2}\rput(6,-.85){\smsize{$(j,0)$}}
\psline[linewidth=.04](2.5,0)(1,1.5)\rput(1,2){\smsize{$\bf 1$}}
\psline[linewidth=.04](6,0)(7.5,1.5)\rput(7.5,2){\smsize{$\bf 2$}}
\rput(4.2,-2){$\Ga_{ij}(d)$}
\psline[linewidth=.04](11.5,0)(15,0)\rput(13.2,.7){\smsize{$d$}}
\psline[linewidth=.04](15,0)(18.5,0)\pscircle*(18.5,0){.2}
\psline[linewidth=.04](18.5,0)(21,2.5)\pscircle*(21,2.5){.2}
\psline[linewidth=.04](18.5,0)(21,-2.5)\pscircle*(21,-2.5){.2}
\pscircle*(11.5,0){.2}\rput(11.5,-.85){\smsize{$(i,0)$}}
\pscircle*(15,0){.2}\rput(15,-.85){\smsize{$(j,2)$}}
\psline[linewidth=.04](11.5,0)(10,1.5)\rput(10,2){\smsize{$\bf 1$}}
\psline[linewidth=.04](21,2.5)(23.5,2.5)\rput(24,2.5){\smsize{$\bf 2$}}
\psline[linewidth=.04](21,-2.5)(23.5,-2.5)\rput(24,-2.5){\smsize{$\bf 3$}}
\psline[linewidth=.04](28.5,0)(32,0)\pscircle*(32,0){.2}
\rput(30.2,.7){\smsize{$d$}}\rput(33.7,.7){\smsize{$d'$}}
\pscircle*(28.5,0){.2}\rput(28.5,-.85){\smsize{$(i,d_0)$}}
\rput(28.5,-1.85){\smsize{$d_0\!\in\!\Z^+$}}
\rput(32,-.85){\smsize{$(j,0)$}}
\rput(35.5,-.85){\smsize{$(k,5)$}}
\psline[linewidth=.04](32,0)(35.5,0)\pscircle*(35.5,0){.2}
\psline[linewidth=.04](28.5,0)(27,1.5)\rput(27,2){\smsize{$\bf 1$}}
\psline[linewidth=.04](35.5,0)(37,1.5)\rput(37,2){\smsize{$\bf 2$}}
\end{pspicture}
\caption{Two trees with $\val(v_{\min})\!=\!2$ and a tree with $\val(v_{\min})\!\ge\!3$}
\label{loopgraph_fig}
\end{figure}

\begin{figure}
\begin{pspicture}(-2.5,-.6)(10,1)
\psset{unit=.4cm}
\psline[linewidth=.04](2.5,0)(6,0)\rput(4.2,.7){\smsize{$d$}}
\pscircle*(2.5,0){.2}\rput(2.5,-.85){\smsize{$(i,0)$}}
\pscircle*(6,0){.2}\rput(6,-.85){\smsize{$(j,0)$}}
\psline[linewidth=.04](2.5,0)(1,1.5)\rput(1,2){\smsize{$\bf 1$}}
\psline[linewidth=.04](6,0)(7.5,1.5)\rput(7.5,2){\smsize{$\bf 2$}}
\psline[linewidth=.04](22.5,0)(26,0)\rput(24.2,.7){\smsize{$d'$}}
\pscircle*(22.5,0){.2}\rput(22.5,-.85){\smsize{$(j,0)$}}
\pscircle*(26,0){.2}\rput(26,-.85){\smsize{$(k,0)$}}
\psline[linewidth=.04](22.5,0)(21,1.5)\rput(21,2){\smsize{$\bf 1$}}
\psline[linewidth=.04](26,0)(27.5,1.5)\rput(27.5,2){\smsize{$\bf 2$}}
\end{pspicture}
\caption{The subtrees corresponding to the edges of the last graph in Figure~\ref{loopgraph_fig}.}
\label{subgraphs_fig}
\end{figure}
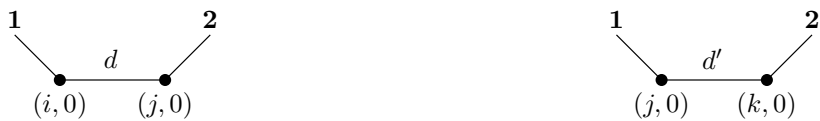

\noindent
As described in \cite[Section~7.3]{MOP09},
the fixed locus $Q_{\Ga}$ of $\ov{Q}_{0,m}(\P^{n-1},|\Ga|)$ corresponding to a
decorated graph $\Ga$ consists of the stable quotients
$$(\cC,y_1,\ldots,y_m;S\subset \C^n\!\otimes\!\cO_{\cC})$$
over quasi-stable rational $m$-marked curves that satisfy the following conditions.
The components of $\cC$ on which the corresponding quotient is torsion-free
are rational and correspond to the edges of~$\Ga$;
the restriction of~$S$ to any such component corresponds
to a morphism to~$\P^{n-1}$ of the opposite degree to that of the subsheaf.
Furthermore, if $e\!=\!\{v_1,v_2\}$ is an edge, the corresponding morphism~$f_e$
is a degree-$\d(e)$ cover of the line
$$\P^1_{\mu(v_1),\mu(v_2)}\subset\P^{n-1}$$
passing through the fixed points $P_{\mu(v_1)}$ and $P_{\mu(v_2)}$;
it is ramified only over $P_{\mu(v_1)}$ and~$P_{\mu(v_2)}$.
In particular, $f_e$ is unique up to isomorphism.
The remaining components of $\cC$ are indexed by the vertices $v\!\in\!\Ver$ 
of valence $\val(v)\!\ge\!3$.
The restriction of~$S$ to such a component~$\cC_v$ of~$\cC$ (or possibly a connected union
of irreducible components) is a subsheaf
of the trivial subsheaf $P_{\mu(v)}\!\subset\!\C^n\!\otimes\!\cO_{\cC_v}$
of degree~$-\d(v)$;
thus, the induced morphism takes~$\cC_v$ to the fixed point~$P_{\mu(v)}\!\in\!\Pn$.
Each such component~$\cC_v$ also carries 
$|\vt^{-1}(v)|\!+\!|\tE_v|$ marked points, corresponding
to the marked points  and/or the nodes of~$\cC$;
we index these points by the set $\vt^{-1}(v)\!\sqcup\!\tE_v$ in the canonical way.
Thus, as stacks,
\BE{Zlocus_e}\begin{split}
Q_{\Ga}&\approx \prod_{\begin{subarray}{c}v\in\Ver\\ \val(v)\ge3\end{subarray}}
\!\!\!\ov{Q}_{0,|\vt^{-1}(v)|+|\tE_v|}(\P^0,\d(v))\times\prod_{e\in\Edg}\!\!Q_{\Ga_e}\\
&\approx  \prod_{\begin{subarray}{c}v\in\Ver\\ \val(v)\ge3\end{subarray}}
\!\!\!\ov\cM_{0,|\vt^{-1}(v)|+|\tE_v||\d(v)}/\bS_{\d(v)}\times\prod_{e\in\Edg}\!\!Q_{\Ga_e}\\
&\approx
\Bigg(\prod_{\begin{subarray}{c}v\in\Ver\\ \val(v)\ge3\end{subarray}}
\!\!\!\ov\cM_{0,|\vt^{-1}(v)|+|\tE_v||\d(v)}/\bS_{\d(v)}\bigg)\Bigg/\prod_{e\in\Edg}\!\!\Z_{\d(e)},
\end{split}\EE
with each cyclic group $\Z_{\d(e)}$ acting trivially.
For example, in the case of the last diagram in Figure~\ref{loopgraph_fig},
$$Q_{\Ga}\approx
\Big(\ov\cM_{0,2|d_0}/\bS_{d_0}\times \ov\cM_{0,2|5}/\bS_5\Big)
\big/\Z_d\!\times\!\Z_{d'}$$
is a fixed locus in $\ov{Q}_{0,2}(\P^{n-1},d_0\!+\!5\!+\!d\!+\!d')$.
If  $m'\!\le\!m$ is as in Lemmas~\ref{recgen_lmm} and~\ref{polgen_lmm}, 
the morphism~$f_{m',m}$ in~\e_ref{fmmdfn_e} sends the locus $Q_{\Ga}$
of $\ov{Q}_{0,m}(\Pn,d)$ to (a subset of) the locus $Q_{\Ga_{m'}}$
of $\ov{Q}_{0,m'}(\Pn,d)$, where
$$\Ga_{m'}=\big(\Ver_{m'},\Edg_{m'};\mu|_{\Ver_{m'}},\d_{m'},\vt|_{[m']} \big),$$
as $f_{m',m}$ contracts the ends of the elements of $\ov{Q}_{0,m'}(\Pn,d)$ that do not
carry any of the marked points indexed by the set~$[m']$.\\

\noindent
If $v\!\in\!\Ver$ and $\val(v)\!\ge\!3$, for the purposes of definitions~\e_ref{V0prdfn_e} and~\e_ref{W0dfn_e} we identify $[|\vt^{-1}(v)|\!+\!|\tE_v|]$ with the set 
$\vt^{-1}(v)\!\sqcup\!\tE_v$ indexing the marked points on~$\cC_v$ so that
the element~1 in the former is identified with $1\!\in\![m]$ if $\vt(1)\!=\!v$ and 
with the unique edge $e_v^-\!=\!\{v_-,v\}$ with $v^-\!\prec\!v$ 
separating $v$ from the marked point~1 otherwise.
Similarly, if $v\!\preceq\!\vt(2)$, we associate the element~2 of $[|\vt^{-1}(v)|\!+\!|\tE_v|]$
with $2\!\in\![m]$ if $\vt(2)\!=\!v$ and 
with the unique edge $e_v^+\!=\!\{v,v_+\}$ with $v_+\!\preceq\!\vt(2)$ 
separating $v$ from the marked point~2 otherwise.
Finally, if $m'\!\le\!m$ is as in Lemmas~\ref{recgen_lmm} and~\ref{polgen_lmm} and 
$v\!\in\!\Ver_{m'}$, we associate the $|\tE_v\!-\!\Edg_{m'}|$ largest elements of  
$[|\vt^{-1}(v)|\!+\!|\tE_v|]$ with the subset $\tE_v\!-\!\Edg_{m'}$ of 
$\vt^{-1}(v)\!\sqcup\!\tE_v$.\\

\noindent
If $\Ga$ is a decorated graph as above and $e\!=\!\{v_1,v_2\}\in\!\Edg$
with $v_1\!\prec\!v_2$, let
$$\pi_e\!:Q_{\Ga}\lra Q_{\Ga_e}\subset \ov{Q}_{0,2}(\Pn,\d(e))$$
be the projection in the decomposition~\e_ref{Zlocus_e} and 
$$\om_{e;v_1}=-\pi_e^*\psi_1,~\om_{e;v_2}=-\pi_e^*\psi_2\in H^2(Q_{\Ga}).$$
Similarly, for each $v\!\in\!\Ver$ such that $\val(v)\!\ge\!3$, let
$$\pi_v\!:Q_{\Ga}\lra\ov\cM_{0,|\vt^{-1}(v)|+|\tE_v||\d(v)}/\bS_{\d(v)}$$
be the corresponding projection and
$$\psi_{v;e}=\pi_v^*\psi_e \in H^2(Q_{\Ga}) \qquad
\forall~v\!\in\!\tE_v\,.$$
By \cite[Section~27.2]{MirSym},
\BE{psiform_e} \om_{e;v_i}=\frac{\al_{\mu(v_i)}-\al_{\mu(v_{3-i})}}{\d(e)}\
\qquad i=1,2.\EE
By \cite[Section~7.4]{MOP09}, the euler class of the normal bundle of $Q_{\Ga}$ in
$\ov{Q}_{0,m}(\Pn,|\Ga|)$ is described~by
\BE{NZform_e}\begin{split}
\frac{\E(\N Q_{\Ga})}{\E(T_{\mu(v_{\min})}\Pn)}
&=\prod_{\begin{subarray}{c}v\in\Ver\\ \val(v)\ge3\end{subarray}}
\prod_{k\neq\mu(v)}\!\!\!\pi_v^*\E\big(\dot\V_1^{(\d(v))}(\al_{\mu(v)}\!-\!\al_k)\big)~
\prod_{e\in\Edg}\!\!\!\pi_e^*\E\big(H^0(f_e^*T\P^n\!\otimes\!\cO(-y_1))/\C\big)\\
&\quad
\times\prod_{\begin{subarray}{c}v\in\Ver\\ \val(v)=2,\vt^{-1}(v)=\eset\end{subarray}}
\hspace{-.3in}\bigg(\sum_{e\in\tE_v}\!\om_{e;v}\bigg)
\prod_{\begin{subarray}{c}v\in\Ver\\ \val(v)\ge3\end{subarray}}\!\!\!
\bigg(\prod_{e\in\tE_v}\!\!\big(\om_{e;v}\!-\!\psi_{v;e}\big)\bigg),
\end{split}\EE
where $\C\!\subset\!H^0(f_e^*T\P^n\!\otimes\!\cO(-y_1))$ denotes the trivial $\T$-representation.
The terms on the first line correspond to the deformations of the sheaf without changing
the domain, while the terms on the second line correspond to the deformations of the domain.
By~\e_ref{Wdfn_e}, \e_ref{Vprdfn_e}, \e_ref{V0prdfn_e}, and~\e_ref{W0dfn_e}, 
\BE{cVform_e}\begin{split}
\E(\dot\V_{n;\a;m'}^{(|\Ga|)})\big|_{Q_{\Ga}}
&=\prod_{\begin{subarray}{c}v\in\Ver_{m'}\\ \val(v)\ge3\end{subarray}}
\!\!\!\!\!\pi_v^*\E\big(\dot\V_{\a;r_{m';v}}^{(\d(v))}(\al_{\mu(v)})\big)
\cdot \prod_{e\in\Edg_{m'}}\!\!\!\!\!\!\pi_e^*\E\big(\dot\V_{n;\a}^{\d(e)}\big)\,,\\
\E(\ddot\V_{n;\a;m'}^{(|\Ga|)})\big|_{Q_{\Ga}}
&=\prod_{\begin{subarray}{c}v\in\Ver_2\\ \val(v)\ge3\end{subarray}}
\!\!\!\pi_v^*\E\big(\ddot\V_{\a;r_{m';v}}^{(\d(v))}(\al_{\mu(v)})\big)
\cdot \prod_{e\in\Edg_2}\!\!\!\pi_e^*\E\big(\ddot\V_{n;\a}^{\d(e)}\big)\\
&\quad\times\prod_{\begin{subarray}{c}v\in\Ver_{m'}-\Ver_2\\ \val(v)\ge3\end{subarray}}
\!\!\!\!\!\!\!\!\!\!\!\!\pi_v^*\E\big(\dot\V_{\a;r_{m';v}}^{(\d(v))}(\al_{\mu(v)})\big)
\cdot \prod_{e\in\Edg_{m'}-\Edg_2}\!\!\!\!\!\!\!\!\!\!\!\pi_e^*\E\big(\dot\V_{n;\a}^{\d(e)}\big)\,.
\end{split}\EE
By \cite[Section~27.2]{MirSym},
\BE{EdgContr_e}\begin{split}
\int_{Q_{\Ga_e}}
\frac{\E(\dot\V_{n;\a}^{\d(e)})}{\E\big(H^0(f_e^*T\P^n\!\otimes\!\cO(-y_1))/\C\big)}
&=\dot\fC_{\mu(v_1)}^{\mu(v_2)}\big(\d(e)\big)\,,\\
\int_{Q_{\Ga_e}}
\frac{\E(\ddot\V_{n;\a}^{\d(e)})}{\E\big(H^0(f_e^*T\P^n\!\otimes\!\cO(-y_1))/\C\big)}
&=\ddot\fC_{\mu(v_1)}^{\mu(v_2)}\big(\d(e)\big)\,,
\end{split}
\qquad\forall~e\!=\!\{v_1,v_2\}~\hbox{with}~v_1\!\prec\!v_2,\EE
with $\dot\fC_{\mu(v_1)}^{\mu(v_2)}(\d(e))$  and 
$\ddot\fC_{\mu(v_1)}^{\mu(v_2)}(\d(e))$ given by~\e_ref{Cdfn}.

\subsection{Proof of Lemma~\ref{recgen_lmm}} 
\label{RecPf_subs}

\noindent
We apply the localization theorem~to
\BE{Zeval_e}\begin{split}
\dot\cZ_{n;\a;m'}^{(\b,\vp)}(\al_i,\hb,q)
&=\sum_{d=0}^{\i}q^d\int_{\ov{Q}_{0,m}(\Pn,d)}
\frac{\E(\dot\V_{n;\a;m'}^{(d)})\ev_1^*\phi_i}{\hb\!-\!\psi_1}\prod_{j=2}^m(\psi_j^{b_j}\ev_j^*\vp_j)\,,\\
\ddot\cZ_{n;\a;m'}^{(\b,\vp)}(\al_i,\hb,q)
&=\sum_{d=0}^{\i}q^d\int_{\ov{Q}_{0,m}(\Pn,d)}
\frac{\E(\ddot\V_{n;\a;m'}^{(d)})\ev_1^*\phi_i}{\hb\!-\!\psi_1}\prod_{j=2}^m(\psi_j^{b_j}\ev_j^*\vp_j)\,,
\end{split}\EE
where $\phi_i$ is the equivariant Poincar\'{e} dual of the fixed point $P_i\!\in\!\Pn$,
as in~\e_ref{phidfn_e}, and 
the degree~0 terms in the $m\!=\!2$ case are defined~by
\begin{equation*}\begin{split}
\int_{\ov{Q}_{0,2}(\Pn,0)}
\frac{\E(\dot\V_{n;\a;m'}^{(d)})\ev_1^*\phi_i}{\hb\!-\!\psi_1}(\psi_2^{b_2}\ev_2^*\vp_2)
&\equiv (-\hb)^{b_2}\vp_2|_{P_i}\,,\\
\int_{\ov{Q}_{0,2}(\Pn,0)}
\frac{\E(\ddot\V_{n;\a;m'}^{(d)})\ev_1^*\phi_i}{\hb\!-\!\psi_1}(\psi_2^{b_2}\ev_2^*\vp_2)
&\equiv (-\hb)^{b_2}\vp_2|_{P_i}\,.
\end{split}\end{equation*}
Since $\phi_i|_{P_j}\!=\!0$ unless $j\!=\!i$, a decorated graph as in~\e_ref{decortgraphdfn_e}
contributes to the two expressions in~\e_ref{Zeval_e} only if the first marked point is attached
to a vertex labeled~$i$, i.e.~$\mu(v_{\min})\!=\!i$ for the smallest element $v_{\min}\!\in\!\Ver$.
We show that, just as for Givental's $J$-function, 
the $(d,j)$-summand in~\e_ref{recurdfn_e2} with $C\!=\!\dot\fC,\ddot\fC$ and
$\F\!=\!\dot\cZ_{n;\a;m'}^{(\b,\vp)},\ddot\cZ_{n;\a;m'}^{(\b,\vp)}$,  i.e.
\BE{djsummand_e}
\frac{\dot\fC_i^j(d)q^d}{\hb-\frac{\al_j-\al_i}{d}}
\dot\cZ_{n;\a;m'}^{(\b,\vp)}(\al_j,(\al_j\!-\!\al_i)/d,q)
\quad\hbox{and}\quad
\frac{\ddot\fC_i^j(d)q^d}{\hb-\frac{\al_j-\al_i}{d}}
\ddot\cZ_{n;\a;m'}^{(\b,\vp)}(\al_j,(\al_j\!-\!\al_i)/d,q)\,,\EE
respectively,
is the sum over all graphs such that $\mu(v_{\min})\!=\!i$, i.e.~the first marked point
is mapped to the fixed point $P_i\!\in\!\Pn$, 
$v_{\min}$ is a bivalent vertex,  i.e.~$\d(v_{\min})\!=\!0$, $\vt^{-1}(v_{\min})\!=\!\{1\}$, 
the only edge leaving this vertex is labeled~$d$, and
the other vertex of this edge is labeled~$j$.
We also show that the first sum on the right-hand side of~\e_ref{recurdfn_e2} is
the sum over  all graphs such that $\mu(v_{\min})\!=\!i$ and $\val(v_{\min})\!\ge\!3$.\\

\noindent
If $\Ga$ is a decorated graph with $\mu(v_{\min})\!=\!i$ as above,
\BE{phirestr_e} \ev_1^*\phi_i\big|_{Q_{\Ga}}
=\prod_{k\neq i}(\al_i\!-\!\al_k)
=\E\big(T_{\mu(v_{\min})}\Pn\big).\EE
Suppose in addition that $\val(v_{\min})\!=\!2$ and $|\tE_{v_{\min}}|\!=\!1$.
Let $v_1\!\equiv\!(v_{\min})_+$ be the immediate successor of $v_{\min}$ in~$\Ga$
and $e_1\!=\!\{v_{\min},v_1\}$ be the edge leaving~$v_{\min}$.
If $|\Edg|\!>\!1$ or $\val(v_1)\!>\!2$,
i.e.~$\Ga$ is not as in the first diagram in Figure~\ref{loopgraph_fig},
we break~$\Ga$ at $v_1$ into two ``sub-graphs":
\begin{enumerate}[label=(\roman*)]
\item $\Ga_1\!=\!\Ga_{e_1}$ consisting of the vertices $v_{\min}\!\prec\!v_1$,
the edge $\{v_{\min},v_1\}$, with the $\d$-value of~0 at both vertices, and a marked point
at $v$ and~$v_1$;
\item $\Ga_2$ consisting of all vertices, edges, and marked points of $\Ga$, other than
the vertex $v_{\min}$ and the edge $\{v_{\min},v_1\}$, and with the marked point~1
attached at~$v_1$;
\end{enumerate}
see Figure~\ref{splitgraph_fig}.
By~\e_ref{Zlocus_e},
\BE{Qsplit_e} Q_{\Ga}\approx Q_{\Ga_1}\times Q_{\Ga_2}.\EE
Let $\pi_1,\pi_2\!:Q_{\Ga}\lra Q_{\Ga_1},Q_{\Ga_2}$ be the component projection maps.
By~\e_ref{NZform_e} and~\e_ref{cVform_e},
\begin{equation*}\begin{split}
\frac{\E(\N Q_{\Ga})}{\E(T_{P_i}\Pn)}&=
\pi_1^*\bigg(\frac{\E(\N Q_{\Ga_1})}{\E(T_{P_i}\P^{n-1})}\bigg)
\cdot\pi_2^*\bigg(\frac{\E(\N Q_{\Ga_2})}{\E(T_{P_{\mu(v_1)}}\P^{n-1})}\bigg)
\cdot \big(\om_{e;v_1}-\pi_2^*\psi_1\big),\\
\E\big(\dot\V_{n;\a;m'}^{(|\Ga|)}\big)\big|_{Q_{\Ga}}&=
  \pi_1^*\E\big(\dot\V_{n;\a}^{(|\Ga_1|)}\big)\cdot\pi_2^*\E(\dot\V_{n;\a;m'}^{(|\Ga_2|)}\big)\,,
  \qquad
\E\big(\ddot\V_{n;\a;m'}^{(|\Ga|)}\big)\big|_{Q_{\Ga}}=
  \pi_1^*\E\big(\ddot\V_{n;\a}^{(|\Ga_1|)}\big)\cdot\pi_2^*\E(\ddot\V_{n;\a;m'}^{(|\Ga_2|)}\big)\,.
\end{split}\end{equation*}\\

\begin{figure}
\begin{pspicture}(-3.5,-1)(10,1)
\psset{unit=.4cm}
\psline[linewidth=.04](1.5,0)(5,0)\rput(3.2,.7){\smsize{$d$}}
\pscircle*(1.5,0){.2}\rput(1.5,-.85){\smsize{$(i,0)$}}
\pscircle*(5,0){.2}\rput(5,-.85){\smsize{$(j,0)$}}
\psline[linewidth=.04](1.5,0)(0,1.5)\rput(0,2){\smsize{$\bf 1$}}
\psline[linewidth=.04](5,0)(6.5,1.5)\rput(6.5,2){\smsize{$\bf 2$}}
\rput(3.3,-1.7){$\Ga_1$}
\psline[linewidth=.04](15,0)(18.5,0)\pscircle*(18.5,0){.2}
\psline[linewidth=.04](18.5,0)(21,2.5)\pscircle*(21,2.5){.2}
\psline[linewidth=.04](18.5,0)(21,-2.5)\pscircle*(21,-2.5){.2}
\pscircle*(15,0){.2}\rput(15,-.85){\smsize{$(j,2)$}}
\psline[linewidth=.04](15,0)(13.5,1.5)\rput(13.5,2){\smsize{$\bf 1$}}
\psline[linewidth=.04](21,2.5)(24,2.5)\rput(24.5,2.5){\smsize{$\bf 2$}}
\psline[linewidth=.04](21,-2.5)(24,-2.5)\rput(24.5,-2.5){\smsize{$\bf 3$}}
\rput(17.5,-1.7){$\Ga_2$}
\end{pspicture}
\caption{The two sub-graphs of the second graph in Figure~\ref{loopgraph_fig}.}
\label{splitgraph_fig}
\end{figure}
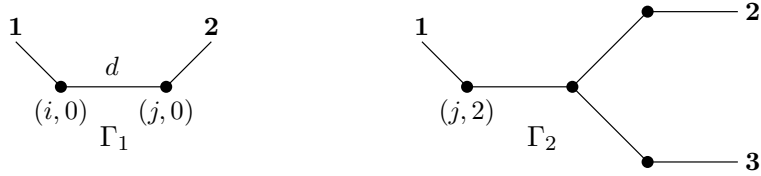

\noindent
Combining the above splittings
with~\e_ref{psiform_e}, \e_ref{EdgContr_e}, and~\e_ref{phirestr_e}, we find that
\begin{equation*}\begin{split}
&q^{|\Ga|}\!\int_{Q_{\Ga}}
\frac{\E(\dot\V_{n;\a;m'}^{(|\Ga|)})\ev_1^*\phi_i}{\hb\!-\!\psi_1}
\prod_{j=2}^{j=m}(\psi_j^{b_j}\ev_j^*\vp_j)\Big|_{Q_{\Ga}}
\frac{1}{\E(\N Q_{\Ga})}\\
&\hspace{.2in}
=\frac{\dot\fC_i^{\mu(v_1)}(\d(e_1))q^{\d(e_1)}}{\hb-\frac{\al_{\mu(v_1)}-\al_i}{\d(e_1)}}
\Bigg(q^{|\Ga_2|}\!\bigg\{\int_{Q_{\Ga_2}}\!\!\!\!\!
\frac{\E(\dot\V_{n;\a;m'}^{(|\Ga_2|)})\ev_1^*\phi_{\mu(v_1)}}{\hb\!-\!\psi_1}
\prod_{j=2}^{j=m}(\psi_j^{b_j}\ev_j^*\vp_j)
\frac{1}{\E(\N Q_{\Ga_2})}\bigg\}\bigg|_{\hb=\frac{\al_{\mu(v_1)}-\al_i}{\d(e_1)}}\Bigg)\,;
\end{split}\end{equation*}
the same identity holds with $\dot\V$ replaced by $\ddot\V$ and
$\dot\fC_i^{\mu(v_1)}(\d(e_1))$ by $\ddot\fC_i^{\mu(v_1)}(\d(e_1))$.
By the first equation in~\e_ref{Zeval_e} with $i$ replaced by $\mu(v_1)$ 
and the localization formula~\e_ref{ABothm_e}, 
the sum of the last factor above over all possibilities for~$\Ga_2$, 
with $\Ga_1$ held fixed,~is
$$\dot\cZ_{n;\a;m'}^{(\b,\vp)}
\big(\al_{\mu(v_1)},(\al_{\mu(v_1)}\!-\!\al_i)/\d(e_1),q\big)-
\de_{m,2}\bigg(\frac{\al_i\!-\!\al_{\mu(v_1)}}{\d(e_1)}\bigg)^{b_2}\vp_2|_{P_{\mu(v_1)}};$$
if $\dot\V$ is replaced by $\ddot\V$, then the sum becomes 
$$\ddot\cZ_{n;\a;m'}^{(\b,\vp)}
\big(\al_{\mu(v_1)},(\al_{\mu(v_1)}\!-\!\al_i)/\d(e_1),q\big)-
\de_{m,2}\bigg(\frac{\al_i\!-\!\al_{\mu(v_1)}}{\d(e_1)}\bigg)^{b_2}\vp_2|_{P_{\mu(v_1)}}\,.$$
In the $m\!=\!2$ case,  the contributions of the one-edge graph $\Ga_{i\mu(v_1)}(\d(e_1))$
such as $\d(v_1)\!=\!0$, as in the first diagram in Figure~\ref{loopgraph_fig}, 
to the two expressions in~\e_ref{Zeval_e} are 
$$\frac{\dot\fC_i^{\mu(v_1)}(\d(e_1))q^{\d(e_1)}}
{\hb_1-\frac{\al_{\mu(v_1)}-\al_i}{\d(e_1)}}
\bigg(\frac{\al_i\!-\!\al_{\mu(v_1)}}{\d(e_1)}\bigg)^{\!\!b_2}\vp_2|_{P_{\mu(v_1)}}
\quad\hbox{and}\quad
\frac{\ddot\fC_i^{\mu(v_1)}(\d(e_1))q^{\d(e_1)}}
{\hb_1-\frac{\al_{\mu(v_1)}-\al_i}{\d(e_1)}}
\bigg(\frac{\al_i\!-\!\al_{\mu(v_1)}}{\d(e_1)}\bigg)^{\!\!b_2}
\vp_2|_{P_{\mu(v_1)}}\,,$$
respectively.
Thus, the contributions to the two expressions in~\e_ref{Zeval_e} 
from all graphs~$\Ga$ such that $\d(v_{\min})\!=\!0$,
$\mu(v_1)\!=\!j$, and $\d(e_1)\!=\!d$ are given by~\e_ref{djsummand_e},
i.e.~they are the $(d,j)$-summands in the recursions~\e_ref{recurdfn_e2} 
for~$\dot\cZ_{n;\a;m'}^{(\b,\vp)}$ and~$\ddot\cZ_{n;\a;m'}^{(\b,\vp)}$.\\

\noindent
Suppose next that $\Ga$ is a graph such that $\mu(v_{\min})\!=\!i$ and $\val(v_{\min})\!\ge\!3$.
If $|\Ver|\!>\!1$, i.e.~$\Ga$ is not as in the first diagram in Figure~\ref{splitgraph_fig2},
we break~$\Ga$ at $v_{\min}$ into ``sub-graphs":
\begin{enumerate}[label=(\roman*)]
\item\label{Ga0split_it}  $\Ga_0$ consisting of the vertex $\{v_{\min}\}$ only,
with the same $\mu$ and $\d$-values as in~$\Ga$, with the same marked points as before,
along with a marked point~$e$ for each edge $e\!\in\!\tE_{v_{\min}}$ from~$v_{\min}$; 
\item for each $e\!\in\!\tE_{v_{\min}}$, $\Ga_{c;e}$ consisting of the branch of $\Ga$
beginning with the edge~$e$ at $v_{\min}$, 
with the $\d$-value of $v_{\min}$ replaced by~0, and with one marked point at~$v_{\min}$;
\end{enumerate}
see Figures~\ref{splitgraph_fig2} and~\ref{splitgraph_fig2b}.
By~\e_ref{Zlocus_e},
\begin{gather}\label{Zlocus_e2}
Q_{\Ga}\approx Q_{\Ga_0}\times \prod_{e\in\tE_{v_{\min}}}\!\!\!\!\!\!Q_{\Ga_{c;e}}=
(\ov\cM_{0,m_0|\d(v_{\min})}/\bS_{\d(v_{\min})})\times 
\prod_{e\in\tE_{v_{\min}}}\!\!\!\!\!\! Q_{\Ga_{c;e}}\,,\\
\hbox{where}\qquad m_0=|\vt^{-1}(v_{\min})|\!+\!|\tE_{v_{\min}}|. \notag
\end{gather}
Let $\pi_0,\pi_{c;e}$ be the component projection maps in~\e_ref{Zlocus_e2}.
Since $\psi_1|_{Q_{\Ga}}=\pi_0^*\psi_1$,
$\T$ acts trivially on $\ov\cM_{0,m_0|\d(v_{\min})}$,
$$\psi_1=1\times\psi_1\in H_{\T}^*\big(\ov\cM_{0,m_0|\d(v_{\min})}\big)
=H_{\T}^*\otimes H^*\big(\ov\cM_{0,m_0|\d(v_{\min})}\big),$$
i.e.~$\T$ acts trivially on the universal cotangent line bundle for the first marked point
on $\ov\cM_{0,m_0|\d(v_{\min})}$, and
the dimension of $\ov\cM_{0,m_0|\d(v_{\min})}$ is $m_0\!+\!\d(v_{\min})\!-\!3$,
$$\frac{1}{\hb-\psi_1}\big|_{Q_{\Ga}}
=\sum_{r=0}^{m_0+\d(v_{\min})-3}\hb^{-(r+1)}\pi_0^*\psi_1^r\,.$$
Since $m_0\!+\!\d(v_{\min})\!\le\!m\!+\!|\Ga|$ and 
$\Ga$ contributes to the coefficient of $q^{|\Ga|}$
in~\e_ref{Zeval_e}, 
it follows that \e_ref{recurdfn_e2} holds with $\F$ replaced by 
$\dot\cZ_{n;\a;m'}^{(\b,\vp)}$ and $\ddot\cZ_{n;\a;m'}^{(\b,\vp)}$
with $N_d\!=\!m\!+\!d\!-\!2$, $C_i^j(d)\!=\!\dot\fC_i^j(d)$ in the first case, and 
$C_i^j(d)\!=\!\ddot\fC_i^j(d)$ in the second case.\\

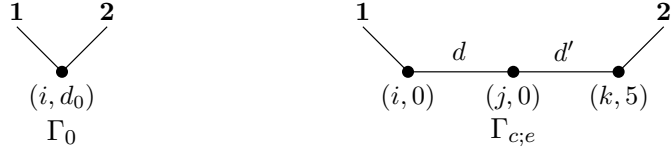
\begin{figure}
\begin{pspicture}(1.5,-.8)(10,1)
\psset{unit=.4cm}
\pscircle*(15,0){.2}\rput(15,-.85){\smsize{$(i,d_0)$}}
\psline[linewidth=.04](15,0)(13.5,1.5)\rput(13.5,2){\smsize{$\bf 1$}}
\psline[linewidth=.04](15,0)(16.5,1.5)\rput(16.5,2){\smsize{$\bf 2$}}
\rput(15,-2){$\Ga_0$}
\psline[linewidth=.04](26.5,0)(30,0)\pscircle*(30,0){.2}
\rput(28.2,.7){\smsize{$d$}}\rput(31.7,.7){\smsize{$d'$}}
\pscircle*(26.5,0){.2}\rput(26.5,-.85){\smsize{$(i,0)$}}
\rput(30,-.85){\smsize{$(j,0)$}}
\rput(33.5,-.85){\smsize{$(k,5)$}}
\psline[linewidth=.04](30,0)(33.5,0)\pscircle*(33.5,0){.2}
\psline[linewidth=.04](26.5,0)(25,1.5)\rput(25,2){\smsize{$\bf 1$}}
\psline[linewidth=.04](33.5,0)(35,1.5)\rput(35,2){\smsize{$\bf 2$}}
\rput(30,-2){$\Ga_{c;e}$}
\end{pspicture}
\caption{The two sub-graphs of the last graph in Figure~\ref{loopgraph_fig}.}
\label{splitgraph_fig2}
\end{figure}

\noindent
The argument in this section extends to products of projective spaces and concavex sheaves~\e_ref{gensheaf_e} as described in \cite[Section~6]{GWvsSQ}.

\section{Polynomiality for stable quotients} 
\label{MPCpf_sec}

\noindent
In this section, we use the classical localization theorem~\cite{ABo}
to show that the generating functions $\hb^{m-2}\dot\cZ_{n;\a;m'}^{(\b,\vp)}$ and 
$\hb^{m-2}\ddot\cZ_{n;\a;m'}^{(\b,\vp)}$ defined in~\e_ref{cZgendfn_e} 
satisfy specific mutual polynomiality
conditions of Definition~\ref{MPC_dfn} with respect to the generating function 
$\dot\cZ_{n;\a}$ defined in~\e_ref{cZdfn_e}.
The argument is similar to the proof in \cite[Section~7]{GWvsSQ} of self-polynomiality
for the generating function $\dot\cZ_{n;\a}$ defined in~\e_ref{cZdfn_e},
but requires some modifications.

\subsection{Proof of Lemma~\ref{polgen_lmm}} 
\label{MPCpf_subs}

\noindent
The proof involves applying the classical localization theorem~\cite{ABo} with $(n\!+\!1)$-torus
$$\wt\T\equiv\C^*\times\T,$$
where $\T=(\C^*)^n$ as before.
We denote the weight of the standard action of the one-torus $\C^*$ on $\C$ by~$\hb$.
Thus, by Section~\ref{equivsetup_subs},
$$H_{\C^*}^*\approx\Q[\hb], \quad H_{\wt\T}^*\approx\Q[\hb,\al_1,\ldots,\al_n]
.$$
Throughout this section, $V\!=\!\C\!\oplus\!\C$ denotes the representation
of $\C^*$ with the weights $0$ and~$-\hb$.
The induced action on $\P V$ has two fixed points:
$$q_1\equiv[1,0], \qquad q_2\equiv[0,1].$$
With $\ga_1\!\lra\!\P V$ denoting the tautological line bundle,
\BE{hbaract_e}
\E(\ga_1^*)\big|_{q_1}=0, \quad \E(\ga_1^*)\big|_{q_2}=-\hb,
\quad \E(T_{q_1}\P V)=\hb, \quad \E(T_{q_2}\P V)=-\hb;\EE
this follows from our definition of the weights in \cite[Section~3]{GWvsSQ}.\\

\noindent
For each $d\!\in\!\Z^{\ge0}$, the action of $\wt\T$ on $\C^n\!\otimes\!\Sym^dV^*$ induces
an action on
$$\ov\X_d\equiv\P\big(\C^n\!\otimes\!\Sym^dV^*\big).$$
It has $(d\!+\!1)n$ fixed points:
$$P_i(r)\equiv \big[\ti{P}_i\otimes u^{d-r}v^r\big], \qquad
i\in[n],~r\in\{0\}\!\cup\![d],$$
if $(u,v)$ are the standard coordinates on $V$ and $\ti{P}_i\!\in\!\C^n$ is
the $i$-th coordinate vector (so that $[\ti{P}_i]\!=\!P_i\!\in\!\P^{n-1}$).
Let
$$\Om\equiv \E(\ga^*)\in H_{\wt\T}^*\big(\ov\X_d\big)$$
denote the equivariant hyperplane class.\\

\noindent
For all $i\!\in\![n]$ and $r\in\{0\}\!\cup\![d]$,
\begin{equation}\label{Xrestr_e}
\Om|_{P_i(r)}=\al_i\!+\!r\hb, \qquad
\E(T_{P_i(r)}\ov\X_d)=\Bigg\{\underset{(s,k)\neq(r,i)}{\prod_{s=0}^d\prod_{k=1}^n}
(\Om\!-\!\al_k\!-\!s\hb)\Bigg\}\bigg|_{\Om=\al_i+r\hb}.
\end{equation}
Since
\begin{gather*}
B\ov\X_d=\P\big(B(\C^n\!\otimes\!\Sym^dV^*)\big)\lra B\wt\T \qquad\hbox{and}\\
c\big(B(\C^n\!\otimes\!\Sym^dV^*)\big)
=\prod_{s=0}^d\prod_{k=1}^n\big(1-(\al_k\!+\!s\hb)\big)
\in H^*\big(B\wt\T),
\end{gather*}
the $\wt\T$-equivariant cohomology of $\ov\X_d$ is given~by
\begin{equation*}\begin{split}
H_{\ti\T}^*\big(\ov\X_d\big)&\equiv H^*\big(B\ov\X_d\big)
=H^*\big(B\wt\T\big)\big[\Om\big]\Big/
 \prod_{s=0}^d\prod_{k=1}^n\big(\Om-(\al_k\!+\!s\hb)\big)\\
&\approx \Q\big[\Om,\hb,\al_1,\ldots,\al_n\big]\Big/
 \prod_{s=0}^d\prod_{k=1}^n\big(\Om-\al_k-\!s\hb\big)\\
&\subset \Q_{\al}[\hb,\Om]\Big/
 \prod_{s=0}^d\prod_{k=1}^n\big(\Om-\al_k-s\hb\big).
\end{split}\end{equation*}
In particular, every element of $H_{\ti\T}^*(\ov\X_d)$ is a polynomial in
$\Om$ with coefficients in $\Q_{\al}[\hb]$ of degree at most $(d\!+\!1)n\!-\!1$.\\

\noindent
For each $d\!\in\!\Z^{\ge0}$, let 
\BE{Xdfn_e}
\X_d'=\big\{b\!\in\!\ov{Q}_{0,m}\big(\P V\!\times\!\P^{n-1},(1,d)\big)\!:
\ev_1(b)\!\in\!q_1\!\times\!\P^{n-1},~\ev_2(b)\!\in\!q_2\!\times\!\P^{n-1}\big\}.\EE
A general element of $b$ of $\X_d'$ determines a morphism
$$(f,g)\!:\P^1\lra(\P V,\P^{n-1}),$$
up to an automorphism  of the domain~$\P^1$.
Thus, the morphism
$$g\circ f^{-1}\!: \P V\lra \P^{n-1}$$
is well-defined and determines an element $\th(b)\!\in\!\ov\X_d$.
By \cite[Section~7]{GWvsSQ}, this morphism extends to a $\wt\T$-equivariant morphism
$$\th\!=\!\th_d\!:\X_d'\lra\ov\X_d.~\footnote{This morphism is the composition of
the morphism $\th_d$ defined in~\cite{GWvsSQ} in the $m\!=\!2$ case with
the forgetful morphism 
$$\ov{Q}_{0,m}\big(\P V\!\times\!\P^{n-1},(1,d)\big)\lra
\ov{Q}_{0,2}\big(\P V\!\times\!\P^{n-1},(1,d)\big).$$}
$$\\

\noindent
If $d\!\in\!\Z^+$, there is also a natural forgetful morphism
$$F\!: \X_d'\lra \ov{Q}_{0,m}\big(\P^{n-1},d\big),$$
which drops the first sheaf in the pair and contracts one component of the domain if necessary.
If in addition $m\!\ge\!m'\!\ge\!2$, $f_{m',m}$ is as in~\e_ref{fmmdfn_e}, and
$\V_{n;\a}^{(d)}$ is as in~\e_ref{Vstandfn_e},  let
$$ \V_{n;\a;m'}^{(d)}=f_{m',m}^*\V_{n;\a}^{(d)}\lra\ov{Q}_{0,m}(\Pn,d)\,.$$
From the usual short exact sequence for the restriction along $\si_1$, we find that
\BE{VvsVpr_e}
\E\big(\V_{n;\a;m'}^{(d)}\big)
=\lr\a \ev_1^*\x^{\ell(\a)}\E\big(\dot\V_{n;\a;m'}^{(d)}\big)
\in H_{\T}^*\big(\ov{Q}_{0,m}(\Pn,d)\big).\EE
In the case $d\!=\!0$, we set
\begin{equation*}\begin{split}
F^*\E(\V_{n;\a;m'}^{(0)})&=\lr\a\ev_1^*\big(1\!\times\!\x^{\ell(\a)}\big)\in 
H^*\big( \ov{Q}_{0,m}(\P V\!\times\!\P^{n-1},(1,0))\big),\\
F^*\E(\ddot\V_{n;\a;m'}^{(0)})&=1\in 
H^*\big( \ov{Q}_{0,m}(\P V\!\times\!\P^{n-1},(1,0))\big).
\end{split}\end{equation*}

\begin{lmm}\label{PhiZstr_lmm}
Let $l\!\in\!\Z^{\ge0}$, $m,m',n\!\in\!\Z^+$ with $m\!\ge\!m'\!\ge\!2$, and $\a\!\in\!(\Z^*)^l$.
With $\dot\cZ_{n;\a},\dot\cZ_{n;\a;m'}^{(\b,\vp)},\ddot\cZ_{n;\a;m'}^{(\b,\vp)}$ 
as in~\e_ref{cZdfn_e} and~\e_ref{cZgendfn_e}, 
\BE{PhiZstr_e}\begin{split}
&(-\hb)^{m-2}\Phi_{\dot\cZ_{n;\a;m'},\ch\cZ_{n;\a;m'}^{(\b,\vp)}}^{\ch\eta}(\hb,z,q) \\
&\hspace{1.5in}=\sum_{d=0}^{\i}q^d\!\!
\int_{\X_d'}\!\!\!\ne^{(\th^*\Om)z}F^*\E(\ch\V_{n;\a;m'}^{(d)})\,
\psi_2^{b_2}\ev_2^*\vp_2
\prod_{j=3}^{j=m}\!\!\psi_j^{b_j}\ev_j^*(\E(\ga_1^*)\vp_j).
\end{split}\EE
with $(\ch\cZ,\ch\V,\ch\eta)=(\dot\cZ,\V,\dot\eta),(\ddot\cZ,\ddot\V,\ddot\eta)$.
\end{lmm}

\noindent
Since the right-hand sides of the above expressions lie in 
$H_{\wt\T}^*[[z,q]]\subset\Q_{\al}[\hb][[z,q]]$,
this lemma is a more precise version of Lemma~\ref{polgen_lmm}.

\subsection{Proof of Lemma~\ref{PhiZstr_lmm}}
\label{PolStrPf_subs}

\noindent
We apply the localization theorem of~\cite{ABo} to the $\wt\T$-action on~$\X_d'$.
We show that each fixed locus of the $\wt\T$-action on $\X_d'$ contributing 
to the right-hand sides in~\e_ref{PhiZstr_e}
corresponds to a pair $(\Ga_1,\Ga_2)$ of decorated graphs as in~\e_ref{decortgraphdfn_e}, 
with $\Ga_1$ and $\Ga_2$ contributing to
the two generating functions in the subscript of the corresponding correlator~$\Phi$
evaluated at $x\!=\!\al_i$ for some $i\!\in\![n]$.\\

\noindent
Similarly to Section~\ref{recpf_sec}, the fixed loci of the $\wt\T$-action on
$\ov{Q}_{0,m}(\P V\!\times\!\P^{n-1},(d',d))$
correspond to decorated graphs $\Ga$ with $m$ marked points distributed between the ends
of~$\Ga$.
The map~$\d$ should now take values in pairs of nonnegative integers, indicating the degrees of
the two subsheaves.
The map~$\mu$ should similarly take values in the pairs~$(i,j)$
with $i\!\in\![2]$ and $j\!\in\![n]$,
indicating the fixed point~$(q_i,P_j)$ to which the vertex is mapped.
The $\mu$-values on consecutive vertices must differ by precisely one of the two components.\\

\noindent
The situation for the  $\wt\T$-action on
$$\X_d'\subset\ov{Q}_{0,m}\big(\P V\!\times\!\P^{n-1},(1,d)\big)$$
is simpler, however.
There is a unique edge of positive $\P V$-degree; we draw it as a thick line
in Figure~\ref{Xlocus_fig}.
The first component of the value of~$\d$ on all other edges
and on all vertices must be~0; so we drop~it.
The first component of the value of~$\mu$ on the vertices changes only when the thick
edge is crossed.
Thus, we drop the first components of the vertex labels as well,
with the convention that these components are~1 on the left side of the thick
edge and~2 on the right.
In particular, the vertices to the left of the thick edge (including the left endpoint)
lie in $q_1\!\times\!\P^{n-1}$ and the vertices to its right lie in $q_2\!\times\!\P^{n-1}$.
Thus, by~\e_ref{Xdfn_e}, the marked point~1 is attached to a vertex to the left of
the thick edge and the marked point~2 is attached to a vertex to the right.
By the localization formula~\e_ref{ABothm_e} and
the first equation in~\e_ref{hbaract_e}, $\Ga$ does not contribute to the right-hand
sides in~\e_ref{PhiZstr_e} unless the marked points indexed by $j\!\ge\!3$ are also
attached to vertices to the right of the thick edge.
Finally, the remaining, second component of~$\mu$ takes the same value $i\!\in\![n]$
on the two vertices of the thick edge.\\

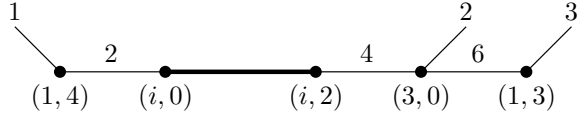
\begin{figure}
\begin{pspicture}(0.3,-.8)(10,1)
\psset{unit=.4cm}
\psline[linewidth=.15](17,0)(22,0)
\pscircle*(17,0){.2}\pscircle*(22,0){.2}
\psline[linewidth=.04](17,0)(13.5,0)\pscircle*(13.5,0){.2}
\psline[linewidth=.04](13.5,0)(12,1.5)\rput(12,2){\smsize{$1$}}
\rput(15.2,.6){\smsize{2}}\rput(13.5,-.85){\smsize{$(1,4)$}}
\rput(17,-.85){\smsize{$(i,0)$}}
\psline[linewidth=.04](22,0)(25.5,0)\pscircle*(25.5,0){.2}
\psline[linewidth=.04](25.5,0)(29,0)\pscircle*(29,0){.2}
\psline[linewidth=.04](25.5,0)(27,1.5)\rput(27,2){\smsize{$2$}}
\psline[linewidth=.04](29,0)(30.5,1.5)\rput(30.5,2){\smsize{$3$}}
\rput(22,-.85){\smsize{$(i,2)$}}\rput(25.5,-.85){\smsize{$(3,0)$}}
\rput(29,-.85){\smsize{$(1,3)$}}
\rput(23.7,.6){\smsize{4}}\rput(27.4,.6){\smsize{6}}
\end{pspicture}
\caption{A graph representing a fixed locus in $\X_d'$; $i\!\neq\!1,3$}
\label{Xlocus_fig}
\end{figure}

\noindent
Let $\cA_i$ denote the set of graphs as above so that the $\mu$-value on
the two endpoints of the thick edge is labeled~$i$; see Figure~\ref{Xlocus_fig}.
We break each graph $\Ga\!\in\!\cA_i$ into three sub-graphs:
\begin{enumerate}[label=(\roman*)]
\item $\Ga_1$ consisting of all vertices of $\Ga$ to the left of the thick edge,
including its left vertex~$v_1$ with its $\d$-value,
and a new marked point attached to~$v_1$;
\item $\Ga_0$ consisting of the thick edge~$e_0$, its two vertices $v_1$ and $v_2$,
with $\d$-values set to~0, and
new marked points $1$ and $2$ attached to $v_1$ and $v_2$, respectively;
\item $\Ga_2$ consisting of all vertices to the right of the thick edge,
including its right vertex~$v_2$ with its $\d$-value, and a new marked point attached to~$v_2$;
\end{enumerate}
see Figure~\ref{Xsplit_fig}.
From~\e_ref{Zlocus_e}, we then obtain a splitting of
the fixed locus in $\X_d'$ corresponding to~$\Ga$:
\BE{Xlocus_e1}
Q_{\Ga}\approx Q_{\Ga_1}\times Q_{\Ga_0}\times Q_{\Ga_2}
\subset\ov{Q}_{0,2}(\P^{n-1},|\Ga_1|)\times
\ov{Q}_{0,2}(\P V,1)\times\ov{Q}_{0,m}(\P^{n-1},|\Ga_2|).\EE
The exceptional cases are $|\Ga_1|\!=\!0$ and $m\!=\!2,|\Ga_2|\!=\!0$;
the above isomorphism then holds with the corresponding component replaced by a point.\\

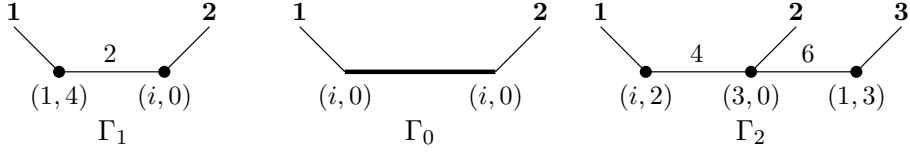
\begin{figure}
\begin{pspicture}(0,-.8)(10,1)
\psset{unit=.4cm}
\psline[linewidth=.04](7.5,0)(11,0)\rput(9.2,.6){\smsize{2}}
\pscircle*(7.5,0){.2}\rput(7.5,-.85){\smsize{$(1,4)$}}
\pscircle*(11,0){.2}\rput(11,-.85){\smsize{$(i,0)$}}
\psline[linewidth=.04](7.5,0)(6,1.5)\rput(6,2){\smsize{$\bf 1$}}
\psline[linewidth=.04](11,0)(12.5,1.5)\rput(12.5,2){\smsize{$\bf 2$}}
\rput(9.3,-2){$\Ga_1$}
\psline[linewidth=.04](17,0)(15.5,1.5)\rput(15.5,2){\smsize{$\bf 1$}}
\psline[linewidth=.04](22,0)(23.5,1.5)\rput(23.5,2){\smsize{$\bf 2$}}
\psline[linewidth=.15](17,0)(22,0)
\rput(17,-.85){\smsize{$(i,0)$}}\rput(22,-.85){\smsize{$(i,0)$}}
\rput(19.5,-2){$\Ga_0$}
\psline[linewidth=.04](27,0)(30.5,0)\pscircle*(30.5,0){.2}
\psline[linewidth=.04](30.5,0)(34,0)\pscircle*(34,0){.2}
\pscircle*(27,0){.2}\rput(27,-.85){\smsize{$(i,2)$}}
\rput(30.5,-.85){\smsize{$(3,0)$}}\rput(34,-.85){\smsize{$(1,3)$}}
\psline[linewidth=.04](27,0)(25.5,1.5)\rput(25.5,2){\smsize{$\bf 1$}}
\psline[linewidth=.04](30.5,0)(32,1.5)\rput(32,2){\smsize{$\bf 2$}}
\psline[linewidth=.04](34,0)(35.5,1.5)\rput(35.5,2){\smsize{$\bf 3$}}
\rput(28.7,.6){\smsize{4}}\rput(32.4,.6){\smsize{6}}
\rput(30.5,-2){$\Ga_2$}
\end{pspicture}
\caption{The three sub-graphs of the graph in Figure~\ref{Xlocus_fig}}
\label{Xsplit_fig}
\end{figure}

\noindent
Let $\pi_1$, $\pi_0$, and $\pi_2$ denote the three component projection maps in~\e_ref{Xlocus_e1}.
By~\e_ref{NZform_e}, 
\BE{bndsplit3_e}
\frac{\E(\N Q_{\Ga})}{\E(T_{P_i}\Pn)}=
\pi_1^*\bigg(\frac{\E(\N Q_{\Ga_1})}{\E(T_{P_i}\P^{n-1})}\bigg)
\cdot\pi_2^*\bigg(\frac{\E(\N Q_{\Ga_2})}{\E(T_{P_i}\P^{n-1})}\bigg)
\cdot \big(\om_{e_0;v_1}-\pi_1^*\psi_2\big) \big(\om_{e_0;v_2}-\pi_2^*\psi_1\big).\EE
Since for every $j\!=\!m'\!+\!1,\ldots\!m$ the closest vertex of $\Ver_{m'}$ lies
to the right of the thick edge, by~\e_ref{cVform_e} and~\e_ref{VvsVpr_e}, 
\BE{bndsplit4_e}\begin{split}
F^*\E\big(\V_{n;\a;m'}^{(|\Ga|)}\big)\big|_{Q_{\Ga}}&=
\dot\eta(\al_i)\pi_1^*\E\big(\ddot\V_{n;\a}^{(|\Ga_1|)}\big)
\pi_2^*\E(\dot\V_{n;\a;m'}^{(|\Ga_2|)}\big)\,,\\
F^*\E\big(\ddot\V_{n;\a;m'}^{(|\Ga|)}\big)\big|_{Q_{\Ga}}&=
\ddot\eta(\al_i)\pi_1^*\E\big(\ddot\V_{n;\a}^{(|\Ga_1|)}\big)
\pi_2^*\E(\ddot\V_{n;\a;m'}^{(|\Ga_2|)}\big).
\end{split}\EE
Since $Q_{\Ga_0}$ consists of a degree~1 map,
by the last two identities in~\e_ref{hbaract_e}
\BE{omrestr_e} \om_{e_0;v_1}=\hb, \qquad \om_{e_0;v_2}=-\hb.\EE
The morphism $\th$ takes the locus $Q_{\Ga}$ to a fixed point $P_k(r)\!\in\!\ov\X_d$.
It is immediate that $k\!=\!i$.
By continuity considerations, $r\!=\!|\Ga_1|$.
Thus, by the first identity in \e_ref{Xrestr_e},
\BE{thOmrestr_e}\th^*\Om\big|_{Q_{\Ga}}=\al_i+|\Ga_1|\hb.\EE
Combining \e_ref{bndsplit3_e}-\e_ref{thOmrestr_e} and the second equation in~\e_ref{hbaract_e}, 
we obtain
\BE{Xbndlsplit_e5}\begin{split}
&q^{|\Ga|}\int_{Q_{\Ga}}\!\!\! \frac{\ne^{(\th^*\Om)z}F^*\E(\V_{n;\a;m'}^{(|\Ga|)})\,
\psi_2^{b_2}\ev_2^*\vp_2
\prod\limits_{j=3}^{j=m}\!\!\psi_j^{b_j}\ev_j^*(\E(\ga_1^*)\vp_j)}{\E(\N Q_{\Ga})}\\
&\hspace{.8in}=(-\hb)^{m-2}
\frac{\dot\eta(\al_i)\ne^{\al_iz}}{\prod\limits_{k\neq i}(\al_i\!-\!\al_k)}
\Bigg\{\ne^{|\Ga_1|\hb z}q^{|\Ga_1|}\!\! \int_{Q_{\Ga_1}} \!\!\!\!\!\!
\frac{\E(\ddot\V_{n;\a}^{(|\Ga_1|)})\ev_2^*\phi_i}{\hb\!-\!\psi_2}
\Big|_{Q_{\Ga_1}}\frac{1}{\E(\N Q_{\Ga_1})}\Bigg\}\\
&\hspace{1.7in}
\times \Bigg\{q^{|\Ga_2|}\int_{Q_{\Ga_2}} \!\!\!\!\!
\frac{\E(\dot\V_{n;\a;m'}^{(|\Ga_2|)})\ev_1^*\phi_i
\prod\limits_{j=2}^{j=m}(\psi_j^{b_j}\ev_j^*\vp_j)}
{(-\hb)\!-\!\psi_1}
\Big|_{Q_{\Ga_2}}\frac{1}{\E(\N Q_{\Ga_2})}\Bigg\}.
\end{split}\EE
This identity remains valid with $|\Ga_1|\!=\!0$ and/or $m\!=\!2,|\Ga_2|\!=\!0$
if we set the corresponding integral to~1 or to~$\hb^{b_2}\vp_2|_{P_i}$, 
respectively.\\

\noindent
We now sum up the last identity over all $\Ga\!\in\!\cA_i$.
This is the same as summing over all pairs $(\Ga_1,\Ga_2)$ of decorated graphs such that
\begin{enumerate}[label=(\arabic*)]
\item $\Ga_1$ is a 2-pointed graph of degree $d_1\!\ge\!0$ such that the marked point~2
is attached to a vertex labeled~$i$;
\item $\Ga_2$ is an $m$-pointed graph of degree $d_2\!\ge\!0$ such that the marked point~1
is attached to a vertex labeled~$i$.
\end{enumerate}
By the localization formula~\e_ref{ABothm_e} and symmetry,
\begin{equation*}\begin{split}
&1+\sum_{\Ga_1}(q\ne^{\hb z})^{|\Ga_1|}\Bigg\{\! \int_{Q_{\Ga_1}} \!\!
\frac{\E(\ddot\V_{n;\a}^{(|\Ga_1|)})\ev_2^*\phi_i}{(\hb\!-\!\psi_2)\E(\N Q_{\Ga_1})}
\Bigg\}\\
&\hspace{1in}=1+\sum_{d=1}^{\i}(q\ne^{\hb z})^d
\int_{\ov{Q}_{0,2}(\P^{n-1},d)}\frac{\E(\ddot\V_{n;\a}^{(d)})\ev_2^*\phi_i}{\hb\!-\!\psi_2}
=\dot\cZ_{n;\a}\big(\al_i,\hb,q\ne^{\hb z}\big);\\
&\de_{m,2}\hb^{b_2}\vp_2|_{P_i}
+\sum_{\Ga_2}q^{|\Ga_2|}\Bigg\{\int_{Q_{\Ga_2}} \!\!\!\
\frac{\E(\dot\V_{n;\a;m'}^{(|\Ga_2|)})\ev_1^*\phi_i
\prod\limits_{j=2}^{j=m}(\psi_j^{b_j}\ev_j^*\vp_j)}
{\E(\N Q_{\Ga_2})(-\hb\!-\!\psi_1)}\Bigg\}\\
&\quad=\de_{m,2}\hb^{b_2}\vp_2|_{P_i}+\sum_{d=\max(3-m,0)}^{\i}\!\!\!\!\!\!q^d\!\!
\int_{\ov{Q}_{0,m}(\P^{n-1},d)}\!\!\!
\frac{\E(\dot\V_{n;\a;m'}^{(|\Ga_2|)})\ev_1^*\phi_i
\prod\limits_{j=2}^{j=m}(\psi_j^{b_j}\ev_j^*\vp_j)}
{(-\hb\!-\!\psi_1)}
=\dot\cZ_{n;\a;m'}^{(\b,\vp)}\big(\al_i,-\hb,q\big).
\end{split}\end{equation*}
Combining with this with~\e_ref{ABothm_e}, we obtain
\begin{equation*}\begin{split}
&\sum_{d=0}^{\i}q^d\!\!
\int_{\X_d'}\ne^{(\th^*\Om)z}F^*\E(\V_{n;\a;m'}^{(d)})\,
\psi_2^{b_2}\ev_2^*\vp_2
\prod_{j=3}^{j=m}\psi_j^{b_j}\ev_j^*(\E(\ga_1^*)\vp_j)\\
&\hspace{.5in}
=(-\hb)^{m-2}
\sum_{i=1}^n\frac{\dot\eta(\al_i)\ne^{\al_iz}}{\prod\limits_{k\neq i}(\al_i\!-\!\al_k)}
\dot\cZ_{n;\a}\big(\al_i,\hb,q\ne^{\hb z}\big)\dot\cZ_{n;\a;m'}^{(\b,\vp)}\big(\al_i,-\hb,q\big)\\
&\hspace{.5in}
=(-\hb)^{m-2}\Phi_{\dot\cZ_{n;\a},\dot\cZ_{n;\a;m'}^{(\b,\vp)}}^{\dot\eta}(\hb,z,q),
\end{split}\end{equation*}
as claimed in the $\dot\cZ$ identity in~\e_ref{PhiZstr_e}.\\

\noindent
From  \e_ref{bndsplit3_e}-\e_ref{thOmrestr_e}, we also find that 
\e_ref{Xbndlsplit_e5} holds with $\V$ and $\dot\V$ replaced by $\ddot\V$ and
$\dot\eta$ by~$\ddot\eta$, with the same conventions in the $|\Ga_1|\!=\!0$ and
$m\!=\!2,|\Ga_2|\!=\!0$ cases.
We then sum up the resulting identity over all pairs $(\Ga_1,\Ga_2)$ of decorated graphs as
in the previous paragraph.
The sum of the terms in the first curly brackets over all possibilities for $\Ga_1$ 
is exactly the same as before, while the sum of the terms in the second curly brackets 
over all possibilities for $\Ga_2$ is described by the same expression as before
with $\dot\V_{n;\a;m'}^{(|\Ga_2|)}$ and $\dot\cZ_{n;\a;m'}^{(\b,\vp)}$ replaced by 
$\ddot\V_{n;\a;m'}^{(|\Ga_2|)}$ and $\ddot\cZ_{n;\a;m'}^{(\b,\vp)}$, respectively.
Thus,
\begin{equation*}\begin{split}
&\sum_{d=0}^{\i}q^d\!\!
\int_{\X_d'}\ne^{(\th^*\Om)z}F^*\E(\ddot\V_{n;\a;m'}^{(d)})\,
\psi_2^{b_2}\ev_2^*\vp_2
\prod_{j=3}^{j=m}\psi_j^{b_j}\ev_j^*(\E(\ga_1^*)\vp_j)\\
&\hspace{.5in}
=(-\hb)^{m-2}
\sum_{i=1}^n\frac{\ddot\eta(\al_i)\ne^{\al_iz}}{\prod\limits_{k\neq i}(\al_i\!-\!\al_k)}
\dot\cZ_{n;\a}\big(\al_i,\hb,q\ne^{\hb z}\big)\ddot\cZ_{n;\a;m'}^{(\b,\vp)}\big(\al_i,-\hb,q\big)\\
&\hspace{.5in}=
(-\hb)^{m-2}\Phi_{\dot\cZ_{n;\a},\ddot\cZ_{n;\a;m'}^{(\b,\vp)}}^{\ddot\eta}(\hb,z,q),
\end{split}\end{equation*}
as claimed in the $\ddot\cZ$ identity in~\e_ref{PhiZstr_e}.\\

\noindent
In the case of products of projective spaces and concavex sheaves~\e_ref{gensheaf_e},
the spaces
$$\ov{Q}_{0,m}(\P V\!\times\!\Pn,(1,d)) \qquad\hbox{and}\qquad
\ov\X_d=\P\big(\C^n\!\otimes\!\Sym^dV^*\big) $$
are replaced~by
$$\ov{Q}_{0,m}\big(\P V\!\times\!\P^{n_1-1}\!\times\!\ldots\!\times\!\P^{n_p-1},
(1,d_1,\ldots,d_p)\big)  \quad\hbox{and}\quad
\P\big(\C^{n_1}\!\otimes\!\Sym^{d_1}V^*\big)\!\!\times\!\ldots\!\times\!
\P\big(\C^{n_p}\!\otimes\!\Sym^{d_p}V^*\big),$$
respectively.
Lemma~\ref{PhiZstr_lmm} extends to this situation by replacing $z$ and $q$ 
 in~\e_ref{PhiZstr_e}
by $z_1,\ldots,z_p$ and $q_1,\ldots,q_p$,
$q^d$ by $q_1^{d_1}\!\ldots\!q_p^{d_p}$, 
$\X_d'$ by $\X_{d_1,\ldots,d_p}'$,
$\ne^{(\th^*\Om z)}$ by $\ne^{(\th^*\Om_1)z_1+\ldots+(\th^*\Om_p)z_p}$,
and the indices $d$ and $n$ on the bundles $\V,\ddot\V$ by $(d_1,\ldots,d_p)$
and $(n_1,\ldots,n_p)$, and summing over $d_1,\ldots,d_p\!\ge\!0$ instead of $d\!\ge\!0$.
The vertices of the thick edge in Figure~\ref{Xlocus_fig}
are now labeled by a tuple $(i_1,\ldots,i_p)$ with $i_s\!\in\![n_s]$,
as needed for the extension of Definition~\ref{MPC_dfn}
described at the end of Section~\ref{poliC_subs}.
The relation~\e_ref{thOmrestr_e} becomes
$$\th^*\Om_s\big|_{Q_{\Ga}}=\al_{s;i_s}+|\Ga_1|_s\hb\,,$$
where $|\Ga_1|_s$ is the sum of the $s$-th components of the values of~$\d$
on the vertices and edges of~$\Ga_1$ (corresponding to the degree of the maps to~$\P^{n_s-1}$).
Otherwise, the proof is identical.

\section{Stable quotients vs.~Hurwitz numbers}
\label{equiv0pf_sec}

\noindent
Our proof of Propositions~\ref{equiv0_prp} and~\ref{equiv0_prp2} that describe twisted 
Hurwitz numbers on $\ov\cM_{0,3|d}$ is analogous to the proof of \cite[Theorem~4]{GWvsSQ},
which describes similar integrals on~$\ov\cM_{0,2|d}$.
In particular, we show that it is sufficient to verify the statements of 
Propositions~\ref{equiv0_prp} and~\ref{equiv0_prp2} for each fixed $\a$ and for all $n$
sufficiently large (compared to~$|\a|$).
For $\nu_n(\a)\!>\!0$, we obtain the statements of Propositions~\ref{equiv0_prp} 
and~\ref{equiv0_prp2}
by analyzing the secondary (middle) terms in the recursion~\e_ref{recurdfn_e2}
for the three-point generating functions $\dot\cZ_{n;\a;3}^{(\0,\1)}$ and
$\dot\cZ_{n;\a;2}^{(\0,\1)}$ defined in~\e_ref{cZ3dfn_e0} and~\e_ref{cZ3dfn_e}, 
respectively.
We also use \e_ref{cZstr_e} and~\e_ref{cZ3vsY_e}.
The latter is the string equation for stable quotients invariants;
in Proposition~\ref{cZstring_prp2}, we show that it is equivalent to 
Proposition~\ref{equiv0_prp2}
whenever $\nu_(\a)\!\ge\!0$.
In Proposition~\ref{cZstring_prp}, we show that~\e_ref{cZ3vsY_e} is equivalent 
to Proposition~\ref{equiv0_prp} whenever $\nu_n(\a)\!\ge\!0$.
We confirm Proposition~\ref{equiv0_prp} whenever $\nu_n(\a)\!>\!0$ using 
Proposition~\ref{uniqueness_prp}; see Corollary~\ref{MirSym_crl}.
Since it is sufficient to verify the statement of Proposition~\ref{equiv0_prp} with 
$\nu_n(\a)\!>\!0$, the $\nu_n(\a)\!=\!0$ case of Proposition~\ref{equiv0_prp} then concludes
the proof of~\e_ref{cZ3vsY_e}.

\subsection{Proof of Propositions~\ref{Z3equiv_prp}, \ref{equiv0_prp}, and~\ref{equiv0_prp2}}
\label{Hurwpf_subs}

\noindent
With $n$ and $\a$ as in Propositions~\ref{equiv0_prp} and~\ref{equiv0_prp2} and 
$b_1,b_2,b_3,r\!\in\!\Z^{\ge0}$, let
\begin{equation*}\begin{split}
\F_{n;\a}^{(b_1,b_2,b_3)}(\al_i,q)&=
\sum_{d=0}^{\i}\frac{q^d}{d!}
\int_{\ov\cM_{0,3|d}}\frac{\E(\dot\V_{\a}^{(d)}(\al_i))\psi_1^{b_1}\psi_2^{b_2}\psi_3^{b_3}}
{\prod\limits_{k\neq i}\!\!\E(\dot\V_1^{(d)}(\al_i\!-\!\al_k))},\\
\F_{n;\a;r}^{(b_1,b_2,b_3)}(\al_i,q)&=
\sum_{d=0}^{\i}\frac{q^d}{d!}
\int_{\ov\cM_{0,3|d}}\frac{\E(\dot\V_{\a;r}^{(d)}(\al_i))\psi_1^{b_1}\psi_2^{b_2}\psi_3^{b_3}}
{\prod\limits_{k\neq i}\!\!\E(\dot\V_1^{(d)}(\al_i\!-\!\al_k))}\,.
\end{split}\end{equation*}
By \cite[Remark 8.5]{GWvsSQ},
\BE{Fred_e}\F_{n;\a}^{(b_1,b_2,b_3)}(\al_i,q)=\frac{\xi_{n;\a}(\al_i,q)^{b_1+b_2+b_3}}{b_1!b_2!b_3!}
\F_{n;\a}^{(0,0,0)}(\al_i,q);\EE
thus, it is sufficient to show that 
\BE{equiv0thm_e}
\F_{n;\a}^{(0,0,0)}(\al_i,q)
=\frac{1}{\dot\Phi_{n;\a}^{(0)}(\al_i,q)}\,.\EE
By the same reasoning as in \cite[Remarks 8.4,8.5]{GWvsSQ},
$$\F_{n;\a;r}^{(b_1,b_2,b_3)}(\al_i,q)=\frac{\xi_{n;\a}(\al_i,q)^{b_1+b_2}}{b_1!b_2!}
\F_{n;\a;r}^{(0,0,b_3)}(\al_i,q);$$
thus, it is sufficient to show that 
\BE{equiv0thm2_e}\sum_{b=0}^{\i}\sum_{r=0}^{\i}\F_{n;\a;r}^{(0,0,b)}(\al_i,q)
\Rs{\hb=0}\Bigg\{\frac{(-1)^b}{\hb^{b+1}}\LR{\dot\cY_{n;\eset}(\al_i,\hb,q)}_{q;r}q^r\Bigg\}
=1.\EE

\begin{crl}\label{MirSym_crl}
Let $l\!\in\!\Z^{\ge0}$, $n\!\in\!\Z^+$, and $\a\!\in\!(\Z^*)^l$.
If $\nu_n(\a)\!>\!0$,
$$\dot\cZ_{n;\a;3}^{(\0,\1)}(\x,\hb,q)=
\hb^{-1}\dot\cZ_{n;\a}(\x,\hb,q)
\in H_{\T}^*(\Pn)\big[\big[\hb^{-1},q\big]\big].$$
\end{crl}

\begin{proof}
By Lemma~\ref{Phistr_lmm4}\ref{mult_ch} and Lemmas~\ref{recgen_lmm} and~\ref{polgen_lmm},
the series $\hb\dot\cZ_{n;\a;3}^{(\0,\1)}(\x,\hb,q)$ and $\dot\cZ_{n;\a}(\x,\hb,q)$
are $\fC$-recursive and satisfy the $\dot\eta$-MPC with respect to~$\dot\cZ_{n;\a}(\x,\hb,q)$,
no matter what $n$ and $\a$ are.
It is immediate that 
$$\dot\cZ_{n;\a}(\x,\hb,q)\cong1 \quad\mod\hb^{-1}\,.$$
If $\nu_n(\a)\!>\!0$ and $d\!\in\!\Z^+$,
\begin{equation*}\begin{split}
\dim\ov{Q}_{0,3}(\Pn,d)-\rk\dot\V_{n;\a;3}^{(d)}
&=\nu_n(\a)d+(n\!-\!1)>n\!-\!1=\dim\Pn.
\end{split}\end{equation*}
Thus,
$$\hb\dot\cZ_{n;\a;3}^{(\0,\1)}(\x,\hb,q) 
\cong1 \quad\mod\hb^{-1}\,,$$
whenever $\nu_n(\a)\!>\!0$.
The claim now follows from Proposition~\ref{uniqueness_prp}.
\end{proof}

\begin{prp}\label{cZstring_prp}
If $l\!\in\!\Z^{\ge0}$, $n\!\in\!\Z^+$, and $\a\!\in\!(\Z^*)^l$ are such that $\nu_n(\a)\!\ge\!0$,
then
\BE{equivMS_e2}
\dot\cZ_{n;\a;3}^{(\0,\1)}(\x,\hb,q)=\hb^{-1}\frac{\dot\cZ_{n;\a}(\x,\hb,q)}{\dot{I}_0(q)}
\in \left(H_{\T}^*(\Pn)\right)\big[\big[\hb^{-1},q\big]\big]\EE
if and only if \e_ref{equiv0thm_e} holds for all $i\!\in\![n]$.
\end{prp}

\begin{prp}\label{cZstring_prp2}
If $l\!\in\!\Z^{\ge0}$, $n\!\in\!\Z^+$, and $\a\!\in\!(\Z^*)^l$ are such that $\nu_n(\a)\!\ge\!n$,
then
\BE{equivMS_e4}
\dot\cZ_{n;\a;2}^{(\0,\1)}(\x,\hb,q)=\hb^{-1}\dot\cZ_{n;\a}(\x,\hb,q)
\in \left(H_{\T}^*(\Pn)\right)\big[\big[\hb^{-1},q\big]\big]\EE
if and only if \e_ref{equiv0thm2_e} holds for all $i\!\in\![n]$.
\end{prp}

\noindent
For any $t,t'\!\in\![d]$ with $t\!\neq\!t'$, let 
$\De_{tt'}\!\in\!H^2(\ov\cM_{0,m|d})$ denote the class of the diagonal divisor
$$\big\{[\cC,y_1,\ldots,y_m;\hat{y}_1,\ldots,\hat{y}_d]\in\!\ov\cM_{g,m|d}\!:~
\hat{y}_t\!=\!\hat{y}_{t'}\big\}.$$
For any $t\!\in\![d]$, let
$$\De_t=\sum_{t'>t}\De_{tt'}\,.$$
We denote by $\fs_1,\fs_2,\ldots$ the elementary symmetric polynomials~in 
$$\{\be_k\}=\big\{(\al_i\!-\!\al_k)^{-1}\!:~k\!\neq\!i\big\}$$
for any given number of formal variables~$\be_k$.
Let
$$A_{\a}(\al_i)
=\prod_{a_k>0}\!\big(a_k^{a_k}\al_i^{a_k}\big)\prod_{a_k<0}\!\big(a_k^{-a_k}\al_i^{-a_k}\big),
\qquad
A_{n;\a}(\al_i)
=\frac{A_{\a}(\al_i)}
{\prod\limits_{k\neq i}(\al_i\!-\!\al_k)}.$$

\begin{proof}[{\bf \emph{Proof of \e_ref{equiv0thm_e}}}]
By~(1) in the proof of \cite[Proposition~8.3]{GWvsSQ}, 
\BE{cHnums_e}\begin{split}
\frac{\LR{\F_{n;\a}^{(0,0,0)}(\al_i,q)}_{q;d}}{A_{n;\a}^d(\al_i)}
&=\int_{\ov\cM_{0,3|d}}
\frac{\prod\limits_{a_k>0}\prod\limits_{t=1}^d\prod\limits_{\la=1}^{a_k}
\!\!\left(1\!-\!\frac{\la\hat\psi_t}{a_k\al_i}\!+\!\frac{\De_t}{\al_i}\right)\,
\prod\limits_{a_k<0}\prod\limits_{t=1}^d\prod\limits_{\la=0}^{-a_k-1}
\!\!\left(1\!+\!\frac{\la\hat\psi_t}{a_k\al_i}\!+\!\frac{\De_t}{\al_i}\right)}
{\prod\limits_{k\neq i}\prod\limits_{t=1}^d
\!\!\left(1\!-\!\frac{\hat\psi_t}{\al_i-\al_k}\!+\!\frac{\De_t}{\al_i-\al_k}\right)}\\
&=\cH_{\a;d}\big(\al_i^{-1},\fs_1,\ldots,\fs_d)
\end{split}\EE
for some $\cH_{\a;d}\!\in\!\Q[y,\fs_1,\ldots,\fs_d]$ dependent only on $\a$ and $d$, but not 
on~$n$.\footnote{Whatever polynomial works for $n\!>\!d$ works for all~$n$;
this can be seen by setting the extra $\be_k$'s to~0.}
Similarly, for any $d,d'\!\in\!\Z^{\ge0}$
there exists $\dot\cY_{\a;d,d'}\!\in\!\Q[y,\fs_1,\ldots,\fs_{d'}]$, independent of~$n$,
such~that
\BE{cYnums_e}\LR{\hb^d\LR{\dot\cY_{n;\a}(\al_i,\hb,q)}_{q;d}}_{\hb;d'}
=A_{n;\a}^d(\al_i)\dot\cY_{\a;d,d'}(y,\fs_1,\ldots,\fs_{d'})\,.\EE
Thus, by \e_ref{cYexp_e},  there exist 
$\xi_{\a;d},\dot\Phi_{\a;d}^{(0)}\!\in\!\Q[y,\fs_1,\ldots,\fs_d]$, 
independent of~$n$, such~that 
\begin{equation*}\begin{split}
\LR{\xi_{n;\a}(\al_i,q)}_{q;d}
&\equiv \Rs{\hb=0}\LR{\log\dot\cY_{n;\a}(\al_i,\hb,q)}_{q;d}
=A_{n;\a}^d(\al_i)\xi_{\a;d}\big(\al_i^{-1},\fs_1,\ldots,\fs_{d-1}),\\
\LR{\dot\Phi_{n;\a}^{(0)}(\al_i,q)}_{q;d}
&\equiv \Rs{\hb=0}\frac1{\hb}\!\LR{\ne^{-\frac{\xi_{n;\a}(\al_i,q)}{\hb}}
\dot\cY_{n;\a}(\al_i,\hb,q)}_{q;d}
=A_{n;\a}^d(\al_i)\dot\Phi_{\a;d}^{(0)}\big(\al_i^{-1},\fs_1,\ldots,\fs_d).
\end{split}\end{equation*}
We conclude that  \e_ref{equiv0thm_e} is equivalent~to
$$\sum_{\begin{subarray}{c}d_1,d_2\ge0\\ d_1+d_2=d\end{subarray}}
\cH_{\a;d_1}\dot\Phi_{\a;d_2}^{(0)}=\de_{d,0}  \qquad\forall\,d\in\Z^{\ge0}\,.$$
By Corollary~\ref{MirSym_crl} and Proposition~\ref{cZstring_prp}, 
these relations hold whenever $\nu_n(\a)\!>\!0$; 
since they do not involve~$n$, they thus hold for all pairs $(n,\a)$.
\end{proof}

\begin{proof}[{\bf \emph{Proof of \e_ref{equiv0thm2_e}}}]
For $t\!\in\![d\!+\!1]$ and $r\!\in\!\Z^{\ge0}$, we define $\hat\psi_t',\De_{t;r}'\!\in\!H^2(\ov\M_{0,3|d})$  by
$$\hat\psi_t'=f_{2;3}^*\hat\psi_t,\qquad
\De_{t;r}=f_{2;3}^*\De_t+\begin{cases}(r\!-\!1)f_{2;3}^*\De_{t,d+1} ,&\hbox{if}~t\!\le\!d;\\
0,&\hbox{if}~t\!=\!d\!+\!1.
\end{cases}$$
Similarly to~(1) in the proof of \cite[Proposition~8.3]{GWvsSQ}, 
\begin{equation*}\begin{split}
a_k\!>\!0\quad&\Lra\quad \E(\dot\V_{a_k;r}^{(d)}(\al_i))
=\prod_{t=1}^d\prod_{\la=1}^{a_k}\big(a_k\al_i-\la\hat\psi_t'+a_k\De_{t;r}'\big)
\cdot \prod_{\la=1}^{ra_k}\big(a_k\al_i-\la\hat\psi_{d+1}'\big)\,;\\
a_k\!<\!0\quad&\Lra\quad \E(\dot\V_{a_k;r}^{(d)}(\al_i))
=\prod_{t=1}^d\prod_{\la=0}^{-a_k-1}\big(a_k\al_i+\la\hat\psi_t'+a_k\De_{t;r}'\big)
\cdot \prod_{\la=0}^{-ra_k-1}\big(a_k\al_i+\la\hat\psi_{d+1}'\big).
\end{split}\end{equation*}
Thus, similarly to~\e_ref{cHnums_e}, 
\begin{equation*}\begin{split}
\frac{\LR{\F_{n;\a;r}^{(0,0,b)}(\al_i,q)}_{q;d}}{A_{n;\a}(\al_i)^d A_{\a}(\al_i)^r}
&=\cH_{\a;r;d}^{(b)}\big(\al_i^{-1},\fs_1,\ldots,\fs_d)
\end{split}\end{equation*}
for some $\cH_{\a;r;d}^{(b)}\!\in\!\Q[y,\fs_1,\ldots,\fs_d]$ dependent only on 
$\a$, $r$, $b$, and $d$, but not on~$n$.
Thus, by~\e_ref{cYnums_e} with $\a\!=\!\eset$, \e_ref{equiv0thm2_e} is equivalent~to
$$\sum_{\begin{subarray}{c}d_1,d_2\ge0\\ d_1+d_2=d\end{subarray}}\sum_{b=0}^{\i}
(-1)^b\cH_{\a;d_2;d_1}^{(b)}\dot\cY_{\eset;d_2,d_2+b}=\de_{d,0}  
\qquad\forall\,d\in\Z^{\ge0}\,.$$
By \e_ref{cZstr_e} and Proposition~\ref{cZstring_prp2}, these relations hold whenever 
$\nu_n(\a)\!\ge\!0$; 
since they do not involve~$n$, they thus hold for all pairs $(n,\a)$.
\end{proof}

\subsection{Proof of Proposition~\ref{cZstring_prp}}
\label{recpf_subs2}

\noindent
We  study the secondary (middle) terms in 
the recursions~\e_ref{recurdfn_e2} for
$$\wt\cZ_{n;\a}(\x,\hb,q)\equiv\hb^{-1}\frac{\dot\cZ_{n;\a}(\x,\hb,q)}{\dot{I}_0(q)}
\qquad\hbox{and}\qquad
\dot\cZ_{n;\a;3}^{(\0,\1)}(\x,\hb,q).$$
We show that \e_ref{equivMS_e2} implies~\e_ref{equiv0thm_e} by considering the $r\!=\!-1$
coefficients in these recursions.
Conversely, if~\e_ref{equiv0thm_e} holds, we show that 
the $r\!=\!-1$ coefficients in these recursions are described in the same degree-recursive way 
in terms of the corresponding power series;
Proposition~\ref{uniqueness_prp} and Lemma~\ref{recgen_lmm} then imply that 
$\dot\cZ_{n;\a;3}^{(\0,\1)}\!=\!\wt\cZ_{n;\a}$.\footnote{The same argument, with slightly more
notation, can be used to show that all secondary coefficients are described in 
the same degree-recursive way, thus bypassing 
Proposition~\ref{uniqueness_prp} and Lemma~\ref{recgen_lmm}.}\\

\noindent
By Lemmas~\ref{Phistr_lmm4} and~\ref{recgen_lmm},
\BE{cZrec_e}\begin{split}
\dot\cZ_{n;\a}(\al_i,\hb,q)&=\sum_{d=0}^{\i}\sum_{r=0}^{N_d-1}
\{\dot\cZ_{n;\a}\}_i^r(d)\hb^{-r}q^d+
\sum_{d=1}^{\i}\sum_{j\neq i}\frac{\dot\fC_i^j(d)q^d}{\hb-\frac{\al_j-\al_i}{d}}
\dot\cZ_{n;\a}(\al_j,(\al_j\!-\!\al_i)/d,q),\\
\wt\cZ_{n;\a}(\al_i,\hb,q)&=\sum_{d=0}^{\i}\sum_{r=1}^{N_d}
\{\wt\cZ_{n;\a}\}_i^r(d)\hb^{-r}q^d+
\sum_{d=1}^{\i}\sum_{j\neq i}\frac{\dot\fC_i^j(d)q^d}{\hb-\frac{\al_j-\al_i}{d}}
\wt\cZ_{n;\a}(\al_j,(\al_j\!-\!\al_i)/d,q),
\end{split}\EE
for some $N_d\!\in\!\Z^+$ and 
$\{\dot\cZ_{n;\a}\}_i^r(d),\{\wt\cZ_{n;\a}\}_i^r(d)\!\in\!\Q_{\al}$.
It is immediate that
\begin{equation*}\begin{split}
&\dot{I}_0(q)\sum_{d=0}^{\i}\{\wt\cZ_{n;\a}\}_i^1(d)q^d
-\sum_{d=0}^{\i}\{\dot\cZ_{n;\a}\}_i^0(d)q^d
=-\sum_{d=1}^{\i}\sum_{j\neq i}
\frac{\dot\fC_i^j(d)q^d}{(\al_j\!-\!\al_i)/d}
\dot\cZ_{n;\a}(\al_j,(\al_j\!-\!\al_i)/d,q)\\
&\qquad=-\sum_{d=1}^{\i}\sum_{j\neq i}
\Rs{\hb=\frac{\al_j-\al_i}{d}}\big\{\hb^{-1}\dot\cZ_{n;\a}(\al_i,\hb,q)\big\}
=\Rs{\hb=0,\i}\big\{\hb^{-1}\dot\cZ_{n;\a}(\al_i,\hb,q)\big\}\\
&\qquad=\Rs{\hb=0}\big\{\hb^{-1}\dot\cZ_{n;\a}(\al_i,\hb,q)\big\}-1\,;
\end{split}\end{equation*}
the first and second equalities above follow from the first equation in~\e_ref{cZrec_e},
while the third from the Residue Theorem on~$\P^1$ and~\e_ref{cZrec_e} again,
which implies that the coefficients of $q^d$ in $\dot\cZ_{n;\a}(\al_i,\hb,q)$
are regular in $\hb$ away from $\hb\!=\!(\al_j\!-\!\al_i)/d$ with $d\!\in\!\Z^+$ and $j\!\neq\!i$
and $\hb\!=\!0,\i$.
Combining the last identity with the first statement in \e_ref{equivGivental_e},
and~\e_ref{cYexp_e}, we obtain
\BE{wtcZ1_e}
\sum_{d=0}^{\i}\{\wt\cZ_{n;\a}\}_i^1(d)q^d
=\frac{\dot\Phi_{n;\a}^{(0)}(q)}{\dot{I}_0(q)^2}
-\sum_{b=1}^{\i}\frac{\xi_{n;\a}(q)^b}{b!}\Rs{\hb=0}
\bigg\{\frac{(-1)^b}{\hb^b}\wt\cZ_{n;\a}(\al_i,\hb,q)\bigg\}.\,\EE\\

\noindent
By Lemma~\ref{recgen_lmm},
$$\dot\cZ_{n;\a;3}^{(\0,\1)}(\al_i,\hb,q)=\sum_{d=0}^{\i}\sum_{r=1}^{N_d}
\{\dot\cZ_{n;\a;3}^{(\0,\1)}\}_i^r(d)\hb^{-r}q^d+
\sum_{d=1}^{\i}\sum_{j\neq i}\frac{\dot\fC_i^j(d)q^d}{\hb-\frac{\al_j-\al_i}{d}}
\dot\cZ_{n;\a;3}^{(\0,\1)}(\al_j,(\al_j\!-\!\al_i)/d,q),$$
for some $N_d\!\in\!\Z^+$ and 
$\{\dot\cZ_{n;\a;3}^{(\0,\1)}\}_i^r(d)\!\in\!\Q_{\al}$.
By Section~\ref{RecPf_subs}, the secondary coefficients 
$\{\dot\cZ_{n;\a;3}^{(\0,\1)}\}_i^r(d)$ arise
from the contributions of decorated graphs~$\Ga$ as in~\e_ref{decortgraphdfn_e} such that 
the vertex~$v_{\min}$ to which the first marked point is attached is of valence~3 or higher.
In this case, there are four types of such graphs:
\begin{enumerate}[label=(\roman*)]
\item single-vertex graphs;
\item graphs with either marked point 2 or 3, but not both, attached to $v_{\min}$,
i.e.~$|\vt^{-1}(v_{\min})|\!=\!2$;
\item graphs with two edges leaving $v_{\min}$, i.e.~$|\tE_{v_{\min}}|\!=\!2$;
\item graphs with $|\vt^{-1}(v_{\min})|,|\tE_{v_{\min}}|\!=\!1$, but $\d(v_{\min})\!>\!0$;
\end{enumerate}
see Figure~\ref{Stermgraphs_fig}.\\

\begin{figure}
\begin{pspicture}(-.7,-1)(10,1.2)
\psset{unit=.4cm}
\pscircle*(2,0){.2}\rput(3.5,0){\smsize{$(i,d_0)$}}
\psline[linewidth=.04](2,0)(.6,1.4)\rput(.6,1.9){\smsize{$\bf 1$}}
\psline[linewidth=.04](2,0)(-.5,0)\rput(-1,0){\smsize{$\bf 2$}}
\psline[linewidth=.04](2,0)(.6,-1.4)\rput(.6,-1.9){\smsize{$\bf 3$}}
\psline[linewidth=.04](8.1,0)(10.8,0)\rput(9.4,.7){\smsize{$d$}}
\psline[linewidth=.04](10.8,0)(13.5,0)\pscircle*(13.5,0){.2}
\psline[linewidth=.04](13.5,0)(14.9,1.4)\rput(14.9,1.9){\smsize{$\bf 3$}}
\psline[linewidth=.04](8.1,0)(6.7,1.4)\rput(6.7,1.9){\smsize{$\bf 1$}}
\psline[linewidth=.04](8.1,0)(6.7,-1.4)\rput(6.7,-1.9){\smsize{$\bf 2$}}
\pscircle*(8.1,0){.2}\rput(6.5,0){\smsize{$(i,d_0)$}}
\pscircle*(10.8,0){.2}\rput(10.8,-.85){\smsize{$(j,*)$}}
\pscircle*(18.1,0){.2}\rput(17.4,-.85){\smsize{$(i,d_0)$}}
\psline[linewidth=.04](18.1,0)(19.8,1.7)\pscircle*(19.8,1.7){.2}
\psline[linewidth=.04](18.1,0)(19.8,-1.7)\pscircle*(19.8,-1.7){.2}
\psline[linewidth=.04](19.8,1.7)(22.5,1.7)\pscircle*(22.5,1.7){.2}
\psline[linewidth=.04](19.8,-1.7)(22.5,-1.7)\pscircle*(22.5,-1.7){.2}
\psline[linewidth=.04](22.5,1.7)(24.5,1.7)\rput(25,1.7){\smsize{$\bf 2$}}
\psline[linewidth=.04](22.5,-1.7)(24.5,-1.7)\rput(25,-1.7){\smsize{$\bf 3$}}
\rput(19.8,2.5){\smsize{$(j_2,*)$}}\rput(19.8,-2.5){\smsize{$(j_3,*)$}}
\psline[linewidth=.04](18.1,0)(16.7,1.4)\rput(16.7,1.9){\smsize{$\bf 1$}}
\rput(19.5,.8){\smsize{$d_2$}}\rput(19.7,-.8){\smsize{$d_3$}}
\psline[linewidth=.04](28.1,0)(30.8,0)\rput(29.4,.7){\smsize{$d$}}
\psline[linewidth=.04](30.8,0)(33.5,0)\pscircle*(33.5,0){.2}
\psline[linewidth=.04](33.5,0)(35.2,1.7)\pscircle*(35.2,1.7){.2}
\psline[linewidth=.04](33.5,0)(35.2,-1.7)\pscircle*(35.2,-1.7){.2}
\pscircle*(28.1,0){.2}\rput(28.1,-.85){\smsize{$(i,d_0)$}}
\pscircle*(30.8,0){.2}\rput(30.8,-.85){\smsize{$(j,*)$}}
\psline[linewidth=.04](28.1,0)(26.7,1.4)\rput(26.7,1.9){\smsize{$\bf 1$}}
\psline[linewidth=.04](35.2,1.7)(37.2,1.7)\rput(37.7,1.7){\smsize{$\bf 2$}}
\psline[linewidth=.04](35.2,-1.7)(37.2,-1.7)\rput(37.7,-1.7){\smsize{$\bf 3$}}
\rput(28.1,-1.9){\smsize{$d_0\!\in\!\Z^+$}}
\end{pspicture}
\caption{The 4 types of graphs determining the secondary coefficients 
$\{\dot\cZ_{n;\a;3}^{(\0,\1)}\}_i^r(d)$}
\label{Stermgraphs_fig}
\end{figure}
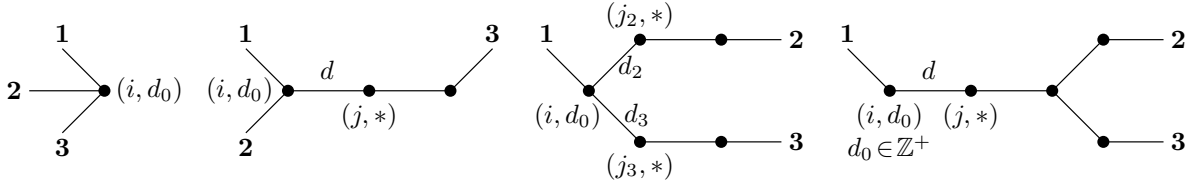

\noindent
By \e_ref{NZform_e}, \e_ref{cVform_e}, and~\e_ref{phirestr_e},
the contribution of the graphs of type~(i) to 
$\sum\limits_{d=0}^{\i}\{\dot\cZ_{n;\a;3}^{(\0,\1)}\}_i^1(d)q^d$ is
\BE{cZcont_e1}
\sum_{d=0}^{\i}\frac{q^d}{d!}
\int_{\ov\cM_{0,3|d}}\frac{\E(\dot\V_{\a}^{(d)}(\al_i))}
{\prod\limits_{k\neq i}\!\!\E(\dot\V_1^{(d)}(\al_i\!-\!\al_k))}
=\F_{n;\a}^{(0,0,0)}(\al_i,q).\EE
In the three remaining cases, we split each decorated graph $\Ga$ into subgraphs
as on page~\pageref{Ga0split_it}; see Figure~\ref{splitgraph_fig2b}.
Let $\pi_0,\pi_{c;e}$ denote the projection maps in the decomposition~\e_ref{Zlocus_e2}.
By~\e_ref{NZform_e} and~\e_ref{cVform_e},
\BE{NZcVsplit_e}\begin{split}
\frac{\E(\N Q_{\Ga})}{\E(T_{P_i}\Pn)}&=
\prod_{k\neq i}\pi_0^*\E\big(\dot\V_1^{(|\Ga_0|)}(\al_i\!-\!\al_k)\big)
\cdot\prod_{e\in\tE_{v_{\min}}}\!\!\!\!\bigg(
\pi_{c;e}^*\frac{\E(\N Q_{\Ga_{c;e}})}{\E(T_{P_i}\P^{n-1})}
\big(\om_{e;v_{\min}}\!-\!\pi_0^*\psi_e\big)\!\!\bigg),\\
\E\big(\dot\V_{n;\a}^{(|\Ga|)}\big)\big|_{Q_{\Ga}}&=
  \pi_0^*\E\big(\dot\V_{\a}^{(|\Ga_0|)}(\al_i)\big)\cdot
\prod_{e\in\tE_{v_{\min}}}\!\!\!\!\!\pi_{c;e}^*\E(\dot\V_{n;\a}^{(|\Ga_{c;e}|)}\big)\,.
\end{split}\EE
Thus, the contribution of $\Ga$ to
$\sum\limits_{d=0}^{\i}\{\dot\cZ_{n;\a;3}^{(\0,\1)}\}_i^1(d)q^d$ is
\BE{cZcontr_e5}\begin{split}
q^{|\Ga|}\!\int_{Q_{\Ga}}
\frac{\E(\dot\V_{n;\a}^{(|\Ga|)})\ev_1^*\phi_i|_{Q_{\Ga}}}{\E(\N Q_{\Ga})}
=\sum_{\b\in(\Z^{\ge0})^{\tE_{v_{\min}}}}\!\!\!\Bigg(
\frac{q^{d_0}}{d_0!}\int_{\ov\cM_{0,m_0|d_0}}\frac{\E(\dot\V_{\a}^{(d_0)}(\al_i))
\prod\limits_{e\in\tE_{v_{\min}}}\!\!\!\!\!\psi_e^{b_e}}
{\prod\limits_{k\neq i}\E(\dot\V_1^{(d_0)}(\al_i\!-\!\al_k))}\qquad&\\
\times
\prod_{e\in\tE_{v_{\min}}}\!\!\!\!\!q^{|\Ga_{c;e}|}\om_{e;v_{\min}}^{-(b_e+1)}\!\!
\int_{Q_{\Ga_{c;e}}}
\frac{\E(\dot\V_{n;\a}^{(|\Ga_{c;e}|)})\ev_1^*\phi_i}{\E(\N Q_{\Ga_{c;e}})}&\Bigg),
\end{split}\EE
where $m_0\!=\!|\vt^{-1}(v_{\min})|\!+\!|\tE_{v_{\min}}|$ 
($=3$ if $\Ga$ is of type~(ii) or~(iii), $=2$ if $\Ga$ is of type~(iv) above)
and $d_0\!=\!\d(v_{\min})$.\\

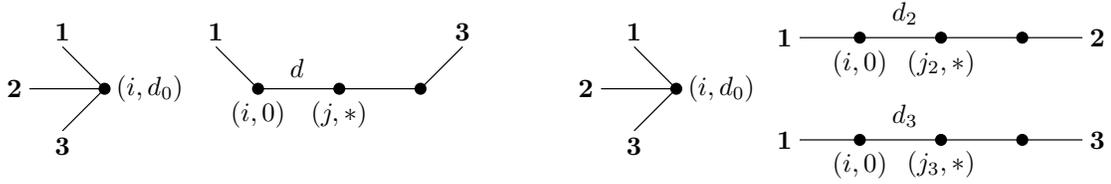
\begin{figure}
\begin{pspicture}(-1.7,-1)(10,1.2)
\psset{unit=.4cm}
\pscircle*(2,0){.2}\rput(3.5,0){\smsize{$(i,d_0)$}}
\psline[linewidth=.04](2,0)(.6,1.4)\rput(.6,1.9){\smsize{$\bf 1$}}
\psline[linewidth=.04](2,0)(-.5,0)\rput(-1,0){\smsize{$\bf 2$}}
\psline[linewidth=.04](2,0)(.6,-1.4)\rput(.6,-1.9){\smsize{$\bf 3$}}
\psline[linewidth=.04](7.1,0)(9.8,0)\rput(8.4,.7){\smsize{$d$}}
\psline[linewidth=.04](9.8,0)(12.5,0)\pscircle*(12.5,0){.2}
\psline[linewidth=.04](12.5,0)(13.9,1.4)\rput(13.9,1.9){\smsize{$\bf 3$}}
\psline[linewidth=.04](7.1,0)(5.7,1.4)\rput(5.7,1.9){\smsize{$\bf 1$}}
\pscircle*(7.1,0){.2}\rput(7.1,-.85){\smsize{$(i,0)$}}
\pscircle*(9.8,0){.2}\rput(9.8,-.85){\smsize{$(j,*)$}}
\pscircle*(21,0){.2}\rput(22.5,0){\smsize{$(i,d_0)$}}
\psline[linewidth=.04](21,0)(19.6,1.4)\rput(19.6,1.9){\smsize{$\bf 1$}}
\psline[linewidth=.04](21,0)(18.5,0)\rput(18,0){\smsize{$\bf 2$}}
\psline[linewidth=.04](21,0)(19.6,-1.4)\rput(19.6,-1.9){\smsize{$\bf 3$}}
\psline[linewidth=.04](25.1,1.7)(27.1,1.7)\rput(24.6,1.7){\smsize{$\bf 1$}}
\psline[linewidth=.04](25.1,-1.7)(27.1,-1.7)\rput(24.6,-1.7){\smsize{$\bf 1$}}
\pscircle*(27.1,1.7){.2}\rput(27.1,.85){\smsize{$(i,0)$}}
\pscircle*(27.1,-1.7){.2}\rput(27.1,-2.55){\smsize{$(i,0)$}}
\psline[linewidth=.04](27.1,1.7)(29.8,1.7)\pscircle*(29.8,1.7){.2}
\psline[linewidth=.04](27.1,-1.7)(29.8,-1.7)\pscircle*(29.8,-1.7){.2}
\psline[linewidth=.04](29.8,1.7)(32.5,1.7)\pscircle*(32.5,1.7){.2}
\psline[linewidth=.04](29.8,-1.7)(32.5,-1.7)\pscircle*(32.5,-1.7){.2}
\psline[linewidth=.04](32.5,1.7)(34.5,1.7)\rput(35,1.7){\smsize{$\bf 2$}}
\psline[linewidth=.04](32.5,-1.7)(34.5,-1.7)\rput(35,-1.7){\smsize{$\bf 3$}}
\rput(29.8,.85){\smsize{$(j_2,*)$}}\rput(29.8,-2.5){\smsize{$(j_3,*)$}}
\rput(28.6,2.5){\smsize{$d_2$}}\rput(28.6,-.9){\smsize{$d_3$}}
\end{pspicture}
\caption{The subgraphs of the 2 middle graphs in Figure~\ref{Stermgraphs_fig}}
\label{splitgraph_fig2b}
\end{figure}

\noindent
We now sum up \e_ref{cZcontr_e5} over all possibilities for $\Ga$ of each of
the three types.
For each $e\!\in\!\tE_{v_{\min}}$, 
let $v_e\!\in\!\Ver$ denote the vertex of~$e$ other than~$v_{\min}$.
By \e_ref{psiform_e} and Section~\ref{RecPf_subs}, the sum of the factor 
corresponding to $e\!\in\!\tE_{v_{\min}}$ over all possibilities for~$\Ga_e$ with $\d(e)\!=\!d_e$
and $\mu(v_e)\!=\!j_e$ fixed~is
$$(-1)^{b_e+1}\Rs{\hb=\frac{\al_{j_e}-\al_i}{d_e}}\big\{\hb^{-(b_e+1)}\dot\cZ(\al_i,\hb,q)\big\},$$
where $\dot\cZ\!=\!\dot\cZ_{n;\a}$ in cases (ii) and (iii) and 
$\dot\cZ\!=\!\dot\cZ_{n;\a;3}^{(\1,\0)}$
in case~(iv).
Thus, by the Residue Theorem on~$\P^1$ and Lemma~\ref{recgen_lmm},
the sum of the factors corresponding to $e\!\in\!\tE_{v_{\min}}$ over all possibilities
for~$\Ga_e$~is
\BE{Ressum_e1} 
(-1)^{b_e}\Rs{\hb=0,\i}\bigg\{\frac{\dot\cZ(\al_i,\hb,q)}{\hb^{b_e+1}}\bigg\}
=(-1)^{b_e}\Rs{\hb=0}\bigg\{\frac{\dot\cZ(\al_i,\hb,q)}{\hb^{b_e+1}}\bigg\}
-\begin{cases}
\de_{b_e,0},&\hbox{in cases (ii),(iii)}; \\
0,&\hbox{in case (iv)}.\end{cases}\EE
Combining \e_ref{cZcontr_e5} and \e_ref{Ressum_e1} with~\e_ref{Fred_e}, 
the first equation in~\e_ref{equivGivental_e}, and~\e_ref{cYexp_e}, 
we find that the contribution to $\sum\limits_{d=0}^{\i}\{\dot\cZ_{n;\a;3}^{(\0,\1)}\}_i^1(d)q^d$
from all graphs~$\Ga$ of types~(ii) and (iii) above is given~by
\begin{equation*}\begin{split}
&\F_{n;\a}^{(0,0,0)}(\al_i,q)\!\!\!\!\!\!\!\!\!
\sum_{\b\in(\Z^{\ge0})^{\tE_{v_{\min}}}}
\!\!\!\!\!\!\!\frac{(-\xi_{n;\a}(\al_i,q))^{|\b|}}{\b!}\!\!\!
\prod\limits_{e\in\tE_{v_{\min}}}\!\!\!\!\!
\bigg(\Rs{\hb=0}\bigg\{\frac{1}{\hb^{b_e+1}}\dot\cZ_{n;\a}(\al_i,\hb,q)\bigg\}-\de_{b_e,0}\bigg)\\
&\qquad=\F_{n;\a}^{(0,0,0)}(\al_i,q)\bigg(\Rs{\hb=0}
\bigg\{\frac{1}{\hb}\ne^{-\frac{\xi_{n;\a}(\al_i,q)}{\hb}}
\dot\cZ_{n;\a}(\al_i,\hb,q)\bigg\}-1\bigg)^{|\tE_{v_{\min}}|}\\
&\qquad
=\F_{n;\a}^{(0,0,0)}(\al_i,q)\bigg(\frac{\dot\Phi_{n;\a}^{(0)}(\al_i,q)}{\dot{I}_0(q)}
-1\bigg)^{|\tE_{v_{\min}}|}\,,
\end{split}\end{equation*}
with $|\tE_{v_{\min}}|\!=\!1$ in (ii) and $=\!2$ in (iii).
Using \cite[Theorem~4]{GWvsSQ} instead of~\e_ref{Fred_e}, we find that 
 the contribution to $\sum\limits_{d=0}^{\i}\{\dot\cZ_{n;\a;3}^{(\0,\1)}\}_i^1(d)q^d$
from all graphs~$\Ga$ of type~(iv) above is given~by
\begin{equation*}\begin{split}
&-\sum_{b=0}^{\i} \frac{(-\xi_{n;\a}(\al_i,q))^{b+1}}{(b\!+\!1)!}
\Rs{\hb=0}\bigg\{\frac{1}{\hb^{b+1}}\dot\cZ_{n;\a;3}^{(\0,\1)}(\al_i,\hb,q)\bigg\}
=-\sum_{b=1}^{\i} \frac{\xi_{n;\a}(\al_i,q)^b}{b!}
\Rs{\hb=1}\bigg\{\frac{(-1)^b}{\hb^b}\dot\cZ_{n;\a;3}^{(\0,\1)}(\al_i,\hb,q)\bigg\}.
\end{split}\end{equation*}
Putting this all together and taking into account that there are two flavors of
type~(ii) graphs, we conclude that 
\BE{wtcZna_e}\begin{split}
\sum_{d=0}^{\i}\{\dot\cZ_{n;\a;3}^{(\0,\1)}\}_i^1(d)q^d
&=\F_{n;\a}^{(0,0,0)}(\al_i,q)\frac{\dot\Phi_{n;\a}^{(0)}(\al_i,q)^2}{\dot{I}_0(q)^2}\\
&\hspace{1in}-\sum_{b=1}^{\i}\frac{\xi_{n;\a}(\al_i,q)^b}{b!}\Rs{\hb=0}
\bigg\{\frac{(-1)^b}{\hb^b}\dot\cZ_{n;\a;3}^{(\0,\1)}(\al_i,\hb,q)\bigg\}\,.
\end{split}\EE
This is the same degree-recursive relation as~\e_ref{wtcZ1_e}
if and only~if \e_ref{equiv0thm_e} holds.

\subsection{Proof of Proposition~\ref{cZstring_prp2}}
\label{recpf_subs3}

\noindent
We next apply the same argument to the power series
$$\wt\cZ_{n;\a}(\x,\hb,q)\equiv\hb^{-1}\dot\cZ_{n;\a}(\x,\hb,q)
\qquad\hbox{and}\qquad
\dot\cZ_{n;\a;2}^{(\0,\1)}(\x,\hb,q).$$
In this case, \e_ref{wtcZ1_e} becomes
\BE{wtcZb1_e}
\sum_{d=0}^{\i}\{\wt\cZ_{n;\a}\}_i^1(d)q^d
=\frac{\dot\Phi_{n;\a}^{(0)}(q)}{\dot{I}_0(q)}
-\sum_{b=1}^{\i}\frac{\xi_{n;\a}(q)^b}{b!}\Rs{\hb=0}
\bigg\{\frac{(-1)^b}{\hb^b}\wt\cZ_{n;\a}(\al_i,\hb,q)\bigg\}.\,\EE\\

\noindent
The graphs contributing to $\{\dot\cZ_{n;\a}\}_i^r(d)$ are the same as before,
as are the decomposition~\e_ref{Zlocus_e2} and the first splitting in~\e_ref{NZcVsplit_e}. 
However, the second splitting in~\e_ref{NZcVsplit_e} changes.
For graphs~$\Ga$ of type~(i) and~(ii) with $\vt(3)\!=\!v_{\min}$, it becomes
$$\E\big(\dot\V_{n;\a;2}^{(|\Ga|)}\big)\big|_{Q_{\Ga}}=
  \pi_0^*\E\big(\dot\V_{\a;0}^{(|\Ga_0|)}(\al_i)\big)\cdot
\pi_{c;e}^*\E(\dot\V_{n;\a}^{(|\Ga_{c;e}|)}\big)$$
with the second factor being 1 for the graphs of type~(i) and $e\!\in\!\tE_{v_{\min}}$ 
denoting the unique element for the graphs of type~(ii).
For graphs~of type~(ii) with $\vt(2)\!=\!v_{\min}$, graphs of type~(iii), and
graphs of type~(iv), it becomes
\begin{equation*}\begin{split}
\E\big(\dot\V_{n;\a;2}^{(|\Ga|)}\big)\big|_{Q_{\Ga}}&=
  \pi_0^*\E\big(\dot\V_{\a;|\Ga_{c;e}|}^{(|\Ga_0|)}(\al_i)\big),\\
\E\big(\dot\V_{n;\a;2}^{(|\Ga|)}\big)\big|_{Q_{\Ga}}&=
  \pi_0^*\E\big(\dot\V_{\a;|\Ga_{c;e_3}|}^{(|\Ga_0|)}(\al_i)\big)\cdot
  \pi_{c;e_2}^*\E(\dot\V_{n;\a}^{(|\Ga_{c;e_2}|)}\big),\\
\E\big(\dot\V_{n;\a;2}^{(|\Ga|)}\big)\big|_{Q_{\Ga}}&=
  \pi_0^*\E\big(\dot\V_{\a}^{(|\Ga_0|)}(\al_i)\big)\cdot
  \pi_{c;e}^*\E(\dot\V_{n;\a;2}^{(|\Ga_{c;e}|)}\big),
\end{split}\end{equation*}
respectively.\\

\noindent
Thus, similarly to~\e_ref{cZcont_e1}, 
the contribution of the graphs of type~(i) to 
$\sum\limits_{d=0}^{\i}\{\dot\cZ_{n;\a;2}^{(\0,\1)}\}_i^1(d)q^d$~is
$$\sum_{d=0}^{\i}\frac{q^d}{d!}
\int_{\ov\cM_{0,3|d}}\frac{\E(\dot\V_{\a;0}^{(d)}(\al_i))}
{\prod\limits_{k\neq i}\!\!\E(\dot\V_1^{(d)}(\al_i\!-\!\al_k))}
=\sum_{b=0}^{\i}\F_{n;\a;0}^{(0,0,b)}(\al_i,q)
\Rs{\hb=0}\bigg\{\frac{(-1)^b}{\hb^{b+1}}\LR{\dot\cZ_{n;\eset}(\al_i,\hb,q)}_{q;0}q^0\bigg\}.$$
Similarly to~\e_ref{Ressum_e1}, the sum of the factor corresponding to 
an edge $e\!\in\!\tE_{v_{\min}}$ in the analogue of~\e_ref{cZcontr_e5} 
over all possibilities for~$\Ga_e$~is
$$(-1)^{b_e}\begin{cases}
\Rs{\hb=0}\bigg\{\frac{\dot\cZ_{n;\a}(\al_i,\hb,q)}{\hb^{b_e+1}}\bigg\}
-\de_{b_e,0},&\hbox{in cases (ii) with $\vt(3)\!=\!v_{\min}$, (iii) with $e\!=\!e_2$}; \\
\Rs{\hb=0}\bigg\{\frac{\dot\cZ_{n;\eset}(\al_i,\hb,q)}{\hb^{b_e+1}}\bigg\}
-\de_{b_e,0},&\hbox{in cases (ii) with $\vt(2)\!=\!v_{\min}$, (iii) with $e\!=\!e_3$}; \\
\Rs{\hb=0}\bigg\{\frac{\dot\cZ_{n;\a;2}^{(\0,\1)}(\al_i,\hb,q)}{\hb^{b_e+1}}\bigg\}
,&\hbox{in case (iv)}.\end{cases}$$
Thus, the contribution to $\sum\limits_{d=0}^{\i}\{\dot\cZ_{n;\a;3}^{(\0,\1)}\}_i^1(d)q^d$
from all graphs~$\Ga$ of types~(ii) with $\vt(3)\!=\!v_{\min}$ and 
$\vt(2)\!=\!v_{\min}$ is
\begin{gather*}
\F_{n;\a;0}^{(0,0,0)}(\al_i,q)
\sum_{b=0}^{\i}\frac{(-\xi_{n;\a}(\al_i,q))^{b}}{b!}
\bigg(\Rs{\hb=0}\bigg\{\frac{\dot\cZ_{n;\a}(\al_i,\hb,q)}{\hb^{b+1}}\bigg\}-\de_{b,0}\bigg)
=\bigg(\frac{\dot\Phi_{n;\a}^{(0)}(\al_i,q)}{\dot{I}_0(q)}-1\bigg)
\F_{n;\a;0}^{(0,0,0)}(\al_i,q)\\
\hbox{and}\qquad
\sum_{b=0}^{\i}\sum_{r=1}^{\i} \F_{n;\a;r}^{(0,0,b)}(\al_i,q)
\Rs{\hb=0}\Bigg\{\frac{(-1)^b}{\hb^{b+1}}\LR{\dot\cZ_{n;\eset}(\al_i,\hb,q)}_{q;r}q^r\Bigg\},
\end{gather*}
respectively.
Similarly, the contribution from all graphs~$\Ga$ of type~(iii) is
\begin{equation*}\begin{split}
&\sum_{b_2,b_3\ge0}^{\i}\sum_{r=1}^{\i}\F_{n;\a;r}^{(0,0,b_3)}(\al_i,q)
\Bigg(\frac{(-\xi_{n;\a}(\al_i,q))^{b_2}}{b_2!} 
\bigg(\Rs{\hb=0}\bigg\{\frac{\dot\cZ_{n;\a}(\al_i,\hb,q)}{\hb^{b_2+1}}\bigg\}-\de_{b_2,0}\bigg)\\
&\hspace{2.5in}\times\Rs{\hb=0}\Bigg\{\frac{(-1)^{b_3}}{\hb^{b_3+1}}
\LR{\dot\cZ_{n;\eset}(\al_i,\hb,q)}_{q;r}q^r\Bigg\}\Bigg)\\
&\qquad\qquad=
\bigg(\frac{\dot\Phi_{n;\a}^{(0)}(\al_i,q)}{\dot{I}_0(q)}-1\bigg)
\sum_{b=0}^{\i}\sum_{r=1}^{\i}\F_{n;\a;r}^{(0,0,b)}(\al_i,q)
\Rs{\hb=0}\Bigg\{\frac{(-1)^b}{\hb^{b+1}}
\LR{\dot\cZ_{n;\eset}(\al_i,\hb,q)}_{q;r}q^r\Bigg\}.
\end{split}\end{equation*}
Finally, the contribution from all graphs~$\Ga$ of type~(iv) is given~by
\begin{equation*}\begin{split}
&-\sum_{b=1}^{\i} \frac{\xi_{n;\a}(\al_i,q)^b}{b!}
\Rs{\hb=1}\bigg\{\frac{(-1)^b}{\hb^b}\dot\cZ_{n;\a;2}^{(\0,\1)}(\al_i,\hb,q)\bigg\}.
\end{split}\end{equation*}
Putting this all together and using the first equation in~\e_ref{equivGivental_e},
but now with $\a\!=\!\eset$ and thus $\dot{I}_0\!=\!1$,
we conclude~that 
\begin{equation*}\begin{split}
\sum_{d=0}^{\i}\{\dot\cZ_{n;\a;2}^{(\0,\1)}\}_i^1(d)q^d
=\frac{\dot\Phi_{n;\a}^{(0)}(\al_i,q)}{\dot{I}_0(q)}
\sum_{b=0}^{\i}\sum_{r=0}^{\i} \F_{n;\a;r}^{(0,0,b)}(\al_i,q)
\Rs{\hb=0}\Bigg\{\frac{(-1)^b}{\hb^{b+1}}
\LR{\dot\cY_{n;\eset}(\al_i,\hb,q)}_{q;r}q^r\Bigg\}&\\
-\sum_{b=1}^{\i}\frac{\xi_{n;\a}(\al_i,q)^b}{b!}\Rs{\hb=0}
\bigg\{ \frac{(-1)^b}{\hb^b}\dot\cZ_{n;\a;2}^{(\0,\1)}(\al_i,\hb,q)\bigg\}&.\,
\end{split}\end{equation*}
This is the same degree-recursive relation as~\e_ref{wtcZb1_e}
if and only~if \e_ref{equiv0thm2_e} holds.

\section{Proof of \e_ref{Z2cJpt_e}}
\label{3ptpf_sec}

\noindent
The equivariant cohomology of $\Pn\!\times\!\Pn\!\times\!\Pn$ is given~by
$$H_{\T}^*(\Pn\!\times\!\Pn\!\times\!\Pn)=
\Q[\al_1,\ldots,\al_n,\x_1,\x_2,\x_3]\bigg/
\Bigg\{\prod_{k=1}^n(\x_1\!-\!\al_k),\prod_{k=1}^n(\x_2\!-\!\al_k),
\prod_{k=1}^n(\x_3\!-\!\al_k)\Bigg\}\,.$$
Thus, by the defining property of the cohomology pushforward \cite[(3.11)]{bcov0},
the three-point power series $\dot\cZ_{n;\a}$ in~\e_ref{cZdfn_e} is  
completely determined by the $n^3$ power series
\BE{cZval_e}
\dot\cZ_{n;\a}(\al_{i_1},\al_{i_2},\al_{i_3},\hb_1,\hb_2,\hb_3,q)
=\sum_{d=0}^{\i}q^d\!\!
\int_{\ov{Q}_{0,3}(\Pn,d)}\!\!\!
\frac{\E(\dot\V_{n;\a}^d)\,\ev_1^*\phi_{i_1}\,\ev_2^*\phi_{i_2}\,\ev_3^*\phi_{i_3}}
{(\hb_1\!-\!\psi_1)(\hb_2\!-\!\psi_2)(\hb_3\!-\!\psi_3)}\,.\EE
The localization formula~\e_ref{ABothm_e} reduces this expression to a sum over decorated 
trees as in Section~\ref{recpf_sec}.
Each of these trees has a unique special vertex~$v_0$: the vertex where the branches from 
the three marked points come together (one or more of the marked points may be
attached to this vertex).
We compute this sum by breaking each such tree~$\Ga$ at~$v_0$
into up to 4 ``sub-graphs":
\begin{enumerate}[label=(\roman*)]
\item\label{Ga0split3_it} $\Ga_0$ consisting of the vertex~$v_0$ only,
with 3 marked points and with the same $\mu$ and $\d$-values as in~$\Ga$;
\item for each marked point $t\!=\!1,2,3$ of $\Ga$ with $\vt(t)\!\neq\!v_0$,
$\Ga_t$ consisting of the branch of $\Ga$ running between the vertices 
$\vt(t)$ and~$v_0$, 
with the $\d$-value of $v_0$ replaced by~0 and with one new marked point attached to~$v_0$;
\end{enumerate}
see Figure~\ref{strands_fig}.
The contribution of the vertex graphs~(i) is accounted for by the Hurwitz numbers
of Proposition~\ref{equiv0_prp}, while 
the contribution of each of the strands is accounted for 
by the SQ-analogue of the double Givental's $J$-function computed by~\e_ref{main_e2},
\e_ref{equivGivental_e}, and~\e_ref{cZs_e}. 
Putting these contributions together, we will obtain~\e_ref{Z2cJpt_e}.\\

\begin{figure}
\begin{pspicture}(-2,-1.8)(10,1.2)
\psset{unit=.4cm}
\psline[linewidth=.04](1.5,0)(5,0)\rput(3.2,.7){\smsize{$d$}}
\psline[linewidth=.04](5,0)(8.5,0)\pscircle*(8.5,0){.2}
\pscircle*(1.5,0){.2}\rput(1.5,-.85){\smsize{$(i_1,0)$}}
\pscircle*(5,0){.2}\rput(5,-.85){\smsize{$(j,2)$}}
\psline[linewidth=.04](1.5,0)(0,1.5)\rput(0,2){\smsize{$\bf 1$}}
\psline[linewidth=.04](8.5,0)(10,1.5)\rput(10,2){\smsize{$\bf 2$}}
\rput(8.5,-.85){\smsize{$(i,0)$}}\rput(6.8,.5){\smsize{$e_1$}}
\rput(5.1,.8){\smsize{$v_1$}}
\rput(5.5,-3){$\Ga_1$}
\pscircle*(15.5,0){.2}
\psline[linewidth=.04](15.5,0)(13,0)\rput(12.5,0){\smsize{$\bf 1$}}
\psline[linewidth=.04](15.5,0)(17,1.5)\rput(17,2){\smsize{$\bf 2$}}
\psline[linewidth=.04](15.5,0)(17,-1.5)\rput(17,-2){\smsize{$\bf 3$}}
\rput(17.2,0){\smsize{$(i,d_0)$}}
\rput(15.5,-3){$\Ga_0$}
\psline[linewidth=.04](24.5,2.5)(23,1)\rput(22.5,1){\smsize{$\bf 1$}}
\psline[linewidth=.04](24.5,-2.5)(23,-1)\rput(22.5,-1){\smsize{$\bf 1$}}
\pscircle*(24.5,2.5){.2}\pscircle*(24.5,-2.5){.2}
\rput(24.5,3.35){\smsize{$(i,0)$}}\rput(24.5,-3.35){\smsize{$(i,0)$}}
\psline[linewidth=.04](24.5,2.5)(28,2.5)\pscircle*(28,2.5){.2}
\psline[linewidth=.04](24.5,-2.5)(28,-2.5)\pscircle*(28,-2.5){.2}
\psline[linewidth=.04](28,2.5)(30.5,2.5)\rput(31,2.5){\smsize{$\bf 2$}}
\psline[linewidth=.04](28,-2.5)(30.5,-2.5)\rput(31,-2.5){\smsize{$\bf 3$}}
\rput(28,-4){$\Ga_3$}\rput(28,4){$\Ga_2$}
\rput(26.3,2){\smsize{$e_2$}}\rput(26.3,-2){\smsize{$e_3$}}
\rput(28,1.7){\smsize{$v_2$}}\rput(28,-1.7){\smsize{$v_3$}}
\end{pspicture}
\caption{The 4 sub-graphs of the second graph in Figure~\ref{loopgraph_fig}, with label $i$
replaced by~$i_1$.}
\label{strands_fig}
\end{figure}
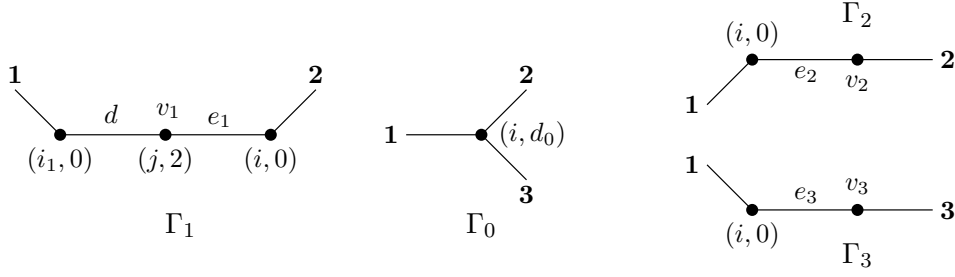

\noindent
Let $i\!=\!\mu(v_0)$ and $d_0\!=\!\d(v_0)$.
For each $t\!=\!1,2,3$ with $\vt(t)\!\neq\!v_0$, let
$e_t\!=\!\{v_0,v_t\}$ be the edge leaving $v_0$ in the direction of~$\vt(t)$.
By~\e_ref{Zlocus_e},
\BE{Z3locus_e2}
Q_{\Ga}\approx Q_{\Ga_0}\times \prod_{t=1}^3\!Q_{\Ga_t}=
(\ov\cM_{0,3|d_0}/\bS_{d_0})\times 
\prod_{t=1}^3\!Q_{\Ga_t}\,,\EE
where the $t$-th factor is defined to be a point if $\vt(t)\!=\!v_0$.
Let $\pi_0,\ldots,\pi_3$ be the component projection maps in~\e_ref{Z3locus_e2}.
By~\e_ref{NZform_e} and~\e_ref{cVform_e},
\BE{NZcVsplit3_e}\begin{split}
\frac{\E(\N Q_{\Ga})}{\E(T_{P_i}\Pn)}&=
\prod_{k\neq i}\pi_0^*\E\big(\dot\V_1^{(d_0)}(\al_i\!-\!\al_k)\big)
\cdot\prod_{t=1}^3\!\bigg(
\pi_t^*\frac{\E(\N Q_{\Ga_t})}{\E(T_{P_i}\P^{n-1})}
\big(\om_{e_t;v_0}\!-\!\pi_0^*\psi_t\big)\!\!\bigg),\\
\E\big(\dot\V_{n;\a}^{(|\Ga|)}\big)\big|_{Q_{\Ga}}&=
  \pi_0^*\E\big(\dot\V_{\a}^{(d_0)}(\al_i)\big)\cdot
\prod_{t=1}^3\!\pi_t^*\E(\dot\V_{n;\a}^{(|\Ga_t|)}\big)\,,
\end{split}\EE
with the $t$-factor defined to be~1 if $\vt(t)\!=\!v_0$.
Thus, the contribution of $\Ga$ to~\e_ref{cZval_e} is
\BE{cZ3contr_e5}\begin{split}
&\frac{1}{\prod\limits_{k\neq i}\!(\al_i\!-\!\al_k)}
\sum_{b_1,b_2,b_3\ge0}\!\!\Bigg(
\frac{q^{d_0}}{d_0!}\int_{\ov\cM_{0,3|d_0}}\frac{\E(\dot\V_{\a}^{(d_0)}(\al_i))
\prod\limits_{t=1}^3\!\psi_t^{b_t}}
{\prod\limits_{k\neq i}\E(\dot\V_1^{(d_0)}(\al_i\!-\!\al_k))}\\
&\hspace{.1in}\times
q^{|\Ga_1|}\om_{e_1;v_0}^{-(b_1+1)}\!\!\int_{Q_{\Ga_1}}\!\!\!\!\!
\frac{\E(\dot\V_{n;\a}^{(|\Ga_1|)})\,\ev_1^*{\phi_{i_1}}\ev_2^*\phi_i}
{\E(\N Q_{\Ga_1})(\hb_1\!-\!\psi_1)}~
\prod_{t=2}^3 q^{|\Ga_t|}\om_{e_t;v_0}^{-(b_t+1)}\!\!
\int_{Q_{\Ga_t}}
\frac{\E(\dot\V_{n;\a}^{(|\Ga_t|)})\ev_1^*\phi_i\,\ev_t^*\phi_{i_t}}
{\E(\N Q_{\Ga_t})(\hb_t\!-\!\psi_t)}\Bigg),
\end{split}\EE
where the $t$-th factor on the second line is defined to be $\hb_t^{-(b_t+1)}$
if $\vt(t)\!=\!v_0$.\\

\noindent
We next sum up \e_ref{cZ3contr_e5} over all possibilities for $\Ga$.
Let
$$\dot\cZ_i(\hb,\al_{i_t},\hb_t,q)
=\begin{cases}
\dot\cZ_{n;\a}(\al_{i_1},\al_i,\hb_1,\hb,q),&\hbox{if}~t\!=\!1;\\
\dot\cZ_{n;\a}(\al_i,\al_{i_t},\hb,\hb_t,q),&\hbox{if}~t\!=\!2,3.
\end{cases}$$
By \e_ref{psiform_e} and Section~\ref{RecPf_subs}, the sum of the factor
in~\e_ref{cZ3contr_e5}  
corresponding to each $t\!=\!1,2,3$ over all possibilities for~$\Ga_t$ with $\d(e_t)\!=\!d_t$
and $\mu(v_t)\!=\!j_t$ fixed~is
$$(-1)^{b_t+1}\Rs{\hb=\frac{\al_{j_t}-\al_i}{d_t}}
\big\{\hb^{-(b_t+1)}\dot\cZ_i(\hb,\al_{i_t},\hb_t,q)\big\}\,.$$
Thus, by the Residue Theorem on~$\P^1$ and Lemma~\ref{recgen_lmm},
the sum of the factor in~\e_ref{cZ3contr_e5} 
corresponding to each $t\!=\!1,2,3$ over all possibilities
for~$\Ga_t$ non-trivial~is
$$(-1)^{b_t}\Rs{\hb=0,\i,-\hb_t}
\bigg\{\frac{\dot\cZ_i(\hb,\al_{i_t},\hb_t,q)}{\hb^{b_t+1}}\bigg\}
=(-1)^{b_t}\Rs{\hb=0}\bigg\{\frac{\cZ_i(\hb,\al_{i_t},\hb_t,q)}{\hb^{b_t+1}}\bigg\}
-\hb_t^{-(b_t+1)}\prod_{k\neq i}(\al_{i_t}\!-\!\al_k).$$
Since the last term above is the contribution from the trivial sub-graph~$\Ga_t$,
the sum of the factor in~\e_ref{cZ3contr_e5} 
corresponding to each $t\!=\!1,2,3$ over all possibilities for~$\Ga_t$ 
with $\mu(v_0)\!=\!i$ fixed~is
\BE{Ressum3_e1} 
\sum_{\Ga_t}\big[\textnormal{$t$-factor in \e_ref{cZ3contr_e5}}\big]
=(-1)^{b_t}\Rs{\hb=0}\bigg\{\frac{\dot\cZ_i(\hb,\al_{i_t},\hb_t,q)}{\hb^{b_t+1}}\bigg\};\EE
this takes into account the graphs $\Ga$ with $\vt(t)\!=\!i$.\\

\noindent
By \e_ref{cZ3contr_e5}, \e_ref{Ressum3_e1}, and Proposition~\ref{equiv0_prp},
\BE{cZ3sum_e}\begin{split}
&\dot\cZ_{n;\a}(\al_{i_1},\al_{i_2},\al_{i_3},\hb_1,\hb_2,\hb_3,q)\\
&\hspace{1in}
=\sum_{i=1}^n
\frac{1}{\bfs_{n-1}(\al_i)\dot\Phi_{n;\a}^{(0)}(\al_i,q)}
\prod_{t=1}^3\Rs{\hb=0}\bigg\{\frac{1}{\hb}\ne^{-\xi_{n;\a}(\al_i,q)/\hb}
\dot\cZ_i(\hb,\al_{i_t},\hb_t,q)\bigg\}.
\end{split}\EE
By \e_ref{main_e2}, \e_ref{cZs_e}, \e_ref{equivGivental_e}, and \e_ref{cYexp_e},
\begin{equation*}\begin{split}
&\Rs{\hb=0}\bigg\{\frac{1}{\hb}\ne^{-\xi_{n;\a}(\al_i,q)/\hb}
\dot\cZ_i(\hb,\al_{i_t},\hb_t,q)\bigg\}
=\sum_{\begin{subarray}{c}s_t',s_t,r_t'\ge0\\ s_t'+s_t+r_t'=n-1 \end{subarray}}
\!\!\!\!\!\!\!\!\Bigg((-1)^{r_t'}\bfs_{r_t'}\\
&\hspace{1.5in}\times
\sum_{r_t''=0}^{s_t'}\ctC_{s_t'-\ell^-(\a),s_t'-r_t''-\ell^-(\a)}^{(r_t'')}(q)
\frac{\dot\Phi_{n;\a}^{(0)}(\al_i,q)L_{n;\a}(\al_i,q)^{s_t'-r_t''}}
{\dot{I}_0(q)\ldots\dot{I}_{s_t'-r_t''}(q)}
\ddot\cZ_{n;\a}^{(s_t)}(\al_{i_t},\hb_t,q)\Bigg)
\end{split}\end{equation*}
for $t\!=\!2,3$.
Combining this with \e_ref{chcCdfn_e}, \cite[Proposition~4.4]{Po2}, and~\e_ref{bIdfn_e},
we find~that 
\BE{cZiRes_e23}\begin{split}
&\Rs{\hb=0}\bigg\{\frac{1}{\hb}\ne^{-\xi_{n;\a}(\al_i,q)/\hb}
\dot\cZ_i(\hb,\al_{i_t},\hb_t,q)\bigg\}\\
&\hspace{1.5in}=
\sum_{s_t=0}^{n-1}\sum_{r_t=0}^{\hat{s}_t}
\dot\cC_{\hat{s}_t}^{(r_t)}(q)
\frac{\dot\Phi_{n;\a}^{(0)}(\al_i,q)L_{n;\a}(\al_i,q)^{\hat{s}_t-r_t}}{\ddot\bI_{s_t+r_t}^c(q)}
\ddot\cZ_{n;\a}^{(s_t)}(\al_{i_t},\hb_t,q)
\end{split}\EE
for $t\!=\!2,3$.
By the same reasoning,
\BE{cZiRes_e1}\begin{split}
&\Rs{\hb=0}\bigg\{\frac{1}{\hb}\ne^{-\xi_{n;\a}(\al_i,q)/\hb}
\dot\cZ_i(\hb,\al_{i_1},\hb_1,q)\bigg\}\\
&\hspace{1.5in}=
\sum_{s_1=0}^{n-1}\sum_{r_1=0}^{\hat{s}_1}
\ddot\cC_{\hat{s}_1}^{(r_1)}(q)
\frac{\ddot\Phi_{n;\a}^{(0)}(\al_i,q)L_{n;\a}(\al_i,q)^{\hat{s}_1-r_1}}{\dot\bI_{s_1+r_1}^c(q)}
\dot\cZ_{n;\a}^{(s_1)}(\al_{i_1},\hb_1,q),
\end{split}\EE
where
\BE{ddotPhi_e}
\ddot\Phi_{n;\a}^{(0)}(\al_i,q)
=\bigg(\frac{L_{n;\a}(\al_i,q)}{\al_i}\bigg)^{-\ell(\a)}\dot\Phi_{n;\a}^{(0)}(\al_i,q)\,.\EE\\

\noindent
On the other hand, by~\e_ref{Phi0dfn_e} and~\e_ref{Ldfn_e},
\begin{equation*}\begin{split}
&\sum_{i=1}^n
\frac{\dot\Phi_{n;\a}^{(0)}(\al_i,q)^3L_{n;\a}(\al_i,q)^s}
{\bfs_{n-1}(\al_i)\dot\Phi_{n;\a}^{(0)}(\al_i,q)}
\bigg(\frac{L_{n;\a}(\al_i,q)}{\al_i}\bigg)^{-\ell(\a)}
=\frac{1}{\a^{\a}}\sum_{i=1}^n L_{n;\a}(\al_i,q)^{s-|\a|}\frac{\nd L}{\nd q}\\
&\hspace{2in}
=\frac{1}{\a^{\a}}\frac{\nd}{\nd q}
\begin{cases}
\ln\prod\limits_{i=1}^nL_{n;\a}(\al_i,q),&\hbox{if}~s\!=\!|\a|\!-\!1;\\
\frac{1}{s+1-|\a|}\sum\limits_{i=1}^nL_{n;\a}(\al_i,q)^{s+1-|\a|},&\hbox{otherwise}.
\end{cases}
\end{split}\end{equation*}
The collection $\{L_{n;\a}(\al_i,q)^{-1}\}$ is the set of $n$ roots~$\y$ of the equation
$$1-\bfs_1 \y+\ldots+(-1)^n\bfs_n\y^n-\a^{\a}q\y^{\nu_n(\a)}=0\,.$$
Thus, if $s\!\ge\!0$ and $s\!+\!1\!<\!|\a|$,
$$\frac{\nd}{\nd q}\sum\limits_{i=1}^nL_{n;\a}(\al_i,q)^{s+1-|\a|}
=\frac{\nd}{\nd q}\cH^{(|\a|-s-1)}\bigg(\!\!-\frac{\bfs_{n-1}}{\bfs_n},\frac{\bfs_{n-2}}{\bfs_n},
\ldots,(-1)^{|\a|-s-1}\frac{\bfs_{\nu_n(\a)+s+1}}{\bfs_n}\bigg)=0,$$
where $\cH^{(r)}$ is as in~\e_ref{Hrdfn_e}.
If $|\a|\!=\!n$, $\{L_{n;\a}(\al_i,q)\}$ is the set of $n$ roots~$\y$ of the equation
$$\y^n-(1\!-\!\a^{\a}q)^{-1}\bfs_1\y^{n-1}+(1\!-\!\a^{\a}q)^{-1}\bfs_2\y^{n-2}-\ldots
+(-1)^n(1\!-\!\a^{\a}q)^{-1}\bfs_n=0\,.$$
Thus, if $s\!+\!1\!\le\!|\a|\!=\!n$,
\BE{RootSum_e}\sum_{i=1}^n L_{n;\a}(\al_i,q)^{s-|\a|}\frac{\nd L}{\nd q}
=\a^{\a}\cH_{\nu_n(\a)}^{(s+1-n)}(\a^{\a} q),\EE
where $\cH_{\nu}^{(r)}$ is as in \e_ref{Hrnudfn_e}.
If $|\a|\!<\!n$, $\{L_{n;\a}(\al_i,q)\}$ is the set of $n$ roots~$\y$ of the equation
\begin{equation*}\begin{split}
\y^n-\bfs_1\y^{n-1}+\ldots+(-1)^{\nu_n(\a)-1}\bfs_{\nu_n(\a)-1}\y^{|\a|+1}
&+(-1)^{\nu_n(\a)}\big(\bfs_{\nu_n(\a)}-(-1)^{\nu_n(\a)}\a^{\a}q\big)\y^{|\a|}\\
&+(-1)^{\nu_n(\a)+1}\bfs_{\nu_n(\a)+1}\y^{|\a|-1}+\ldots+(-1)^n\bfs_n=0.
\end{split}\end{equation*}
Thus, if $s\!+\!1\!\le\!|\a|\!<\!n$, \e_ref{RootSum_e} still holds.
Combining the equations in this paragraph, we find~that 
\begin{equation*}\begin{split}
\sum_{i=1}^n
\frac{\dot\Phi_{n;\a}^{(0)}(\al_i,q)^3L_{n;\a}(\al_i,q)^s}
{\bfs_{n-1}(\al_i)\dot\Phi_{n;\a}^{(0)}(\al_i,q)}
\bigg(\frac{L_{n;\a}(\al_i,q)}{\al_i}\bigg)^{-\ell(\a)}
=\begin{cases}
\cH_{\nu_n(\a)}^{(s+1-n)}(\a^{\a} q),&\hbox{if}~s\!\ge\!n-\!1;\\
0,&\hbox{if}~0\!\le\!s\!<\!n\!-\!1.\end{cases}
\end{split}\end{equation*}
Combining this with \e_ref{cZ3sum_e}-\e_ref{ddotPhi_e} and~\e_ref{ctC3dfn_e}, 
we obtain~\e_ref{Z2cJpt_e}.

\vspace{.2in}

\noindent
{\it Department of Mathematics, SUNY Stony Brook, NY 11794-3651\\
azinger@math.sunysb.edu}

\end{document}